\documentclass[12pt]{amsart}
\usepackage[latin1]{inputenc}
\usepackage[british]{babel}
\usepackage[T1]{fontenc}
\usepackage{amssymb,amsmath,amsthm,amsfonts,enumerate}
\usepackage[autostyle]{csquotes}
\usepackage{alphalph,cite}
\usepackage{nameref}
\usepackage{verbatim}
\usepackage{mathrsfs}
\usepackage{manfnt}
\usepackage{pifont}
\usepackage{yfonts}
\usepackage{enumitem}
\usepackage[all]{xy}
\usepackage{wasysym}

\usepackage{stmaryrd}

\usepackage{apptools}
\usepackage{stackengine}
\usepackage{scalerel}
\usepackage{array}
\usepackage{datetime}
\usepackage{xparse}
\usepackage{tikz-cd}
\usepackage{rotating}
\usepackage{esvect}
\usepackage[colorlinks=true, linkcolor={blue}, citecolor={magenta}, urlcolor={green}]{hyperref}
\usepackage{mathtools}
\usepackage{dutchcal}
\usepackage{sansmath}
\usepackage{orcidlink}
\usepackage{chessfss}
\usepackage{caption}
\usepackage{makecell}
\usepackage{float}
\usepackage{diagbox}
\usepackage{lmodern}
\usepackage{epigraph}

\makeatletter
\@namedef{subjclassname@2010}{ \textup{2010} Mathematics Subject Classification}
\makeatother
\theoremstyle{definition}
\newtheorem{theorem}{Theorem}[subsection]

\newtheorem{theoremletters}{Theorem}

\newtheorem{claimletters}{Claim}

\theoremstyle{definition}
\newtheorem{lemma}[theorem]{Lemma}
\newtheorem{corollary}[theorem]{Corollary}
\newtheorem{proposition}[theorem]{Proposition}

\newtheorem{remark}[theorem]{Remark}
\newtheorem{example}[theorem]{Example}
\newtheorem{question}[theorem]{Question}
\newtheorem{definition}[theorem]{Definition}
\numberwithin{equation}{section}

\newtheorem{notation}[theorem]{Notation}
\newtheorem{scholium}[theorem]{Scholium}

\newtheorem{digression}[theorem]{Digression}

\DeclareMathOperator{\supp}{supp}
\DeclareMathOperator{\sign}{sign}

\DeclareMathOperator*{\esssup}{ess\,sup}

\newcommand{\norm}[1]{\left\lVert#1\right\rVert}
\newcommand{\vertiii}[1]{{\left\vert\kern-0.25ex\left\vert\kern-0.25ex\left\vert #1 
		\right\vert\kern-0.25ex\right\vert\kern-0.25ex\right\vert}}
\newcommand{\abs}[1]{\left\lvert#1\right\rvert}

\newcommand{\cembed}[1]{\overset{\mathrm{#1}}{\hookrightarrow}_{c}}

\newcommand{\ces}{\mathcal{k}}
\newcommand{\Ces}{\mathcal{K}}

\makeatletter
\def\author@andify{
	\nxandlist {\unskip ,\penalty-1 \space\ignorespaces}
	{\unskip {} \@@and~}
	{\unskip \penalty-2 \space \@@and~}
}
\makeatother

\newcommand\xqed[1]{%
	\leavevmode\unskip\penalty9999 \hbox{}\nobreak\hfill
	\quad\hbox{#1}}
\newcommand\demo{\xqed{{\Large $\blacktriangle$}}}

\makeatletter
\def\l@subsection{\@tocline{2}{0pt}{2pc}{6pc}{}}
\makeatother

\let\OLDthebibliography\thebibliography
\renewcommand\thebibliography[1]{
	\OLDthebibliography{#1}
	\setlength{\parskip}{0pt}
	\setlength{\itemsep}{5pt plus 0.3ex}
}

\begin{document}
	\title{On the isomorphism of the spaces $\ces$ and $\Ces$}
	
	\author{Tomasz Kiwerski~\orcidlink{0000-0002-8349-8785}}
	\address[Tomasz Kiwerski]{Pozna\'{n} University of Technology, Institute of Mathematics, Piotrowo 3A, 61-138 Pozna\'{n} (Poland)}
	\email{\href{mailto:tomasz.kiwerski@gmail.com}{\tt tomasz.kiwerski@gmail.com}}
	
	\author{Jakub Tomaszewski~\orcidlink{0000-0002-0715-5553}}
	\address[Jakub Tomaszewski - corresponding author]{Pozna\'{n} University of Technology, Institute of Mathematics, Piotrowo 3A, 61-138 Pozna\'{n} (Poland)}
	\email{\href{mailto:tomaszewskijakub@protonmail.com}{\tt tomaszewskijakub@protonmail.com}}
	
	\date{\longdate{\today}}
	
	\maketitle
	
	\begin{center}
		{\it Dla Juliana}
	\end{center}
	
	\begin{abstract}
		Let $1 \leqslant p \leqslant q < \infty$, let $\mathsf{X}$ be a separable $\mathcal{L}_p$-space and let $\mathsf{Y}$ be a separable $\mathcal{L}_q$-space.
		We show that the space $\mathfrak{L}(\mathsf{X},\mathsf{Y})$ of bounded linear operators from $\mathsf{X}$ into $\mathsf{Y}$
		is isomorphic to $\mathfrak{L}(\ell_p,\ell_q)$ if, and only if, $\mathsf{Y}$ is complemented in its bidual;
		in particular, this is always the case when $q > 1$.
		We also show that $\mathfrak{L}(\mathsf{Z},\ell_{\infty})$ is isomorphic to $\mathfrak{L}(c_0,\ell_{\infty})$ for every separable $\mathcal{L}_{\infty}$-space $\mathsf{Z}$,
		and that the analogous statement for $\mathfrak{L}(\mathsf{X},\mathsf{Y})$ fails in general when $p > q$.
		This extends and complements the earlier result of Arias and Farmer
		[{\it On the structure of tensor products of $\ell_p$-spaces}, Pacific J. Math. {\bf 175} (1996), no.~1, 13--37],
		a special case of which was already known to Haagerup.
		
		Based on this general result, we give a completely different proof of the main result of Astashkin, Le{\' s}nik and Maligranda
		[{\it Isomorphic structure of Ces{\` a}ro and Tandori spaces}, Canad. J. Math. {\bf 71} (2019), no.~3, 501--532],
		which states that the space $\Ces$ and its sequence counterpart $\ces$ are isomorphic.
		Furthermore, we complete the picture by showing that the spaces $\ces$ and $\Ces$,
		although isomorphic, are not isometric.
		All this is a direct analogue of Pe{\l}czy{\' n}ski's classical result that the space $L_{\infty}$
		of essentially bounded measurable functions is isomorphic, but not isometric, to its sequence counterpart $\ell_{\infty}$.
		
		Along the way, we show that for every infinite compact metrizable space $K$ the $\ell_{\infty}$-sum of copies of the space $\mathscr{M}(K)$ of regular Borel measures
		on $K$ is isomorphic to $( \bigoplus_{n=1}^{\infty} L_1[0,1])_{\ell_{\infty}} \approx \mathfrak{L}(\ell_1,L_1[0,1])$.
		This is somewhat surprising, because $\mathscr{M}(K)$, although an $\mathcal{L}_1$-space, is not separable.
		This provides an abstract blueprint for another result of Astashkin, Le{\' s}nik and Maligranda, namely that $\Ces$ is isomorphic
		to $( \bigoplus_{n=1}^{\infty} \mathscr{M}[0,1])_{\ell_{\infty}}$.
		Moreover, we show that $\ces$ has a unique isometric predual and, using a Szlenk index argument, at least continuum many pairwise non-isomorphic isomorphic preduals.
		
		Finally, we show that there are models of set theory without the Axiom of Choice in which the spaces $\ces$ and $\Ces$ are no longer isomorphic.
		This last part was inspired by V{\" a}th's paper
		[{\it The dual space of $L_{\infty}$ is $L_1$}, Indag. Math. (N.S.) {\bf 9} (1998), no.~4, 619--625].
	\end{abstract}
	
	\tableofcontents
	
	\section{{\bf Introduction}} \label{SECTION: Introduction}
	
	This work outgrowths from several interwoven threads which, in various combinations, were the main catalysts for the results obtained in the body of the paper.
	The glue that binds all these stories together is a pair of rather niche  Banach spaces which --- for historical reasons --- we denote by $\ces$ and $\Ces$.
	Their norms\footnote{Let us note right away that, up to the averages, these norms are simply the sup norms of the spaces $\ell_{\infty}$ and $L_{\infty}$.
		Hence the title of our work, which is an evident allusion to Pe{\l}czy{\' n}ski's classical paper \cite{Pel58}.},
	like those of the overwhelming majority of spaces appearing in analysis, are generated by positive sublinear operators and $L_p$-norms:
	\begin{equation*}
		\norm{x}_{\ces} = \sup_{n \in \mathbb{N}} \,\, \frac{1}{n} \sum_{k=1}^n \abs{x_k}
		\quad\text{and}\quad
		\norm{f}_{\Ces} = \sup_{x > 0} \,\, \frac{1}{x} \int_0^x \abs{f(t)}dt.
	\end{equation*}
	One could say that they are the {\bf optimal domains} of the Hardy operator
	(the general framework for this type of construction was developed in \cite{ORS08}; {\it cf.} \cite{CR16}, \cite{DS07}, \cite{KKM21}, \cite{Mas91} and \cite{NP11}).
	This is probably the most popular point of view on $\ces$ and $\Ces$, but only one of many;
	for example, Sinnamon \cite{Si94} (see also \cite{KMS07} and \cite{MS06}) prefers to see them as the so-called {\bf down spaces},
	while Bennett in his Memoir \cite{Be96} regards them as the missing pieces that turn inequalities into equalities
	(they, so to speak, {\bf factorize} Hardy's inequality).
	The reader who wants to know more about the metamorphic properties of these spaces, their applications and relationship to other (maybe more recognizable)
	constructions is referred to \cite{FLM16}, \cite{GE98}, \cite{KT24} and \cite{KMP07}.
	
	However, it seems to us that one possible interpretation of $\ces$ and $\Ces$ has been overlooked so far.
	Our point is that --- up to a Banach space isomorphism --- $\ces$ and $\Ces$ coincide with the spaces of bounded linear operators from $\ell_1$ into $\ell_1$
	and from $\ell_1$ into $L_1$, respectively; symbolically,
	\begin{equation} \label{2}
		\ces \approx \mathfrak{L}(\ell_1,\ell_1)
		\quad\text{and}\quad
		\Ces \approx \mathfrak{L}(\ell_1,L_1).
	\end{equation}
	At its heart, this observation is a simple juxtaposition of Grosse-Erdmann's {\bf blocking technique} from \cite{GE98}
	({\it cf.} \cite{KT24}, where his ideas were transmuted through the prism of interpolation theory) and the following relation:
	\begin{equation*}
		\ell_{\infty}(\mathsf{X}) \cong \mathfrak{L}(\ell_1,\mathsf{X}),
	\end{equation*}
	which most likely belongs to folklore ({\it cf.} Lemma~\ref{LEMMA : 1}).
	Nonetheless, representation \eqref{2} has far-reaching consequences for the isomorphic structure of these spaces, which we now describe.
	
	\subsection{Prologue: Astashkin and Maligranda's question} \label{SUBSECTION : AM's question}
	
	The question \enquote{{\it Are the spaces $\ces$ and $\Ces$ isomorphic?}} (in our notation)
	goes back to Astashkin and Maligranda's paper \cite[Problem~1]{AM09} ({\it cf.} \cite[Problem~4]{AM14}).
	We should mention that from the perspective of \cite{AM09} the spaces $\ces$ and $\Ces$ are not isolated
	objects but rather the end-points of a whole scale of spaces called {\bf Ces{\` a}ro spaces}.
	In fact, there are two scales, $\{ ces_p \}_{1\leqslant p \leqslant \infty}$ and $\{ Ces_p \}_{1\leqslant p \leqslant \infty}$,
	corresponding to Ces{\` a}ro sequence and function spaces, respectively.
	As the following diagram shows, the scale $\{ ces_p \}_{1\leqslant p \leqslant \infty}$
	is in a sense \enquote{parallel} to the classical scale $\{ \ell_p \}_{1 \leqslant p \leqslant \infty}$: 
	\begin{equation*}
		\begin{tikzcd}
			\ell_1 \arrow[r, hook] \arrow[d, no head, dashed] & \ell_q \arrow[r, hook] \arrow[d, no head, dashed] & \ell_2 \arrow[r, hook] \arrow[d, no head, dashed] & \ell_r \arrow[r, hook] \arrow[d, no head, dashed] & \ell_{\infty} \arrow[d, no head, dashed] \\
			ces_1 \arrow[r, hook]                             & ces_q \arrow[r, hook]                             & ces_2 \arrow[r, hook]                             & ces_r \arrow[r, hook]                             & ces_{\infty}                                    
		\end{tikzcd}
	\end{equation*}
	Here,
	\begin{enumerate}
		\item[$\bullet$] $ces_p$ is the Banach space endowed with the norm
		\begin{equation*}
			\norm{x}_{ces_p} = \left( \sum_{n=1}^{\infty} \left( \frac{1}{n} \sum_{i=1}^n \abs{x_i} \right)^p \right)^{1/p} \quad (x = \{x_n\}_{n=1}^{\infty} \in \mathbb{K}^{\mathbb{N}});
		\end{equation*}
		\item[$\bullet$] the indices $q$ and $r$ are chosen so that $1 < q < 2 < r < \infty$;
		\item[$\bullet$] the horizontal hooked arrows denote continuous inclusions;
		\item[$\bullet$] the vertical dashed lines connect the corresponding points of both scales;
		\item[$\bullet$] the space $ces_{\infty}$ in the lower right corner is nothing other than $\ces$;
		\item[$\bullet$] the space $ces_1$ is\footnote{Because, according to ancient wisdom, the harmonic series $\sum_{n=1}^{\infty} \frac{1}{n}$
			diverges.} trivial, that is, it contains only the zero sequence
		(this flaw slightly distorts the analogy, but only in the neighbourhood of the left end-point).
	\end{enumerate}
	The benevolent reader can easily imagine that a similar picture holds for the scales $\{ Ces_p \}_{1\leqslant p \leqslant \infty}$ and $\{ L_p \}_{1 \leqslant p \leqslant \infty}$.
	So, even for purely heuristic reasons dictated by the positions of the spaces in the diagrams above, AM's question is an evident allusion to Pe{\l}czy{\' n}ski's result from \cite{Pel58},
	which states that the classical spaces $\ell_{\infty}$ and $L_{\infty}$ are isomorphic.
	
	\subsection{General variant of Arias and Farmer's result}
	
	In the very first lines of their paper \cite{CS88}, Christensen and Sinclair noted that --- as they had been informed by Haagerup ---
	the following isomorphism of Banach spaces holds:
	\begin{equation*}
		\mathfrak{L}(\mathsf{H}) \approx ( \bigoplus_{n=1}^{\infty} \boldsymbol{M}_n(\mathbb{C}) )_{\ell_{\infty}},
	\end{equation*}
	where $\mathsf{H}$ stands for the infinite-dimensional separable Hilbert space and $\boldsymbol{M}_n(\mathbb{C})$
	denotes the algebra of $(n \times n)$-matrices over $\mathbb{C}$ endowed with the operator norm inherited from $\mathfrak{L}(\mathsf{H})$.
	It soon became apparent that this result is not exclusive to the realm of Hilbert spaces:
	In the 1990s Arias and Farmer \cite[Theorem~2.1]{AF96} (compare with Propositions~8 and 15 in Wark's PhD thesis \cite{War97}) showed that $\mathfrak{L}(\mathsf{X},\mathsf{Y})$,
	where $\mathsf{X}$ is a separable $\mathcal{L}_p$-space and $\mathsf{Y}$ is a separable $\mathcal{L}_q$-space with $1 < p \leqslant q < \infty$,
	is isomorphic --- as a Banach space --- to $( \bigoplus_{n=1}^{\infty} \boldsymbol{M}_n(\mathbb{C}) )_{\ell_{\infty}}$,
	where $\boldsymbol{M}_n(\mathbb{C})$ inherits the norm from $\mathfrak{L}(\ell_p,\ell_q)$.
	
	The observant reader can guess where we are heading:
	If Arias and Farmer's result could be extended to the case $p=q=1$ (that is, to the situation where both $\mathsf{X}$ and $\mathsf{Y}$ are separable $\mathcal{L}_1$-spaces),
	then, thanks to representation \eqref{2}, it would immediately give a positive answer to AM's question.
	One of the main achievements of this work is the following structural theorem, which extends the above-mentioned results to the end-point $p=1$
	and, at the same time, describes precisely what happens at $q = 1$ and beyond the range $p \leqslant q$.
	
	\begin{theoremletters}[Theorems~\ref{LEMMA : 6} and \ref{LEMMA : 6infty}, Remarks~\ref{REMARK : p>q} and \ref{REMARK : L1 not complemented in bidual}] \label{THEOREM A}
		\hfill
		\begin{enumerate}[label={\normalfont{(\arabic*)}}]
			\item {\it Let $1 \leqslant p \leqslant q < \infty$, let $\mathsf{X}$ be a separable $\mathcal{L}_p$-space and let $\mathsf{Y}$ be a separable $\mathcal{L}_q$-space.
				Then the following conditions are equivalent:}
			\begin{enumerate}
				\item $\mathfrak{L}(\mathsf{X},\mathsf{Y}) \approx \mathfrak{L}(\ell_p,\ell_q)$;
				\item $\mathfrak{L}(\mathsf{X},\mathsf{Y}) \approx ( \bigoplus_{n=1}^{\infty} \boldsymbol{M}_n(\mathbb{K}) )_{\ell_{\infty}}$,
				{\it where $\boldsymbol{M}_n(\mathbb{K})$ carries the norm of $\mathfrak{L}(\ell_p^n,\ell_q^n)$;}
				\item {\it $\mathsf{Y}$ is complemented in its bidual $\mathsf{Y}^{**}$.}
			\end{enumerate}
			{\it In particular, these conditions hold whenever $q > 1$, and also when $q = 1$ and $\mathsf{Y}$ is $\ell_1$ or $L_1(\mu)$.}
			\item {\it For every separable $\mathcal{L}_\infty$-space $\mathsf{Z}$ we have $\mathfrak{L}(\mathsf{Z}, \ell_\infty) \approx \mathfrak{L}(c_0,\ell_\infty)$.}
			\item {\it For $1 \leqslant q < p < \infty$ the conclusion of {\normalfont (1)} fails in general; for instance, the space $\mathfrak{L}(L_p[0,1],L_q[0,1])$ is not isomorphic to $\mathfrak{L}(\ell_p,\ell_q)$.}
		\end{enumerate}
	\end{theoremletters}
	
	Since $L_1[0,1]$ is complemented in its bidual, Theorem~\ref{THEOREM A} gives, in particular, $\mathfrak{L}(\ell_1,\ell_1) \approx \mathfrak{L}(\ell_1,L_1[0,1])$, and therefore:
	
	\begin{theoremletters}[Theorem~\ref{COR: Ces and ces isomorphic}] \label{THEOREM BB}
		{\it $\ces$ and $\Ces$ are isomorphic.}
	\end{theoremletters}
	
	We must emphasise that our solution to AM's question is {\bf not} the first;
	priority belongs to \cite[Theorem~5.1]{ALM19}.
	Our approach, however, is more general and unifies certain aspects of \cite{ALM19}.
	
	\subsection{Anatomy of the isomorphism} \label{SUBSECTION : Grand unification}
	
	According to \cite[Theorem~5.5]{ALM19}, the space $\Ces$ is also isomorphic to $( \bigoplus_{n=1}^{\infty} \mathscr{M}[0,1] )_{\ell_{\infty}}$,
	where $\mathscr{M}[0,1]$ denotes the space of regular Borel measures on $[0,1]$ with finite variation.
	This (almost) falls into the pattern that we used above to show that $\ces \approx \Ces$:
	The space $( \bigoplus_{n=1}^{\infty} \mathscr{M}[0,1] )_{\ell_{\infty}}$ is isometric to $\mathfrak{L}(\ell_1,\mathscr{M}[0,1])$,
	and $\mathscr{M}[0,1]$ is an $\mathcal{L}_1$-space, so we can apply Theorem~\ref{THEOREM A}\dots\
	Not so fast: $\mathscr{M}[0,1]$ is {\bf not} separable, so the use of Theorem~\ref{THEOREM A}(1) is not justified.
	This slight mismatch was a puzzle
	that plagued at least one of the authors.
	We can bypass this obstacle by passing to the predual of $\mathscr{M}[0,1]$:
	\begin{align*}
		( \bigoplus_{n=1}^{\infty} \mathscr{M}[0,1] )_{\ell_{\infty}}
		& \cong \mathfrak{L}(\ell_1,\mathscr{M}[0,1]) \quad \quad \, (\text{by Lemma~\ref{LEMMA : 1}}) \\
		& \cong \mathfrak{L}(\mathcal{C}[0,1],\ell_\infty) \quad \quad \, (\text{since $\mathfrak{L}(\mathsf{X},\mathsf{Y}^{*}) \cong \mathfrak{L}(\mathsf{Y},\mathsf{X}^{*})$}) \\
		&\approx \mathfrak{L}(c_0,\ell_\infty) \quad \quad \quad \quad \, (\text{by Theorem~\ref{THEOREM A}(2)}) \\
		& \cong \mathfrak{L}(\ell_1,\ell_1) \quad \quad \quad \quad \, \, \, (\text{again, since $\mathfrak{L}(\mathsf{X},\mathsf{Y}^{*}) \cong \mathfrak{L}(\mathsf{Y},\mathsf{X}^{*})$}) \\
		& \approx \mathfrak{L}(\ell_1,L_1[0,1]) \quad \quad \, (\text{by Theorem~\ref{THEOREM A}(1)}) \\
		& \approx \Ces \quad \quad \quad \quad \quad \, \, \, \, \, \, \quad (\text{by representation \eqref{2}}).
	\end{align*}
	To sum up, we have the following.
	
	\begin{theoremletters}[Theorem~\ref{COROLLARY : Michau}] \label{THEOREM BBB}
		{\it $\Ces$ is isomorphic to $( \bigoplus_{n=1}^{\infty} \mathscr{M}[0,1] )_{\ell_{\infty}}$.}
	\end{theoremletters}
	
	It seems instructive to compare the above argument with that of \cite[Theorem~5.5]{ALM19}:
	While our argument is again a manifestation of the mechanism behind Theorem~\ref{THEOREM A},
	the proof in \cite{ALM19} relies\footnote{The authors attribute this idea to Micha{\l} Wojciechowski; see p.~529 in {\it loc. cit.}}
	on an approximation by convolutions with Fej{\' e}r's kernel and apparently has nothing to do with \cite[Theorem~5.1]{ALM19} (Theorem~\ref{THEOREM BB}).
	We believe that the bridge connecting Theorems~\ref{THEOREM BB} and \ref{THEOREM BBB} is one of the most important findings of this work.
	
	Theorem~\ref{THEOREM BBB} in its natural generality also seems interesting in itself.
	
	\begin{theoremletters}[Theorem~\ref{THM : suma L1 isomorficzna z miarami}] \label{THEOREM BBBB}
		{\it Let $K$ be an infinite compact metrizable space, and let $\mu$ be a finite regular Borel measure on $K$ such that $L_1(\mu)$ is infinite-dimensional.
			Then the spaces
			$( \bigoplus_{n=1}^{\infty} L_1(\mu) )_{\ell_{\infty}}$,
			$( \bigoplus_{n=1}^{\infty} \mathscr{M}(K) )_{\ell_{\infty}}$ and $\ces$
			are all isomorphic.}
	\end{theoremletters}
	
	\subsection{Vademecum} \label{SUBSECTION : Vademaecum}
	
	Our aim in the rest of this paper is to explore what more can be said about the picture in which the pair $\ces$ and $\Ces$
	is seen as a non-classical analogue of the pair $\ell_{\infty}$ and $L_{\infty}$.
	
	Apart from Theorems~\ref{THEOREM BB} and \ref{THEOREM BBB}, this seems to be completely uncharted territory, and this paper is a first guide to it.
	We summarise our findings about the isomorphic and isometric structure of $\ces$ and $\Ces$ as follows.
	
	\begin{theoremletters}[Omnibus theorem] \label{THEOREM B}
		\hfill
		\begin{enumerate}[label={\normalfont{(\arabic*)}}]
			\item\label{ALM19} {\it {\bf Alternative solution to AM's question:} $\ces$ is isomorphic to $\Ces$}
			(Theorem~\ref{COR: Ces and ces isomorphic}).
			\vspace{5pt}
			\item\label{BM-distance}  {\it {\bf Banach--Mazur distance:} $d_{\mathrm{BM}}(\ces,\Ces) \leqslant 1296$} (Remark~\ref{REMARK : d_BM(k,K)<=}).
			\vspace{5pt}
			\item\label{NOT-isometric} {\it {\bf Isometric variant of AM's question:} $\ces$ and $\Ces$ are not isometric}
			(Theorem~\ref{THEOREM : ces and Ces NOT isometric}).
			\vspace{5pt}
			\item\label{Hagler-Stegall} {\it {\bf Hagler--Stegall\footnote{We chose this name for the following reason:
						$\Ces$ has a complemented subspace isomorphic to $L_1[0,1]$, so --- due to Hagler and Stegall \cite{HS73} --- it has a complemented subspace
						isomorphic to $\mathscr{M}[0,1] \cong \mathcal{C}[0,1]^{*}$ (by the Riesz--Markov--Kakutani representation theorem).} style result:}
				$\Ces$ is isomorphic to $( \bigoplus_{n=1}^{\infty} \mathscr{M}[0,1] )_{\ell_{\infty}}$}
			(Theorem~\ref{COROLLARY : Michau}).
			\vspace{5pt}
			\item\label{isometric-preduals} {\it {\bf Unique isometric predual:} $\ces$ has a unique isometric predual, namely the space $\widetilde{\ell}_1$}
			(Theorem~\ref{THEOREM : ces_infty has UNIQUE predual}).
			\vspace{5pt}
			\item {\it {\bf Number of pairwise non-isomorphic preduals:} $\ces$ has at least continuum many pairwise non-isomorphic preduals}
			(Corollary~\ref{COROLLARY : c many preduals}; {\it cf.} Theorem~\ref{PROP : ces has uncountably many preduals}).
			\vspace{5pt}
			\item\label{Strong-Schur} {\it {\bf Schur property:} $\widetilde{\ell}_1$ has the strong Schur property}
			(Lemma~\ref{PROPOSITION: Schur Tandori l1 ZF}).
		\end{enumerate}
	\end{theoremletters}
	
	Theorem~\ref{THEOREM B}\ref{Strong-Schur} strengthens \cite[Theorem~3.1]{ALM19}, where it was shown that $\widetilde{\ell}_1$ has the Schur property
	(see Remark~\ref{REMARK : Schur} for a more detailed discussion).
	Again, it is a manifestation of a more general mechanism:
	A quantitative version of the Schur property, called the strong Schur property, is preserved by direct sums of finite-dimensional Banach spaces.
	
	\begin{theoremletters}[Lemma~\ref{PROPOSITION: Schur Tandori l1 ZF}] \label{THEOREM CCC}
		{\it Let $\mathsf{X}$ be a Banach sequence space with the $c$-strong Schur property and the Fatou property,
			and let $\{\mathsf{F}_n\}_{n=1}^{\infty}$ be a sequence of finite-dimensional Banach spaces.
			Then the space $(\bigoplus_{n=1}^{\infty} \mathsf{F}_n)_{\mathsf{X}}$ has the $c$-strong Schur property.}
	\end{theoremletters}
	
	\subsection{Solovayan functional analysis}
	
	Finally, we allow ourselves a certain extravagance in the spirit of V{\" a}th's paper \cite{Vat98}:
	we show that there are models of set theory without\footnote{It should be noted that we do not
		abandon the Axiom of Choice $\mathsf{AC}$ completely, but retain part of it in the form of the
		Principle of Dependent Choice $\mathsf{DC}$, which is the most recognisable and probably the most natural
		weakening of $\mathsf{AC}$.} the Axiom of Choice
	(for example, Solovay's model, in which every subset of the real line $\mathbb{R}$ is Lebesgue measurable)
	in which $\ces$ and $\Ces$ are {\bf not} isomorphic.
	
	\begin{theoremletters}[Theorem~\ref{THM : In ZF ces_infty and Ces_infty are NOT isomorphic}] \label{THEOREM C}
		{\it Assume $\mathsf{ZF}$, $\mathsf{DC}$ and $\mathsf{PM}_{\omega}$.
			Then $\ces$ and $\Ces$ are not isomorphic.}
	\end{theoremletters}
	
	Roughly speaking, Theorem~\ref{THEOREM C} says that any isomorphism between $\ces$ and $\Ces$ must be somewhat \enquote{exotic}.
	The same happens for $\ell_{\infty}$ and $L_{\infty}$, which supports the main message of Section~\ref{SUBSECTION : Vademaecum}.
	
	The argument behind Theorem~\ref{THEOREM C} is simple:
	if $\ces$ and $\Ces$ were isomorphic, then their duals $\ces^{*}$ and $\Ces^{*}$ would also be isomorphic;
	this is, however, impossible, because $\ces^{*}$ has the Schur property while $\Ces^{*}$ does not.
	The lack of $\mathsf{AC}$ removes some convenient tools from our toolbox,
	and this makes analysis a slightly stranger\footnote{But perhaps also a more interesting one.} place (we recommend Karagila's essay \cite{Kar20}
	on the interactions between various forms of $\mathsf{AC}$ and functional analysis).
	Therefore, on the way to Theorem~\ref{THEOREM C}, we cannot rely entirely on the results available in the literature
	(Scholium~\ref{SCHOLIUM : reinventing the wheel}).
	We encounter three scenarios:
	\begin{enumerate}
		\item[$\bullet$] some results remain true with the same proofs (Lemma~\ref{LEMMA : tandori_1 = direct sum});
		\item[$\bullet$] some results remain true, but their proofs require a different argument (Lemmas~\ref{LEMMA : dual to ces_infty is t_1 in ZF+DC} and~\ref{PROPOSITION: Schur Tandori l1 ZF});
		\item[$\bullet$] some results are no longer true (Theorem~\ref{COR: Ces and ces isomorphic}).
	\end{enumerate}
	The resolution of these issues is the content of Section~\ref{SUBSECTION : Vath}.
	
	\subsection{Quick recap and outline}
	
	Below we give a condensed summary of the main goals and ideas.
	
	\vspace{5pt}
	\noindent
	$\bigstar$ Our primary goal was to build an abstract framework in which the two main results about $\ces$ and $\Ces$ from \cite{ALM19}
	follow as simple consequences (Theorems~\ref{THEOREM BB} and \ref{THEOREM BBB}).
	This approach has at least two advantages.
	On the one hand, it produces results that may be of interest in general Banach space theory
	(Theorems~\ref{THEOREM A}, \ref{THEOREM BBBB} and \ref{THEOREM CCC}).
	On the other hand, it allows us to view certain phenomena as realisations of the same mechanism (see Section~\ref{SUBSECTION : Grand unification}).
	
	\vspace{5pt}
	\noindent
	$\bigstar$ Since the spaces $\ces$ and $\Ces$ seem to be rather little known (or known under many different guises),
	we do not assume any prior knowledge of their isomorphic structure.
	Instead, we have tried to make this work self-contained in this respect (Section~\ref{SUBSECTION : aide})
	and to present them as natural, and perhaps even interesting, objects closely related to their classical counterparts
	$\ell_{\infty}$ and $L_{\infty}$.
	
	\vspace{5pt}
	\noindent
	$\bigstar$ The idea that $\ces$ and $\Ces$ behave in many ways like the classical pair $\ell_{\infty}$ and $L_{\infty}$
	is one of the main motifs developed in Sections~\ref{SECTION : Main results} and \ref{SECTION : complementary material}
	(this is summarised in Theorem~\ref{THEOREM B}).
	
	\vspace{5pt}
	\noindent
	$\bigstar$ Our intention was to gather all the necessary ideas and arguments under one roof.
	This has, of course, affected the length of this work.
	Nevertheless, we hope that the accessibility of the ideas and the completeness of the presentation make the paper reasonably easy to read.

	The paper is divided into seven sections, followed by the bibliography;
	informally, the logical road map of the paper can be depicted as follows:
	\begin{equation*}
		\begin{tikzcd}
			\ref{SECTION : Gateway} \arrow[rd, dashed] &             &                     &   \\
			& \ref{SECTION : Applications} \arrow[r] & \ref{SECTION : complementary material} \arrow[r, dashed] & \ref{SECTION : Open questions} \\
			\ref{SECTION : Main results} \arrow[ru] &             &                     &  
		\end{tikzcd}
	\end{equation*}
	
	Each section begins with a short vignette explaining its contents, so we do not dwell on this here.
	
	Sections~\ref{SECTION : Gateway} and \ref{SECTION : Main results} can be read independently of the rest of the paper and
	require nothing but a general knowledge of Banach space theory
	(Section~\ref{SUBSECTION : Notation} fixes the notation).
	We emphasise, however, that Section~\ref{SECTION : Main results} contains the main results of this work, and the role of Section~\ref{SECTION : Gateway} is only auxiliary.
	
	The results concerning the spaces $\ces$ and $\Ces$ occupy Sections~\ref{SECTION : Applications} and \ref{SECTION : complementary material}
	(and we recommend reading them in that order).
	We briefly introduce the spaces $\ces$ and $\Ces$ in Section~\ref{SUBSECTION : aide}.
	The results of Section~\ref{SECTION : Applications} rely critically on Section~\ref{SECTION : Main results}.
	
	Reading Section~\ref{SECTION : Open questions} is optional: We have collected there several open questions arising from
	Sections~\ref{SECTION : Main results}, \ref{SECTION : Applications} and \ref{SECTION : complementary material}.
	
	\section{{\bf Declarations}}
	\addtocontents{toc}{\protect\setcounter{tocdepth}{0}}
	
	\subsection{Acknowledgements}
	Some of the main results of this work were presented (in a preliminary form) by the first-named author at the conference Function Spaces~13 held in Pozna{\' n} in 2024.
	We heartily thank Marek C{\' u}th for pointing out an inaccuracy in one of our earlier arguments, which eventually
	led us to revise our views.
	We would also like to thank Hugh Wark for a very pleasant conversation about \cite{War97} (and many other matters) at the FAME 2026 conference banquet.
	Last, but not least, we thank Karol Le{\' s}nik and Lech Maligranda for some clarifications regarding \cite{ALM19} and \cite{AM09}.
	
	We also apologise for the long delay in completing this work, which resulted mostly from the birth of the first author's second son,
	to whom this work is dedicated.
	
	\subsection{Funding}
	Our research was carried out at Pozna{\' n} University of Technology (grant number 0213/SBAD/0116).
	
	\subsection{Open access}
	
	For the purpose of open access, the authors have applied a Creative Commons Attribution
	(CC-BY) licence to any Author Accepted Manuscript version arising.
	
	\subsection{Conflict of interest}
	On behalf of both authors, the corresponding author declares that they have no known competing financial interests or personal
	relationships that could have appeared to influence the work reported in this paper.
	
	\subsection{Data availability}
	This work is purely theoretical; no datasets were generated or analysed.
	
	\subsection{Use of artificial intelligence}
	We have been working on this project since March 2022
	(that was more or less the moment when the first-named author found \cite{War97}, but both of us had been aware of the main result from \cite{ALM19} since at least 2015).
	A draft of the paper containing our most interesting findings
	(notably Theorems~\ref{LEMMA : 6}, \ref{LEMMA : 6infty}, \ref{THM : suma L1 isomorficzna z miarami}, \ref{THEOREM : ces and Ces NOT isometric},
	\ref{PROP : ces has uncountably many preduals}
	and \ref{THM : In ZF ces_infty and Ces_infty are NOT isomorphic})
	was ready before June 2024
	(but it was never made publicly available).
	Since then, through conversations with various mathematicians, these results have undergone some transformation
	(mainly concerning Theorem~\ref{LEMMA : 6}).
	The first complete version was written only around September 2026 without any use of artificial intelligence.
	
	In the final stage of preparation, Claude Opus 5.5 (developed by Anthropic) was utilized to assist in:
	\begin{enumerate}
		\item[$\bullet$] language editing;
		\item[$\bullet$] reviewing the structural logic of the proofs and identifying potential gaps;
		\item[$\bullet$] checking existing literature and comparing it with our results.
	\end{enumerate}
	All suggestions provided by AI were independently reconstructed, checked, and validated by the authors.
	Both authors retain sole and absolute responsibility for the correctness, logical mathematical validity, adequacy of all arguments,
	and the accuracy of all prior citations presented in this work.
	However, we would like to emphasize that all the ideas and core concepts contained in this work are the authors' own and they were devised without the slightest use of AI.
	
	\addtocontents{toc}{\protect\setcounter{tocdepth}{2}}
	
	\section{{\bf Gateway}} \label{SECTION : Gateway}
	
	In this section, we collect a number of auxiliary lemmas that will be crucial later.
	Some of them are certainly well known, some may be familiar, and others are probably folklore among specialists.
	Be that as it may, for the sake of completeness we present all the details.
	
	We do not hide the obvious inspirations, and sometimes even borrowings, from \cite{AF96}, \cite{CG94}, \cite{CKL77}, \cite{HS73} and \cite{War97}
	(see also \cite{AA17}, \cite{AK26}, \cite[p.~317]{CS88}, \cite[Remark after Theorem~1]{Joh72}, \cite{Kan13}, \cite[Section~4]{Lus96} and \cite{War07}).
	
	\subsection{Setup} \label{SUBSECTION : Notation}
	
	We follow the principle that definitions and concepts are introduced only where they are actually needed.
	It seems to us that the only objects going beyond the canon of Banach space theory are the spaces $\ces$ and $\Ces$ themselves
	(Section~\ref{SUBSECTION : aide} is devoted to them).
	This eliminates the need for an extensive preliminary section.
	However, for the reader's convenience, we fix the most important notation and conventions below.
	
	\vspace{5pt}
	\paragraph{{\bf General notation}} \label{PARAGRAPH : General notation}
	
	\begin{enumerate}
		\item[$\bullet$] As usual, $\mathbb{N} = \{1,2, ...\}$, $\mathbb{Z}$, $\mathbb{R}$ and $\mathbb{C}$ denote the sets of natural numbers,
		integers, reals and complex numbers, respectively, and $\mathbb{N}_0 = \mathbb{N} \cup \{0\}$.
		\item[$\bullet$] The ground field is $\mathbb{R}$, unless otherwise stated.
		However, up to some potential changes in the constants, all results presented in this paper remain valid in both the real and complex settings.
		\item[$\bullet$] The letter $\mathbb{K}$ stands for either $\mathbb{R}$ or $\mathbb{C}$.
		\item[$\bullet$] The symbol $\coloneqq$ means \enquote{is defined as}.
		\item[$\bullet$] We use the convention that the symbol $1/p$ means $0$ when $p = \infty$.
	\end{enumerate}
	
	\vspace{5pt}
	\paragraph{{\bf Axioms and cardinals}}
	
	\begin{enumerate}
		\item[$\bullet$] Throughout, $\mathsf{ZFC}$ denotes the usual setting of set theory, consisting of the Zermelo--Fraenkel axioms $\mathsf{ZF}$
		together with the Axiom of Choice ($\mathsf{AC}$, for short).
		Unless explicitly stated otherwise (as in Section~\ref{SUBSECTION : Vath}), we work in $\mathsf{ZFC}$.
		\item[$\bullet$] We follow the presentation of ordinal and cardinal numbers in \cite[Appendix~A]{Kec95}.
		In particular, the cardinality of $\mathbb{N}$ is denoted by $\aleph_0$, $\mathfrak{c} = 2^{\aleph_0}$ denotes the cardinality of the {\bf continuum},
		and $\aleph_1$ is the first uncountable cardinal.
		\item[$\bullet$] By $\mathrm{card}\, S$ we denote the {\bf cardinality} of a set $S$.
	\end{enumerate}
	
	\vspace{5pt}
	\paragraph{{\bf Measurable functions}}
	
	\begin{enumerate}
		\item[$\bullet$] By $\lambda$ we denote the {\bf Lebesgue measure} on the real line $\mathbb{R}$,
		and by $\#$ the {\bf counting measure} on $\mathbb{N}$, that is, $\#(A) = \sum_{n \in A} 1$ for $A \subset \mathbb{N}$.
		\item[$\bullet$] The {\bf characteristic function} $\chi_A$ of a measurable set $A \subset \Omega$ (here, $\Omega$ is the ambient measure space)
		takes the value one on $A$ and zero on $\Omega \setminus A$.
		\item[$\bullet$] The {\bf support} $\supp f$ of a measurable function $f$ on $\Omega$ is
		the set of those points of $\Omega$ at which $f$ is non-zero.
	\end{enumerate}
	
	\vspace{5pt}
	\paragraph{{\bf Geometry of Banach spaces}}
	
	Our general references for the geometry of Banach spaces are the monographs of
	Albiac and Kalton \cite{AK06},
	Fabian, Habala, H{\' a}jek, {\it et al.} \cite{FHHMSPZ01},
	and Megginson \cite{Meg98}.
	
	\begin{enumerate}
		\item[$\bullet$] We denote Banach spaces by capital sans serif letters such as $\mathsf{X,Y,Z,}\ldots$ to ty\-po\-graph\-i\-cal\-ly distinguish them from other objects.
		\item[$\bullet$] Unless explicitly stated otherwise, a Banach space $\mathsf{X}$ is always an infinite-dimensional real Banach space,
		and a subspace of $\mathsf{X}$ is a closed linear subspace of $\mathsf{X}$.
		\item[$\bullet$] For a Banach space $\mathsf{X}$, we write $\text{Ball}(\mathsf{X})$ for the {\bf closed unit ball} of $\mathsf{X}$, that is,
		$\text{Ball}(\mathsf{X}) \coloneqq \{ x \in \mathsf{X} \colon \norm{x}_{\mathsf{X}} \leqslant 1 \}$,
		and $S(\mathsf{X}) \coloneqq \{ x \in \mathsf{X} \colon \norm{x}_{\mathsf{X}} = 1 \}$ for the {\bf unit sphere} of $\mathsf{X}$.
		\item[$\bullet$] The symbol $\boldsymbol{e}_{n}$ with $n \in \mathbb{N}$ stands for the $n^{\mathrm{th}}$ element of the unit vector basis of $c_0$,
		that is,
		\begin{equation*}
			\boldsymbol{e}_n \coloneqq (0,0,..., 0, \underbrace{1}_{n^{\mathrm{th}} \text{ position}}, 0, ...) \qquad (n \in \mathbb{N}).
		\end{equation*}
		\item[$\bullet$] As usual, $\mathfrak{L}(\mathsf{X},\mathsf{Y})$, or just $\mathfrak{L}(\mathsf{X})$ when $\mathsf{X} = \mathsf{Y}$, stands for the space
		of all {\bf bounded linear operators} from $\mathsf{X}$ into $\mathsf{Y}$ equipped with the {\bf operator norm}
		$\norm{T \colon \mathsf{X} \rightarrow \mathsf{Y}} \coloneqq \sup \{ \norm{Tx}_{\mathsf{Y}} \colon x \in \text{Ball}(\mathsf{X}) \}$.
		Depending on the context, we also abbreviate $\norm{T \colon \mathsf{X} \rightarrow \mathsf{Y}}$
		to $\lVert T \rVert_{\mathsf{X} \rightarrow \mathsf{Y}}$ or even $\norm{T}$.
		By an operator we always mean a bounded linear operator.
		\item[$\bullet$] For a Banach space $\mathsf{X}$, by $\mathsf{X}^{*}$ we denote the {\bf topological dual} of $\mathsf{X}$,
		that is, $\mathsf{X}^{*} \coloneqq \mathfrak{L}(\mathsf{X},\mathbb{K})$.
		We write $\langle x, x^{*} \rangle = \langle x^{*}, x \rangle \coloneqq x^{*}(x)$ for $x \in \mathsf{X}$ and $x^{*} \in \mathsf{X}^{*}$.
		We always regard $\mathsf{X}$ as a subspace of its bidual $\mathsf{X}^{**}$ via the canonical embedding.
	\end{enumerate}
	
	\vspace{5pt}
	\paragraph{{\bf Special Banach lattices}}
	
	Except for
	Remark~\ref{REMARK : K^p},
	Section~\ref{SUBSECTION : Ces_infty and ces_infty are NOT isometric},
	Section~\ref{SUBSECTION : Vath}
	and Section~\ref{SUBSECTION : On the Schur property},
	we require no knowledge of Banach lattices,
	Banach ideal spaces
	or rearrangement-invariant spaces.
	Nevertheless, in some situations --- such as those mentioned above --- it is most convenient to regard $\ces$ and $\Ces$
	as a Banach sequence space and a Banach function space, respectively.
	
	Our general references for function space theory are
	Bennett and Sharpley's monograph \cite{BS88},
	Lindenstrauss and Tzafriri's two volumes \cite{LT77,LT79},
	and Zaanen's classical book \cite{Zaa67}
	(see \cite{Vat97} whenever we work with $\mathsf{DC}$ instead of $\mathsf{AC}$).
	
	\begin{enumerate}
		\item[$\bullet$] We never assume {\it a priori} that a given Banach ideal space $\mathsf{X}$ has the Fatou property
		(therefore, the use of \cite{BS88} requires some caution).
		\item[$\bullet$] We consider either {\bf Banach sequence spaces}, that is, Banach ideal spaces on $\mathbb{N}$
		with the counting measure $\#$, or {\bf Banach function spaces}, that is, Banach ideal spaces on $\mathbb{I}$,
		where $\mathbb{I}=[0,1]$ or $\mathbb{I}=[0,\infty)$, with the Lebesgue measure $\lambda$
		(but see Remark~\ref{REMARK : measure theory without choice}).
		\item[$\bullet$] By the Open Mapping Theorem, the (set-theoretic) inclusion $\mathsf{X} \subset \mathsf{Y}$
		between two Banach ideal spaces $\mathsf{X}$ and $\mathsf{Y}$ is always continuous, that is,
		\begin{equation*}
			\sup \left\{ \norm{f}_{\mathsf{Y}} : f \in \mathrm{Ball}(\mathsf{X}) \right\}
			= \norm{\mathrm{id} \colon \mathsf{X} \rightarrow \mathsf{Y}}
			< \infty.
		\end{equation*}
		For this reason, we prefer to write $\mathsf{X} \hookrightarrow \mathsf{Y}$ rather than $\mathsf{X} \subset \mathsf{Y}$.
		Abusing notation a little, we write $\mathsf{X} = \mathsf{Y}$ if the spaces $\mathsf{X}$ and
		$\mathsf{Y}$ coincide as sets and their norms are equivalent,
		and $\mathsf{X} \equiv \mathsf{Y}$ if, moreover, the norms $\norm{\cdot}_{\mathsf{X}}$
		and $\norm{\cdot}_{\mathsf{Y}}$ are equal.
		\item[$\bullet$] Given a Banach ideal space $\mathsf{X}$, the symbol $\mathsf{X}^{\times}$ stands for the {\bf K{\" o}the dual} of $\mathsf{X}$.
		Note that $\mathsf{X}^{\times} = \mathsf{X}^{*}$ provided $\mathsf{X}$ is separable.
		\item[$\bullet$] An element $x$ of a Banach sequence space $\mathsf{X}$ is {\bf order continuous} (or has an {\bf absolutely continuous norm})
		if for every $\delta > 0$ there is $n \in \mathbb{N}$ such that $\lVert x\chi_{[n,\infty)} \rVert_{\mathsf{X}} < \delta$.
		The space $\mathsf{X}$ is {\bf order continuous} provided $\mathsf{X} = \mathsf{X}_o$, where $\mathsf{X}_o$ is the ideal of all order continuous
		elements of $\mathsf{X}$.
	\end{enumerate}
	
	\vspace{5pt}
	\paragraph{{\bf Pointwise multipliers}}
	
	The {\bf space of pointwise multipliers} $\mathscr{M}(\mathsf{X},\mathsf{Y})$ between two Banach ideal spaces $\mathsf{X}$
	and $\mathsf{Y}$ consists of those measurable functions $f$ for which $fg \in \mathsf{Y}$
	for all $g \in \mathsf{X}$.
	It is natural to equip $\mathscr{M}(\mathsf{X},\mathsf{Y})$ with the operator norm
	\begin{equation*}
		\norm{f}_{\mathscr{M}(\mathsf{X},\mathsf{Y})} \coloneqq \norm{M_f \colon \mathsf{X} \rightarrow \mathsf{Y}},
	\end{equation*}
	where $M_f$ is the multiplication operator with symbol $f$, that is, $M_f(g) = fg$ for $g \in \mathsf{X}$.
	This makes $\mathscr{M}(\mathsf{X},\mathsf{Y})$ a Banach ideal space.
	Note that $\mathsf{X}^{\times} = \mathscr{M}(\mathsf{X},L_1)$, so the construction $\mathscr{M}(\mathsf{X},\mathsf{Y})$
	can be seen as an abstraction of K{\" o}the duality.
	This leads, among other things, to the following general variant of the {\bf H{\" o}lder--Rogers inequality}:
	for any two measurable functions $f$ and $g$,
	\begin{equation*}
		\norm{fg}_{\mathsf{Y}} \leqslant \norm{g}_{\mathsf{X}} \norm{f}_{\mathscr{M}(\mathsf{X},\mathsf{Y})}.
	\end{equation*}
	For more information about this construction we refer to \cite{KT24}, \cite{KT25}, \cite{KT26} and the references therein.
	
	\vspace{5pt}
	\paragraph{{\bf Direct sums}}
	
	Given a sequence $\{ \mathsf{X}_n \}_{n=1}^{\infty}$ of Banach spaces and a Banach sequence space $\mathsf{Z}$,
	we denote by $( \bigoplus_{n=1}^{\infty} \mathsf{X}_n )_{\mathsf{Z}}$ the {\bf $\mathsf{Z}$-sum} of $\{ \mathsf{X}_n \}_{n=1}^{\infty}$,
	that is, the vector space of those ${\bf f} \in \prod_{n \in \mathbb{N}} \mathsf{X}_n$ for which the (extended) norm
	\begin{equation*}
		\norm{ {\bf f} }_{( \bigoplus_{n=1}^{\infty} \mathsf{X}_n )_{\mathsf{Z}}}
		\coloneqq \norm{\sum_{n=1}^{\infty} \norm{{\bf f}(n)}_{\mathsf{X}_n} \boldsymbol{e}_{n}}_{\mathsf{Z}}
	\end{equation*}
	is finite.
	Endowed with the coordinatewise operations,
	$( \bigoplus_{n=1}^{\infty} \mathsf{X}_n )_{\mathsf{Z}}$ is a Banach space.
	The $\mathsf{Z}$-sum $( \bigoplus_{n=1}^{\infty} \mathsf{F}_n )_{\mathsf{Z}}$ of finite-dimensional Banach spaces $\mathsf{F}_n$
	can be naturally understood as a Banach sequence space.
	For the most part (except for Lemma~\ref{PROPOSITION: Schur Tandori l1 ZF}
	and Section~\ref{SUBSECTION : On the Schur property}) we use only $\ell_p$-sums with $1 \leqslant p \leqslant \infty$,
	but it is convenient to have the general construction at hand.
	For two Banach spaces we write $\mathsf{X} \oplus_p \mathsf{Y}$ for their $\ell_p$-sum, and $\mathsf{X} \oplus \mathsf{Y} \coloneqq \mathsf{X} \oplus_\infty \mathsf{Y}$.
	
	An important special case arises when the sequence $\{ \mathsf{X}_n \}_{n=1}^{\infty}$ is {\bf constant}, that is,
	$\mathsf{X}_n = \mathsf{X}$ for all $n \in \mathbb{N}$.
	In this situation the $\mathsf{Z}$-sum $( \bigoplus_{n=1}^{\infty} \mathsf{X} )_{\mathsf{Z}}$ is abbreviated to $\mathsf{Z}(\mathsf{X})$
	and called the {\bf K{\" o}the--Bochner space}.
	Members of the families $\{ \ell_p(\ell_q) : 1 \leqslant p,q \leqslant \infty \}$
	and $\{ (\bigoplus_{n=1}^{\infty} \ell_p^n)_{\ell_q} : 1 \leqslant p,q \leqslant \infty \}$ are sometimes called {\bf matrix spaces}
	or {\bf mixed-norm sequence spaces} (they are isomorphic to Besov spaces; see \cite{AA17} and the references therein).
	
	\subsection{Some notes on the Banach--Mazur distance}
	
	\begin{notation}[Banach--Mazur distance]
		Two Banach spaces $\mathsf{X}$ and $\mathsf{Y}$ are {\bf $c$-isomorphic}, written $\mathsf{X} \approx_{c} \mathsf{Y}$,
		if there is a bounded linear bijection $T \colon \mathsf{X} \rightarrow \mathsf{Y}$ with distortion $\norm{T} \lVert T^{-1}\rVert$ at most $c$.
		When the precise value of $c$ is irrelevant,
		we simply say that $\mathsf{X}$ and $\mathsf{Y}$ are isomorphic and write $\mathsf{X} \approx \mathsf{Y}$.
		The symbol $\mathsf{X} \cong \mathsf{Y}$ is reserved for the case where $\mathsf{X}$ and $\mathsf{Y}$ are {\bf isometrically isomorphic}
		(briefly, {\bf isometric}),
		that is, there is a bijection $T$ as above with $\norm{T} = \lVert T^{-1} \rVert = 1$.
		Following \cite[Chapter~9]{TJ89}, the {\bf Banach--Mazur
			distance} between two isomorphic Banach spaces $\mathsf{X}$ and $\mathsf{Y}$ is
		\begin{equation*}
			d_{\mathrm{BM}}(\mathsf{X},\mathsf{Y}) \coloneqq \inf \left\{ \norm{T} \lVert T^{-1} \rVert \right\},
		\end{equation*}
		where the infimum runs over all isomorphisms of $\mathsf{X}$ onto $\mathsf{Y}$
		(with the natural convention that $d_{\mathrm{BM}}(\mathsf{X},\mathsf{Y}) = \infty$ if $\mathsf{X}$ and $\mathsf{Y}$ are not isomorphic).
		This notion first appeared in Banach's monograph \cite{Ban32}
		(see the remarks to \cite[Chapter~XI]{Ban32}).
		
		Roughly speaking, the Banach--Mazur distance measures how far $\text{Ball}(\mathsf{X})$ is from a linear image of $\text{Ball}(\mathsf{Y})$.
		Plainly, $\mathsf{X} \cong \mathsf{Y}$ implies $d_{\mathrm{BM}}(\mathsf{X},\mathsf{Y}) = 1$;
		however, in general $d_{\mathrm{BM}}(\mathsf{X},\mathsf{Y}) = 1$ does {\bf not} imply that $\mathsf{X}$ and $\mathsf{Y}$ are isometric
		(see \cite[\S37]{TJ89} and the references therein).
		Moreover, even though we call $d_{\mathrm{BM}}$ \enquote{distance}, only its logarithm satisfies the triangle inequality.
		Anyway, $d_{\mathrm{BM}}$ is symmetric, that is, $d_{\mathrm{BM}}(\mathsf{X},\mathsf{Y}) = d_{\mathrm{BM}}(\mathsf{Y},\mathsf{X})$.
	\end{notation}
	
	\begin{remark}[Banach--Mazur distance between finite-dimensional $\ell_p$- and $\ell_q$-spaces] \label{REMARK : d_BM(l_p^n,l_q^n) <=}
		Let $1\leqslant p,q\leqslant\infty$ and $n \in \mathbb{N}$.
		We will repeatedly use the well-known fact that we have the following estimate:
		\begin{equation*}
			d_{\mathrm{BM}}(\ell_p^n,\ell_q^n) \leqslant 2^{\abs{1/p-1/q}}.
		\end{equation*}
		Indeed, by the symmetry of $d_{\mathrm{BM}}$ we may assume that $p \leqslant q$, so that $\abs{1/p-1/q} = 1/p-1/q$.
		Let $T$ be the formal identity from $\ell_p^n$ into $\ell_q^n$.
		Then, it is clear that $\norm{T} \leqslant 1$ and $\lVert T^{-1} \rVert \leqslant n^{1/p-1/q}$.
		Hence $d_{\mathrm{BM}}(\ell_p^n,\ell_q^n) \leqslant \norm{T}\lVert T^{-1} \rVert \leqslant n^{1/p-1/q}$.
		
		Note that the above inequality is optimal if $p$ and $q$ lies on the same side of 2, that is, if either $1 \leqslant p \leqslant q \leqslant 2$
		or $2 \leqslant p \leqslant q \leqslant \infty$:
		By a classical result of Gurari{\u \i}, Kadec and Macaev \cite{GKM65}, we have
		\begin{equation*}
			d_{\mathrm{BM}}(\ell_p^n,\ell_q^n) = 2^{\abs{1/p-1/q}}.
		\end{equation*}
		In general, when $p < 2 < q$, the correct order of $d_{\mathrm{BM}}(\ell_p^n,\ell_q^n)$ is $\max\{n^{1/p-1/2},n^{1/2-1/q}\}$,
		up to some constant (see, once again, \cite{GKM65}).
		\demo
	\end{remark}
	
	We now record a quantitative version of Pe{\l}czy{\' n}ski's decomposition method from \cite{Pel60} ({\it cf.} \cite[Theorem~2.2.3, p.~34]{AK06}).
	Several ready-to-use versions are available in the literature (see \cite[Theorem~A.1]{AK26}, \cite[Theorem~5.1]{KL26} and \cite[Theorem~5.1]{KP25}).
	However, for the reader's convenience, we present a short self-contained version adapted to our needs.

	\begin{definition} \label{DEF : complementably embedded}
		Let $\lambda \geqslant 1$, and let $\mathsf{X}$ and $\mathsf{Y}$ be two Banach spaces.
		We say that $\mathsf{Y}$ is {\bf $\lambda$-complementably embedded} in $\mathsf{X}$, written $\mathsf{Y} \cembed{\lambda} \mathsf{X}$,
		if there are operators $J \colon \mathsf{Y} \rightarrow \mathsf{X}$ and $Q \colon \mathsf{X} \rightarrow \mathsf{Y}$ with $QJ = \mathrm{id}_{\mathsf{Y}}$
		and $\norm{J}\norm{Q} \leqslant \lambda$.
		Moreover, we write $\mathsf{Y} \cembed{} \mathsf{X}$ if $\mathsf{Y} \cembed{\lambda} \mathsf{X}$ for some $\lambda>0$.
	\end{definition}
	
	Note that if $\mathsf{X}$ has a $\lambda$-complemented subspace which is $\mu$-isomorphic to $\mathsf{Y}$,
	then $\mathsf{Y} \cembed{\lambda\mu} \mathsf{X}$.
	Moreover, the relation $\cembed{}$ is transitive, that is, $\mathsf{Z} \cembed{\lambda} \mathsf{Y} \cembed{\mu} \mathsf{X}$ implies $\mathsf{Z} \cembed{\lambda\mu} \mathsf{X}$.
	(Recall, referring to Section~\ref{SUBSECTION : Notation}, that the symbol $\oplus$ stands for the $\ell_\infty$-sum.)
	
	\begin{lemma} \label{LEMMA : splitting}
		{\it If $\mathsf{Y} \cembed{\lambda} \mathsf{X}$ via $J$ and $Q$, then $\mathsf{X} \approx_{1+2\lambda} \mathsf{Y} \oplus \mathsf{F}$, where $\mathsf{F} \coloneqq \ker Q$.}
	\end{lemma}
	\begin{proof}
		Note that replacing $J$ and $Q$ by $tJ$ and $t^{-1}Q$, respectively, we may assume that $\norm{J} = t$ and $\norm{Q} \leqslant \lambda/t$
		(the exact value of $t$ will be specified in a moment).
		Let $Sx \coloneqq (Qx, x - JQx)$ for $x \in \mathsf{X}$.
		Then $S \colon \mathsf{X} \rightarrow \mathsf{Y} \oplus \mathsf{F}$ is a bijection with $S^{-1}(y,z) = Jy + z$.
		Moreover,
		\begin{equation*}
			\norm{S} \leqslant \max\{ \lambda/t, 1 + \lambda \}
			\quad\text{and}\quad
			\norm{S^{-1}} \leqslant t + 1.
		\end{equation*}
		Taking $t = \lambda/(1+\lambda)$, we get
		\begin{equation*}
			\norm{S}\norm{S^{-1}} \leqslant (1+\lambda) \cdot \frac{1+2\lambda}{1+\lambda} = 1 + 2\lambda
		\end{equation*}
		which concludes the proof.
	\end{proof}

	We say that $\mathsf{X}$ satisfies {\bf Pe{\l}czy{\' n}ski's accordion property} with constant $h$
	provided $\mathsf{X} \approx_h \ell_p(\mathsf{X})$ for some $1 \leqslant p \leqslant \infty$ or $\mathsf{X} \approx_h c_0(\mathsf{X})$.
	
	\begin{theorem}[Quantitative variant of Pe{\l}czy{\' n}ski's decomposition method] \label{THM : QPDM}
		{\it Let $\mathsf{X}$ and $\mathsf{Y}$ be two Banach spaces such that $\mathsf{Y} \cembed{\lambda} \mathsf{X}$ and $\mathsf{X} \cembed{\mu} \mathsf{Y}$,
			and put $a \coloneqq 1 + 2\mu$ and $b \coloneqq 1 + 2\lambda$.}
		\begin{enumerate}[label={\normalfont{(\arabic*)}}]
			\item\label{1C} {\it If $\mathsf{X} \approx_e \mathsf{X} \oplus \mathsf{X}$ and $\mathsf{Y} \approx_f \mathsf{Y} \oplus \mathsf{Y}$,
			then}
			\begin{equation*}
				d_{\mathrm{BM}}(\mathsf{X},\mathsf{Y}) \leqslant a^2 b^2 e f.
			\end{equation*}
			\item\label{2C} {\it If $\mathsf{X}$ satisfies Pe{\l}czy{\' n}ski's accordion property with constant $h$,
			then
			\begin{equation*}
				d_{\mathrm{BM}}(\mathsf{X},\mathsf{Y}) \leqslant \kappa^4 a^2 b^2 h^4,
			\end{equation*}
			where the constant $\kappa \geqslant 1$ depends only on $\mathsf{X}$.
			Moreover, the proof below gives an explicit value of $\kappa$.}
		\end{enumerate}
	\end{theorem}
	\begin{proof}
		By Lemma~\ref{LEMMA : splitting}, $\mathsf{X} \approx_b \mathsf{Y} \oplus \mathsf{F}$ and $\mathsf{Y} \approx_a \mathsf{X} \oplus \mathsf{E}$ for some Banach spaces $\mathsf{E}$ and $\mathsf{F}$.
		
		Recall also that, for any Banach spaces $\mathsf{V}$ and $\mathsf{W}$, the identity maps give
		\begin{equation*}
			\ell_p(\mathsf{V} \oplus \mathsf{W}) \approx_{\kappa_p} \ell_p(\mathsf{V}) \oplus \ell_p(\mathsf{W})\text{ and }
			\mathsf{V} \oplus \ell_p(\mathsf{V}) \approx_{\kappa_p} \ell_p(\mathsf{V}),
		\end{equation*}
		where $\kappa_p = d_{\mathrm{BM}}(\ell_p^2,\ell_{\infty}^2) = 2^{1/p}$ for $1 \leqslant p \leqslant \infty$
		(the case of $c_0$-sums is the same as of $\ell_{\infty}$-sums; {\it cf}. Remark~\ref{REMARK : d_BM(l_p^n,l_q^n) <=}).
		
		\ref{1C} We have
		\begin{equation*}
			\mathsf{X} \approx_b \mathsf{Y} \oplus \mathsf{F} \approx_f \mathsf{Y} \oplus (\mathsf{Y} \oplus \mathsf{F}) \approx_b \mathsf{Y} \oplus \mathsf{X}
			\approx_a (\mathsf{X} \oplus \mathsf{X}) \oplus \mathsf{E} \approx_e \mathsf{X} \oplus \mathsf{E} \approx_a \mathsf{Y}.
		\end{equation*}
		
		\ref{2C} Assume that $\mathsf{X} \approx_h \ell_p(\mathsf{X})$ (again, the case of $c_0$-sums is identical as of $\ell_{\infty}$-sums). Then
		\begin{align*}
			\mathsf{Y} \oplus \mathsf{X}
			& \approx_h \mathsf{Y} \oplus \ell_p(\mathsf{X})
			\approx_b \mathsf{Y} \oplus \ell_p(\mathsf{Y} \oplus \mathsf{F})
			\approx_{\kappa_p} \mathsf{Y} \oplus \ell_p(\mathsf{Y}) \oplus \ell_p(\mathsf{F}) \\
			& \approx_{\kappa_p} \ell_p(\mathsf{Y}) \oplus \ell_p(\mathsf{F})
			\approx_{\kappa_p} \ell_p(\mathsf{Y} \oplus \mathsf{F})
			\approx_b \ell_p(\mathsf{X})
			\approx_h \mathsf{X},
		\end{align*}
		so that $\mathsf{Y} \oplus \mathsf{X} \approx_{\kappa_p^3 b^2 h^2} \mathsf{X}$.
		On the other hand,
		\begin{equation*}
			\mathsf{Y}
			\approx_a \mathsf{X} \oplus \mathsf{E}
			\approx_h \ell_p(\mathsf{X}) \oplus \mathsf{E}
			\approx_{\kappa_p} \mathsf{X} \oplus \ell_p(\mathsf{X}) \oplus \mathsf{E}
			\approx_h \mathsf{X} \oplus \mathsf{X} \oplus \mathsf{E}
			\approx_a \mathsf{X} \oplus \mathsf{Y},
		\end{equation*}
		so that $\mathsf{Y} \approx_{\kappa_p a^2 h^2} \mathsf{X} \oplus \mathsf{Y}$.
		Combining both chains, we get $d_{\mathrm{BM}}(\mathsf{X},\mathsf{Y}) \leqslant \kappa_p^4 a^2 b^2 h^4$.
	\end{proof}
	
	\begin{remark}[About Theorem~\ref{THM : QPDM}] \label{REMARK : QPDM}
		\hfill
		\begin{enumerate}[label={\normalfont{(\arabic*)}}]
			\item If all the decompositions involved are {\bf isometric} $\ell_p$-decompositions, that is,
			$\mathsf{X} \cong \mathsf{Y} \oplus_p \mathsf{F}$, $\mathsf{Y} \cong \mathsf{X} \oplus_p \mathsf{E}$ and $\mathsf{X} \cong \ell_p(\mathsf{X})$,
			then the same chains of identifications consist of isometries, so $\mathsf{X} \cong \mathsf{Y}$.
			This mechanism lies behind Lemma~\ref{LEMMA : 3} below.
			\item The constants in Theorem~\ref{THM : QPDM} are certainly not optimal.
			Sharper estimates in various special situations can be found in \cite{AK26}, \cite{KL26} and \cite{KP25}.
		\end{enumerate}
		\demo
	\end{remark}
	
	\subsection{A few facts about $\mathfrak{L}(\mathsf{X},\mathsf{Y})$}
	
	We record three simple lemmas.
	
	\begin{lemma} \label{LEMMA : 1}
		{\it Let $\mathsf{X}$ be a Banach space.
			Then the space $( \bigoplus_{n=1}^{\infty} \mathsf{X} )_{\ell_{\infty}}$ is isometric to the space $\mathfrak{L}(\ell_1,\mathsf{X})$.}
	\end{lemma}
	\begin{proof}
		For $T \in \mathfrak{L}(\ell_1,\mathsf{X})$, let $\Phi(T) \coloneqq \left\{ T(\boldsymbol{e}_n) \right\}_{n=1}^{\infty}$.
		Plainly, $\Phi \colon \mathfrak{L}(\ell_1,\mathsf{X}) \rightarrow ( \bigoplus_{n=1}^{\infty} \mathsf{X} )_{\ell_{\infty}}$ is a linear bijection.
		For any $x = \sum_{n=1}^{\infty} a_n \boldsymbol{e}_n$ in $\text{Ball}(\ell_1)$, we have
		\begin{equation*}
			\norm{T(x)}_{\mathsf{X}}
			\leqslant \sum_{n=1}^{\infty} \abs{a_n} \norm{ T(\boldsymbol{e}_n) }_{\mathsf{X}}
			\leqslant \norm{x}_{\ell_1} \norm{\Phi(T)}_{( \bigoplus_{n=1}^{\infty} \mathsf{X} )_{\ell_{\infty}}}
			\leqslant \norm{\Phi(T)}_{( \bigoplus_{n=1}^{\infty} \mathsf{X} )_{\ell_{\infty}}},
		\end{equation*}
		so that
		\begin{equation} \label{EQ : w jedna strone}
			\norm{T}_{\mathfrak{L}(\ell_1,\mathsf{X})} \leqslant \norm{ \Phi(T) }_{( \bigoplus_{n=1}^{\infty} \mathsf{X} )_{\ell_{\infty}}}.
		\end{equation}
	
		Conversely, since $\boldsymbol{e}_n \in \text{Ball}(\ell_1)$ for all $n$,
		\begin{equation} \label{EQ : w druga strone}
			\norm{ \Phi(T) }_{( \bigoplus_{n=1}^{\infty} \mathsf{X} )_{\ell_{\infty}}}
			= \sup_{n \in \mathbb{N}} \norm{T(\boldsymbol{e}_n)}_{\mathsf{X}}
			\leqslant \norm{T}_{\mathfrak{L}(\ell_1,\mathsf{X})}.
		\end{equation}
		Combining \eqref{EQ : w jedna strone} with \eqref{EQ : w druga strone} completes the proof.
	\end{proof}
	
	\begin{digression}[About Lemma~\ref{LEMMA : 1}]
		In the very special case where $\mathsf{X}$ is the scalar field $\mathbb{K}$, the above lemma reduces to the well-known
		duality between $\ell_1$ and $\ell_{\infty}$:
		\begin{equation*}
			\ell_{\infty}
			\cong ( \bigoplus_{n=1}^{\infty} \mathbb{K} )_{\ell_{\infty}}
			\cong \mathfrak{L}(\ell_1,\mathbb{K})
			\cong \ell_1^{*}.
		\end{equation*}
		\demo
	\end{digression}
	
	\begin{lemma} \label{LEMMA : 4}
		{\it Let $\mathsf{X}$, $\mathsf{Y}$, $\mathsf{Z}$ and $\mathsf{W}$ be Banach spaces such that
			$\mathsf{Z} \cembed{\lambda} \mathsf{X}$ and $\mathsf{W} \cembed{\mu} \mathsf{Y}$.
			Then $\mathfrak{L}(\mathsf{Z},\mathsf{W}) \cembed{\lambda\mu} \mathfrak{L}(\mathsf{X},\mathsf{Y})$.}
	\end{lemma}
	\begin{proof}[Sketch of the proof]
		If $J_{\mathsf{Z}}, Q_{\mathsf{Z}}$ and $J_{\mathsf{W}}, Q_{\mathsf{W}}$ witness the two complementable embeddings,
		then $T \mapsto Q_{\mathsf{W}} T J_{\mathsf{Z}}$ witness the projection $ \mathfrak{L}(\mathsf{X},\mathsf{Y})\to \mathfrak{L}(\mathsf{Z},\mathsf{W})$
		and $S \mapsto J_W SQ_Z $ witness the embedding.
	\end{proof}
	
	\begin{lemma} \label{LEMMA : 5}
		{\it Let $\mathsf{X}$ and $\mathsf{Y}$ be two Banach spaces.
			Then the space $\mathfrak{L}(\mathsf{X},\mathsf{Y}^*)$ is isometric to the space $\mathfrak{L}(\mathsf{Y},\mathsf{X}^*)$.}
	\end{lemma}
	
	\begin{remark} \label{REMARK : isometries we need}
		Combining Lemmas~\ref{LEMMA : 1} and \ref{LEMMA : 5}, we obtain, for every Banach space $\mathsf{X}$,
		\begin{equation} \label{EQ : L(X,l_infty)}
			\mathfrak{L}(\mathsf{X},\ell_{\infty}) \cong \mathfrak{L}(\ell_1,\mathsf{X}^{*}) \cong \ell_{\infty}(\mathsf{X}^{*}).
		\end{equation}
		\demo
	\end{remark}
	
	We shall also need the following elementary observation.
	
	\begin{lemma} \label{LEMMA : bidual}
		{\it Let $\mathsf{Y}$ and $\mathsf{Z}$ be Banach spaces such that $\mathsf{Y} \cembed{} \mathsf{Z}$ and $\mathsf{Z}$ is complemented in $\mathsf{Z}^{**}$.
			Then $\mathsf{Y}$ is complemented in $\mathsf{Y}^{**}$.
			In particular, every space isomorphic to a dual space is complemented in its bidual.}
	\end{lemma}
	\begin{proof}[Sketch of the proof]
		Let $J \colon \mathsf{Y} \rightarrow \mathsf{Z}$ and $Q \colon \mathsf{Z} \rightarrow \mathsf{Y}$ satisfy $QJ = \mathrm{id}_{\mathsf{Y}}$,
		and let $\pi \colon \mathsf{Z}^{**} \rightarrow \mathsf{Z}$ be a bounded projection.
		Then $\rho \coloneqq Q\pi J^{**} \colon \mathsf{Y}^{**} \rightarrow \mathsf{Y}$ is bounded and, for $y \in \mathsf{Y}$, we have $J^{**}y = Jy \in \mathsf{Z}$,
		so $\rho y = QJy = y$.
		The last part follows from Dixmier's observation that $\mathsf{V}^{*}$ is complemented in $\mathsf{V}^{***}$ for every Banach space $\mathsf{V}$
		(to see this, just note that the restriction of functionals from $\mathsf{V}^{**}$ to $\mathsf{V}$ is a norm one projection).
	\end{proof}
	
	\subsection{On the isomorphic structure of matrix spaces}
	
	\begin{lemma} \label{LEMMA : 3}
		{\it Let $1 \leqslant p,q \leqslant \infty$.
			Then, for any two unbounded sequences $\{a_n\}_{n=1}^{\infty}$ and $\{b_n\}_{n=1}^{\infty}$ of positive integers,
			the spaces $( \bigoplus_{n=1}^{\infty} \ell_q^{a_n} )_{\ell_{p}}$ and $( \bigoplus_{n=1}^{\infty} \ell_q^{b_n} )_{\ell_{p}}$ are isomorphic with}
		\begin{equation*}
			d_{\mathrm{BM}}\left( ( \bigoplus_{n=1}^{\infty} \ell_q^{a_n} )_{\ell_{p}},( \bigoplus_{n=1}^{\infty} \ell_q^{b_n} )_{\ell_{p}} \right) \leqslant 4^{\abs{1/p-1/q}}.
		\end{equation*}
		{\it In particular, the space $( \bigoplus_{n=1}^{\infty} \ell_q^{a_n} )_{\ell_{p}}$ is isomorphic to $( \bigoplus_{n=1}^{\infty} \ell_q^n )_{\ell_{p}}$.}
	\end{lemma}
	\begin{proof}
		We need a little preparation: Let $\{ \mathcal{N}_k \}_{k=1}^{\infty}$ be a partition of $\mathbb{N}$ into infinitely many (pairwise disjoint) infinite sets, and let
		\begin{equation*}
			\mathfrak{X}
			\coloneqq \Bigl( \bigoplus_{k=1}^{\infty} \bigl( \bigoplus_{n \in \mathcal{N}_k} \ell_q^k \bigr)_{\ell_p} \Bigr)_{\ell_p}.
		\end{equation*}
		Moreover, given any unbounded sequence of positive integers $\{d_n\}_{n=1}^{\infty}$, let $$\mathsf{Y} \coloneqq ( \bigoplus_{n=1}^{\infty} \ell_q^{d_n} )_{\ell_{p}}.$$
		The whole argument is based on the following:
		
		\begin{claimletters} \label{CLAIM o odleglosci BM}
			$d_{\mathrm{BM}}(\mathsf{Y}, \mathfrak{X}) \leqslant 2^{\abs{1/p-1/q}}$.
		\end{claimletters}
		\begin{proof}[Proof of Claim~\ref{CLAIM o odleglosci BM}]
			Clearly, up to isometry, $\mathfrak{X}$ can be written as $( \bigoplus_{n=1}^{\infty} \ell_q^{k_n} )_{\ell_{p}}$ for some unbounded sequence of positive integers $\{k_n\}_{n=1}^{\infty}$.
			Even more obvious is that since $\{d_n\}_{n=1}^{\infty}$ is unbounded, so for every $k$, there are infinitely many $n$ with $d_n \geqslant k$.
			Thus, we can find a permutation $\pi\colon\mathbb{N}\to\mathbb{N}$ such that $d_ {\pi(n)}\geqslant k_n$ for $n\in\mathbb{N}$. 
			Keeping all this data in mind, we have
			\begin{equation*}
				\mathsf{Y}
				\cong ( \bigoplus_{n=1}^{\infty} \ell_q^{d_ {\pi(n)}} )_{\ell_p} 
				\cong ( \bigoplus_{n=1}^{\infty} ( \ell_q^{k_n} \oplus_q \ell_q^{d_ {\pi(n)}-k_n} ) )_{\ell_p}
				\approx_c ( \bigoplus_{n=1}^{\infty} ( \ell_q^{k_n} \oplus_p \ell_q^{d_ {\pi(n)}-k_n} ) )_{\ell_p}
				\cong \mathfrak{X} \oplus_p \mathsf{Z},
			\end{equation*}
			where  $\mathsf{Z} \coloneqq ( \bigoplus_{n=1}^{\infty} \ell_q^{d_ {\pi(n)} - k_n} )_{\ell_p}$ and
			\begin{equation*}
				c = d_{\mathrm{BM}}(\ell_p^2,\ell_q^2) \leqslant 2^{\abs{1/p - 1/q}}
			\end{equation*}
			(see Remark~\ref{REMARK : d_BM(l_p^n,l_q^n) <=}).
			However, since any $\ell_p$-sum does not depend, up to isometry, on the order of its components,
			and every space $\ell_p^k$, $k \in \mathbb{N}$, occurs infinitely many times in $\mathfrak{X}$, 
			\begin{equation*}
				\mathsf{Y} \cong_c \mathfrak{X} \oplus_p \mathsf{Z} \cong \mathfrak{X}.
			\end{equation*}
			Our claim follows.
		\end{proof}
		
		Let $\{a_n\}_{n=1}^{\infty}$ and $\{b_n\}_{n=1}^{\infty}$ be two unbounded sequences of positive integers.
		Remembering that $d_{\mathrm{BM}}$ satisfies the multiplicative $\triangle$-inequality and is symmetric, and applying
		Claim~\ref{CLAIM o odleglosci BM} twice, we have
		\begin{equation*}
			d_{\mathrm{BM}} ( ( \bigoplus_{n=1}^{\infty} \ell_q^{a_n} )_{\ell_{p}},( \bigoplus_{n=1}^{\infty} \ell_q^{b_n} )_{\ell_{p}} )
			\leqslant d_{\mathrm{BM}} ( ( \bigoplus_{n=1}^{\infty} \ell_q^{a_n} )_{\ell_{p}},\mathfrak{X} )
			\cdot d_{\mathrm{BM}} (( \bigoplus_{n=1}^{\infty} \ell_q^{b_n} )_{\ell_{p}},\mathfrak{X} )
			\leqslant  4^{\abs{1/p-1/q}}.
		\end{equation*} 
		This proves our result.
	\end{proof}

	\begin{lemma} \label{LEMMA : 2}
		{\it Let $1 \leqslant p \leqslant \infty$.
			Then the space $( \bigoplus_{n=1}^{\infty} \ell_p^n )_{\ell_{\infty}}$
			is isomorphic to the space $( \bigoplus_{n=1}^{\infty} \ell_p )_{\ell_{\infty}}$.
			Further,}
		\begin{equation*}
			d_{\mathrm{BM}}\left( ( \bigoplus_{n=1}^{\infty} \ell_p^n )_{\ell_{\infty}},( \bigoplus_{n=1}^{\infty} \ell_p )_{\ell_{\infty}} \right) \leqslant 81.
		\end{equation*}
	\end{lemma}
	\begin{proof}
		Clearly, both $( \bigoplus_{n=1}^{\infty} \ell_p^n )_{\ell_{\infty}}$
		and $( \bigoplus_{n=1}^{\infty} \ell_p )_{\ell_{\infty}}$ are isometric to their squares
		(with respect to the sup norm).
		Further, using the obvious embedding of $\ell_p^n$ into $\ell_p$, we get
		\begin{equation*}
			( \bigoplus_{n=1}^{\infty} \ell_p^n )_{\ell_{\infty}}\cembed{1} ( \bigoplus_{n=1}^{\infty} \ell_p )_{\ell_{\infty}}.
		\end{equation*}
		Thus, by Theorem~\ref{THM : QPDM}(1), it remains to show that
		\begin{equation} \label{EQ : to chcemy}
			( \bigoplus_{n=1}^{\infty} \ell_p )_{\ell_{\infty}} \cembed{1} ( \bigoplus_{n=1}^{\infty} \ell_p^n )_{\ell_{\infty}}.
		\end{equation}	
		
		First, we will explain that $\ell_p \cembed{1} ( \bigoplus_{n \in \mathcal{N}} \ell_p^n )_{\ell_{\infty}}$ for every infinite set $\mathcal{N} \subset \mathbb{N}$.
		Finding a subspace of $( \bigoplus_{n \in \mathcal{N}} \ell_p^n )_{\ell_{\infty}}$ isometrically isomorphic to $\ell_p$ is easy.
		Consider the map $T \colon \ell_p \rightarrow ( \bigoplus_{n \in \mathcal{N}} \ell_p^n )_{\ell_{\infty}}$ which sends
		$x = \{x_k\}_{k=1}^{\infty} \in \ell_p$ to the sequence of its truncations $\{ \sum_{k=1}^n x_k \boldsymbol{e}_k \}_{n \in \mathcal{N}}$,
		that is, $T \colon \sum_{k=1}^{\infty} x_k \boldsymbol{e}_k \mapsto \{ \sum_{k=1}^n x_k \boldsymbol{e}_k \}_{n \in \mathcal{N}}$.
		Clearly, $T$ is an isometric embedding.
		However, finding a projection onto the range of $T$ requires a little more effort.
		Let $\mathcal{U}$ be a free ultrafilter on $\mathcal{N}$.
		Given a matrix $[x_{n,k}]_{n \in \mathcal{N}, k \in \mathbb{N}} \in ( \bigoplus_{n \in \mathcal{N}} \ell_p^n )_{\ell_{\infty}}$,
		let $u = \{u_k\}_{k\in\mathbb{N}} \in \mathbb{K}^{\mathbb{N}}$ be defined as $u_k \coloneqq \lim_{\mathcal{U}} [x_{n,k}]_{n \in \mathcal{N}, k \in \mathbb{N}}$ for $k\in\mathbb{N}$.
		(In other words, $u_k$ is defined as the limit along $\mathcal{U}$ of the $k^{\mathrm{th}}$ column $\{ x_{n,k} \}_{n \in \mathcal{N}}$
		of the matrix $[x_{n,k}]_{n \in \mathcal{N}, k \in \mathbb{N}}$.
		This definition makes sense as $\{ x_{n,k} \}_{n \in \mathcal{N}}$ is bounded for each $k \in \mathbb{N}$).
		Since
		\begin{equation*}
			\norm{\sum_{k=1}^M u_k \boldsymbol{e}_k}_{\ell_p}
				= \lim_{\mathcal{U}} \norm{\sum_{k=1}^M x_{n,k} \boldsymbol{e}_k}_{\ell_p}
				\leqslant \norm{[x_{n,k}]_{n \in \mathcal{N}, k \in \mathbb{N}}}_{( \bigoplus_{n \in \mathcal{N}} \ell_p^n )_{\ell_{\infty}}}
				\quad (M \in \mathbb{N}),
		\end{equation*}
		so the map $Q \colon [x_{n,k}]_{n \in \mathcal{N}, k \in \mathbb{N}} \mapsto u$ defines a norm one operator into $\ell_p$ with $QT = \mathrm{id}_{\ell_p}$.

		Now, let $\{ \mathcal{N}_k \}_{k=1}^{\infty}$ be a partition of $\mathbb{N}$ into infinitely many infinite sets.
		Then
		\begin{equation*}
			( \bigoplus_{n=1}^{\infty} \ell_p^n )_{\ell_{\infty}} \cong ( \bigoplus_{k=1}^{\infty} ( \bigoplus_{n \in \mathcal{N}_k} \ell_p^n )_{\ell_{\infty}} )_{\ell_{\infty}},
		\end{equation*}
		and applying the previous step in each coordinate $k \in \mathbb{N}$, we get
		\begin{equation*}
			( \bigoplus_{n=1}^{\infty} \ell_p )_{\ell_{\infty}} \cembed{1} ( \bigoplus_{n=1}^{\infty} \ell_p^n )_{\ell_{\infty}}.
		\end{equation*}
		This yields \eqref{EQ : to chcemy} and, consequently, confirms that $( \bigoplus_{n=1}^{\infty} \ell_p^n )_{\ell_{\infty}}$
		and $( \bigoplus_{n=1}^{\infty} \ell_p )_{\ell_{\infty}}$ are indeed isomorphic.
		
		It is therefore enough to explain the quantitative part of the result:
		Using Theorem~\ref{THM : QPDM}(1) with parameters $\lambda = \mu = e = f = 1$, we immediately get an upper bound
		\begin{equation*}
			d_{\mathrm{BM}}(( \bigoplus_{n=1}^{\infty} \ell_p^n )_{\ell_{\infty}},( \bigoplus_{n=1}^{\infty} \ell_p )_{\ell_{\infty}}) \leqslant 3^2 \cdot 3^2 = 81,
		\end{equation*}
		as desired.
	\end{proof}
	
	\section{{\bf Main results}} \label{SECTION : Main results}
	
	This section contains the main results of this work, that is, Theorems~\ref{LEMMA : 6} and \ref{LEMMA : 6infty}.
	
	\subsection{Operators into separable $\mathcal{L}_q$-spaces} \label{SUBSECTION : operators Script L_p spaces}
	
	Recall that a Banach space $\mathsf{X}$ is an {\bf $\mathcal{L}_{p,\lambda}$-space} if every finite-dimensional subspace
	$\mathsf{F}$ of $\mathsf{X}$ is contained in a finite-dimensional subspace $\mathsf{G}$ of $\mathsf{X}$ whose Banach--Mazur distance
	to $\ell_p^{\dim(\mathsf{G})}$ is at most $\lambda$.
	A Banach space $\mathsf{X}$ is an {\bf $\mathcal{L}_p$-space} if it is an $\mathcal{L}_{p,\lambda}$-space for some $\lambda \geqslant 1$.
	The class of $\mathcal{L}_p$-spaces was introduced by Lindenstrauss and Pe{\l}czy{\' n}ski in their remarkable paper \cite{LP68}
	as the class of Banach spaces whose local structure resembles that of $\ell_p$.
	Up to isomorphism, there are infinitely many separable infinite-dimensional $\mathcal{L}_p$-spaces
	(see, for example, \cite{Lin70}); in fact, Johnson and Lindenstrauss \cite{JL80} showed that there are at least continuum many of them.
	
	We begin with an approximation lemma.
	
	\begin{lemma} \label{LEMMA : pieces}
		{\it Let $1 \leqslant p < \infty$, and let $\mathsf{X}$ be a separable $\mathcal{L}_{p,\lambda}$-space.
			Then there is a constant $C \geqslant 1$, an increasing sequence $\mathsf{E}_1 \subset \mathsf{E}_2 \subset \dots$ of finite-dimensional subspaces of $\mathsf{X}$
			with dense union, isomorphisms $u_n \colon \ell_p^{k_n} \rightarrow \mathsf{E}_n$ and projections $R_n$ of $\mathsf{X}$ onto $\mathsf{E}_n$ such that}
		\begin{equation*}
			\norm{u_n}\lVert u_n^{-1} \rVert \leqslant C \quad\text{and}\quad \norm{R_n} \leqslant C \quad (n \in \mathbb{N}).
		\end{equation*}
		{\it Moreover, $R_n \rightarrow \mathrm{id}_{\mathsf{X}}$ in the strong operator topology.}
	\end{lemma}
	\begin{proof}
		By \cite[Theorem~III(c)]{LR69}, there is a constant $\varrho \geqslant 1$ such that every finite-dimensional subspace of $\mathsf{X}$
		is contained in a finite-dimensional subspace $\mathsf{G}$ with
		$d_{\mathrm{BM}}(\mathsf{G}, \ell_p^{\dim(\mathsf{G})}) \leqslant 2\lambda$
		which is the range of a projection of norm at most $\varrho$.
		Let $\{x_n\}_{n=1}^{\infty}$ be dense in $\mathsf{X}$.
		Applying the above inductively to $\mathsf{E}_{n-1} + \mathrm{span}\{x_n\}$ (with $\mathsf{E}_0 \coloneqq \{0\}$), we obtain the subspaces $\mathsf{E}_n$, isomorphisms $u_n$ and projections $R_n$ with $C \coloneqq \max\{2\lambda, \varrho\}$.
		
		For the last part, fix $x \in \mathsf{X}$ and $\varepsilon > 0$, and choose $N$ and $y \in \mathsf{E}_N$ with $\norm{x - y} \leqslant \varepsilon$.
		Since $\mathsf{E}_N \subset \mathsf{E}_n$ for $n \geqslant N$, we have $R_n y = y$, and hence
		\begin{equation*}
			\norm{R_n x - x} \leqslant \norm{R_n(x-y)} + \norm{y - x} \leqslant (C+1)\varepsilon
		\end{equation*}
		for $n \geqslant N$.
	\end{proof}
	
	\begin{theorem}[General variant of Arias and Farmer's result] \label{LEMMA : 6}
		{\it Let $1 \leqslant p \leqslant q < \infty$, let $\mathsf{X}$ be a separable $\mathcal{L}_p$-space, and let $\mathsf{Y}$ be a separable $\mathcal{L}_q$-space.
			Then the following conditions are equivalent:}
		\begin{enumerate}[label={\normalfont{(\arabic*)}}]
			\item\label{AF3} {\it the space $\mathsf{Y}$ is complemented in its bidual $\mathsf{Y}^{**}$.
				\item\label{AF1} {\it the space $\mathfrak{L}(\mathsf{X},\mathsf{Y})$ is isomorphic to $\mathfrak{L}(\ell_p,\ell_q)$;}
				
			}
		\end{enumerate}
		{\it In particular, condition \ref{AF3} always hold whenever $q > 1$, and also when $q = 1$ and $\mathsf{Y}$ is isomorphic to $L_1(\mu)$ for some measure $\mu$.}
	\end{theorem}
	\begin{proof}	{\bf Zero step: Strategy.}
		First denote
		\begin{equation*}
			\mathfrak{W}_{p,q} \coloneqq \bigl( \bigoplus_{n=1}^{\infty} \mathfrak{L}(\ell_p^n,\ell_q^n) \bigr)_{\ell_{\infty}}.
		\end{equation*}
		The implication \ref{AF3} $\Rightarrow$ \ref{AF1} rests on Pe{\l}czy{\' n}ski's decomposition method.
		Assuming \ref{AF3}, we will show that
		\begin{equation} \label{CONDITION 1} \tag{$\spadesuit$}
			\text{the space $\mathfrak{L}(\mathsf{X},\mathsf{Y})$ is isomorphic to a complemented subspace of $\mathfrak{W}_{p,q}$;}
		\end{equation}
		\begin{equation} \label{CONDITION 2} \tag{$\diamondsuit$}
			\text{the space $\mathfrak{W}_{p,q}$ is isometric to a complemented subspace of $\mathfrak{L}(\ell_p,\ell_q)$;}
		\end{equation}
		\begin{equation}\label{CONDITION 3} \tag{$\clubsuit$}
			\text{the space $\mathfrak{L}(\ell_p,\ell_q)$  is isomorphic to a complemented subspace of $\mathfrak{L}(\mathsf{X},\mathsf{Y})$;}
		\end{equation}
		\begin{equation} \label{CONDITION 4} \tag{$\heartsuit$}
			\text{$\mathfrak{L}(\ell_p,\ell_q)$ satisfies Pe{\l}czy{\' n}ski's accordion property;}
		\end{equation}
		Theorem~\ref{THM : QPDM}\ref{2C} then yields \ref{AF1}. 
		The implication \ref{AF1} $\Rightarrow$ \ref{AF3} is proved in the {\it \enquote{Seventh step}}.
		
		{\bf First step: Finite-dimensional pieces.}
		Let $\mathsf{E}_n$, $u_n \colon \ell_p^{k_n} \rightarrow \mathsf{E}_n$ and $R_n$ be as in Lemma~\ref{LEMMA : pieces} for $\mathsf{X}$,
		and let $\mathsf{F}_n$, $v_n \colon \ell_q^{d_n} \rightarrow \mathsf{F}_n$ and $Q_n$ be their counterparts for $\mathsf{Y}$. We put  $C$ as a greater of constants given by constructions form Lemma~\ref{LEMMA : pieces}.
		We may assume that
		\begin{equation*}
			\lVert u_n^{-1} \rVert = \lVert v_n^{-1} \rVert = 1.
		\end{equation*}
		For $n \in \mathbb{N}$, define
		$\Phi_n \colon \mathfrak{L}(\mathsf{X},\mathsf{Y}) \rightarrow \mathfrak{L}(\ell_p^{k_n},\ell_q^{d_n})$ and
		$\Theta_n \colon \mathfrak{L}(\ell_p^{k_n},\ell_q^{d_n}) \rightarrow \mathfrak{L}(\mathsf{X},\mathsf{Y})$
		so that the diagrams
		\[\begin{tikzcd}
			{\mathsf{X}} & {\mathsf{Y}} & {\mathsf{X}} & {\mathsf{Y}} \\
			{\ell_p^{k_n}} & {\ell_q^{d_n}} & {\ell_p^{k_n}} & {\ell_q^{d_n}}
			\arrow["T", from=1-1, to=1-2]
			\arrow["{v_n^{-1}Q_n}"', from=1-2, to=2-2]
			\arrow["{\Theta_n S}", dotted, from=1-3, to=1-4]
			\arrow["{u_n^{-1}R_n}"', from=1-3, to=2-3]
			\arrow["{u_n}", from=2-1, to=1-1]
			\arrow["{\Phi_n T}", dotted, from=2-1, to=2-2]
			\arrow["S", from=2-3, to=2-4]
			\arrow["{v_n}", from=2-4, to=1-4]
		\end{tikzcd}\]
		commute, that is, $\Phi_n(T) \coloneqq v_n^{-1}Q_nTu_n$ and $\Theta_n(S) \coloneqq v_nSu_n^{-1}R_n$.
		
		Since $Q_n v_n = v_n$ and $R_n u_n = u_n$, we have
		\begin{equation} \label{CONDITION 5}
			\Phi_n\Theta_n = \mathrm{id}_{\mathfrak{L}(\ell_p^{k_n},\ell_q^{d_n})}
			\quad\text{and}\quad
			\Theta_n\Phi_n(T) = Q_n T R_n \quad (T \in \mathfrak{L}(\mathsf{X},\mathsf{Y})),
		\end{equation}
		with $\norm{\Phi_n}, \norm{\Theta_n} \leqslant C^2$ for all $n\in\mathbb{N}$.
		Moreover, by Lemma~\ref{LEMMA : pieces},
		\begin{equation} \label{CONDITION 6}
			\norm{\Theta_n\Phi_n(T)x - Tx}_{\mathsf{Y}}
			\leqslant C \norm{T} \norm{R_n x - x}_{\mathsf{X}} + \norm{Q_n Tx - Tx}_{\mathsf{Y}} \rightarrow 0 \quad (x \in \mathsf{X}).
		\end{equation}
		
		{\bf Second step: Proof of \eqref{CONDITION 1}.}
		Let $\pi \colon \mathsf{Y}^{**} \rightarrow \mathsf{Y}$ be a bounded projection (this is where \ref{AF3} is used) and let $\mathcal{U}$ be a free ultrafilter on $\mathbb{N}$.
		Put
		\begin{equation*}
			\mathfrak{W}' \coloneqq \bigl( \bigoplus_{n=1}^{\infty} \mathfrak{L}(\ell_p^{k_n},\ell_q^{d_n}) \bigr)_{\ell_{\infty}}
		\end{equation*}
		and define $\Phi \colon \mathfrak{L}(\mathsf{X},\mathsf{Y}) \rightarrow \mathfrak{W}'$ and $\Theta \colon \mathfrak{W}' \rightarrow \mathfrak{L}(\mathsf{X},\mathsf{Y})$ by
		\begin{equation*}
			\Phi(T) \coloneqq \{ \Phi_n(T) \}_{n=1}^{\infty}
			\quad\text{and}\quad
			\Theta(\{S_n\}_{n=1}^{\infty})x \coloneqq \pi\Bigl( \mathrm{weak}^{*}\text{-}\lim_{\mathcal{U}} \Theta_n(S_n)x \Bigr) \quad (x \in \mathsf{X}).
		\end{equation*}
		The weak$^{*}$-limit exists in $\mathsf{Y}^{**}$, because the sequence $\{\Theta_n(S_n)x\}_{n=1}^{\infty}$ is bounded by
		$C^2 \lVert \{S_n\}_{n=1}^{\infty} \rVert \, \norm{x}$
		and bounded subsets of $\mathsf{Y}^{**}$ are weak$^{*}$-compact.
		Clearly, $\Theta$ is linear and $\norm{\Theta} \leqslant C^2\norm{\pi}$.
		
		By \eqref{CONDITION 6}, $\Theta_n\Phi_n(T)x \rightarrow Tx$ in norm, so
		\begin{equation*}
			\Theta\Phi(T)x = \pi(Tx) = Tx,
		\end{equation*}
		that is, $\Theta\Phi = \mathrm{id}_{\mathfrak{L}(\mathsf{X},\mathsf{Y})}$.
		Hence $\mathfrak{L}(\mathsf{X},\mathsf{Y}) \cembed{} \mathfrak{W}'$.
		
		Next, choose a strictly increasing sequence $\{m_n\}_{n=1}^{\infty}$ with $m_n \geqslant \max\{k_n,d_n\}$.
		Embed $\mathfrak{L}(\ell_p^{k_n},\ell_q^{d_n})$ into the $m_n^{\mathrm{th}}$ block of $\mathfrak{W}_{p,q}$ via $S \mapsto \jmath S \varrho$,
		where $\varrho \colon \ell_p^{m_n} \rightarrow \ell_p^{k_n}$ restricts to the first $k_n$ coordinates and
		$\jmath \colon \ell_q^{d_n} \rightarrow \ell_q^{m_n}$ is the natural inclusion.
		We see that $\mathfrak{W}' \cembed{1} \mathfrak{W}_{p,q}$.
		
		{\bf Third step: Proof of \eqref{CONDITION 2}.}
		To show that $\mathfrak{W}_{p,q} \cembed{1} \mathfrak{L}(\ell_p,\ell_q)$, put
		\begin{equation*}
			A_n \coloneqq \left\{ \frac{n(n-1)}{2}+1, \dots, \frac{n(n-1)}{2}+n \right\} \qquad (n \in \mathbb{N}),
		\end{equation*}
		that is, $A_n$'s are consecutive subsets of the natural numbers with $\mathrm{card} A_n=n$. Let $E_{A_n}x=x\chi_{A_n}$ for $x \in \mathbb{K}^{\mathbb{N}}$ and
		define
		\begin{equation*}
			\Delta T \coloneqq \sum_{n=1}^\infty E_{A_n} T E_{A_n} \qquad (T \in \mathfrak{L}(\ell_p,\ell_q)).
		\end{equation*}
		Plainly, the following hexagon commutes:
		\begin{equation*}
			\begin{tikzcd}
				& \ell_p \arrow[r, "T"]                                     & \ell_q \arrow[rd, "E_{A_n}"] &                             \\
				\ell_p(A_n) \arrow[ru] &                                                           &                        & \ell_q(A_n) \arrow[ld] \\
				& \ell_p \arrow[lu, "E_{A_n}"] \arrow[r, "\Delta T", dotted] & \ell_q                 &                            
			\end{tikzcd}
		\end{equation*}
		Since the sets $A_n$ are pairwise disjoint and $p \leqslant q$, for $x \in \mathrm{Ball}(\ell_p)$, we have
		\begin{equation} \label{EQ : Delta contraction}
			\begin{aligned}
				\norm{\Delta T x}_{\ell_q}
				& = \Bigl( \sum_{n=1}^{\infty} \norm{ E_{A_n} T E_{A_n} x }_{\ell_q}^q \Bigr)^{1/q}
				\leqslant \norm{T} \Bigl( \sum_{n=1}^{\infty} \norm{ E_{A_n} x }_{\ell_p}^q \Bigr)^{1/q} \\
				& \leqslant \norm{T} \Bigl( \sum_{n=1}^{\infty} \norm{ E_{A_n} x }_{\ell_p}^p \Bigr)^{1/p}
				\leqslant \norm{T}
			\end{aligned}
		\end{equation}
		which means that $\Delta$ is a contractive projection on $\mathfrak{L}(\ell_p,\ell_q)$.
		
		Therefore, it remains to argue that $\Delta \mathfrak{L}(\ell_p,\ell_q)$ is isometric to $\mathfrak{W}_{p,q}$.
		To this end, take $\boldsymbol{W} \coloneqq \{ W_n \colon \ell_p^n \rightarrow \ell_q^n\}_{n=1}^{\infty}$ from $\mathfrak{W}_{p,q}$.
		We can represent each of $W_n$'s as a matrix, in the sense that there is $[a_n(i,j)] \in \mathbb{K}^{n \times n}$ such that
		$W_n = \sum_{i=1}^n \sum_{j=1}^n a_n(i,j) \, \boldsymbol{e}_i \otimes \boldsymbol{e}_j$.
		Given $\boldsymbol{W}$ from $\mathfrak{W}_{p,q}$, let us define the block-diagonal operator $\text{diag} \, \boldsymbol{W} \colon \ell_p \rightarrow \ell_q$
		in the following way
		\begin{equation*}
			\text{diag} \, \boldsymbol{W}
			\coloneq \sum_{n=1}^\infty \sum_{i=1}^n \sum_{j=1}^n a_n(i,j) \, \boldsymbol{e}_{\frac{1}{2}n(n+1)+i} \otimes \boldsymbol{e}_{\frac{1}{2}n(n+1)+j}.
		\end{equation*}
		The same computation as in \eqref{EQ : Delta contraction} shows that the mapping $\Upsilon \colon \mathfrak{W}_{p,q} \rightarrow \mathfrak{L}(\ell_p,\ell_q)$ defined as
		$\Upsilon(\boldsymbol{W}) \coloneqq \text{diag} \, \boldsymbol{W}$
		gives the said isometry between $\mathfrak{W}_{p,q}$ and $\Delta \mathfrak{L}(\ell_p,\ell_q)$.		Thus  \eqref{CONDITION 2} follows.
		
		{\bf Fourth step: Proof of \eqref{CONDITION 4}.}
		Every infinite-dimensional $\mathcal{L}_s$-space with $1 \leqslant s < \infty$ contains a complemented subspace isomorphic to $\ell_s$
		(see \cite[Proposition~7.3]{LP68}; {\it cf.} \cite[Theorem~I(iv)]{LR69}).
		Hence $\ell_p \cembed{} \mathsf{X}$ and $\ell_q \cembed{} \mathsf{Y}$, and Lemma~\ref{LEMMA : 4} gives \eqref{CONDITION 4}.
		
		{\bf Fifth step: Proof of \eqref{CONDITION 3}.}
		We now will show that $\ell_\infty(\mathfrak{L}(\ell_p,\ell_q))\cong \mathfrak{L}(\ell_p,\ell_q)$.
		Write $\mathbb{N} = \bigcup_{n=1}^{\infty} \mathcal{N}_n$, where the sets $\mathcal{N}_n$ are infinite and pairwise disjoint,
		and   let $\sigma_n \colon \mathbb{N} \rightarrow \mathcal{N}_n$ be bijections.   
		Given $\boldsymbol{T} = \{T_n\}_{n=1}^{\infty} \in \ell_{\infty}(\mathfrak{L}(\ell_p,\ell_q))$, let $\Sigma\boldsymbol{T} \in \mathfrak{L}(\ell_p,\ell_q)$ be the operator which,
		for each $n$, maps $\ell_p(\mathcal{N}_n)$ into $\ell_q(\mathcal{N}_n)$
		(note that each of $\sigma_n$'s induce an isometry $\ell_s \cong \ell_s(\mathcal{N}_n)$ for $s \in \{p,q\}$).
		Exactly as in \eqref{EQ : Delta contraction} (here again $p \leqslant q$ is used), $\Sigma$ is an isometry of $\ell_{\infty}(\mathfrak{L}(\ell_p,\ell_q))$ into $\mathfrak{L}(\ell_p,\ell_q)$,
		and $$P T \coloneqq \sum_{n=1}^{\infty} E_{\mathcal{N}_n} T E_{\mathcal{N}_n}$$ is a contractive projection of $\mathfrak{L}(\ell_p,\ell_q)$ onto its range.
		Thus
		\begin{equation*}
			\ell_{\infty}(\mathfrak{L}(\ell_p,\ell_q)) \cembed{1} \mathfrak{L}(\ell_p,\ell_q)
		\end{equation*}
		and, obviously,
		\begin{equation*}
			\mathfrak{L}(\ell_p,\ell_q)\cembed{1} \ell_{\infty}(\mathfrak{L}(\ell_p,\ell_q)).
		\end{equation*}
		Since
		\begin{equation*}
			\ell_{\infty}(\mathfrak{L}(\ell_p,\ell_q)) \cong \ell_{\infty}(\ell_{\infty}(\mathfrak{L}(\ell_p,\ell_q))),
		\end{equation*}
		Theorem~\ref{THM : QPDM}(2) gives \eqref{CONDITION 4}.
		
		{\bf Sixth step: \ref{AF3} implies \ref{AF1}.}
		By \eqref{CONDITION 1}, \eqref{CONDITION 2}, \eqref{CONDITION 3} and \eqref{CONDITION 4}, Theorem~\ref{THM : QPDM}\ref{2C} applied to $\mathfrak{L}(\ell_p,\ell_q)$ and $\mathfrak{L}(\mathsf{X},\mathsf{Y})$
		gives $\mathfrak{L}(\mathsf{X},\mathsf{Y}) \approx \mathfrak{L}(\ell_p,\ell_q)$, that is, the condition \ref{AF1} follows.
		
		{\bf Seventh step: \ref{AF1} implies \ref{AF3}.}
		The space $\mathfrak{L}(\ell_p,\ell_q)$ is a dual space:
		By \cite[Theorem~2.9]{Rya02}, $\mathfrak{L}(\ell_p,\ell_q) \cong (\ell_p \widetilde{\otimes}_{\pi} \ell_{q'})^{*}$ for $q > 1$,
		where $1/q + 1/q' = 1$, and $\mathfrak{L}(\ell_p,\ell_1) \cong (\ell_p \widetilde{\otimes}_{\pi} c_0)^{*}$.
		Hence, if \ref{AF1} holds, then $\mathfrak{L}(\mathsf{X},\mathsf{Y})$ is complemented in its bidual by Lemma~\ref{LEMMA : bidual}.
		On the other hand, $\mathsf{Y} \cembed{1} \mathfrak{L}(\mathsf{X},\mathsf{Y})$. To see this, fix $x_0 \in S(\mathsf{X})$ and $x_0^{*} \in S(\mathsf{X}^{*})$ with $x_0^{*}(x_0) = 1$,
		and consider $y \mapsto x_0^{*}\otimes y$ and $T \mapsto Tx_0$.
		The former provides isometry while the latter projection.
		Now, Lemma~\ref{LEMMA : bidual} gives \ref{AF3}.
		
		{\bf The \enquote{in particular} part.}
		For $q > 1$, every $\mathcal{L}_q$-space is isomorphic to a complemented subspace of an $L_q(\mu)$-space \cite{LR69}, hence reflexive, so \ref{AF3} holds trivially.
		The space $\ell_1 = c_0^{*}$ is a dual space, and $L_1(\mu)$ is the range of a norm one projection on its bidual (it is an $L$-summand in $L_1(\mu)^{**}$).
	\end{proof}

	\begin{remark}[The case $p > q$] \label{REMARK : p>q}
		The restriction $p \leqslant q$ in Theorem~\ref{LEMMA : 6} cannot be dropped.
		 We will show that $\mathfrak{L}(L_p[0,1], L_q[0,1])$ is  \textbf{not} isomorphic to $\mathfrak{L}(\ell_p,\ell_q)$ when $1 \leqslant q < p < \infty$ (so that $p > 1$).		
		
		Recall that, due to Khintchine's inequality, Rademacher functions span a complemented subspace  isomorphic to $\ell_2$ in both $L_p[0,1]$ and $L_q[0,1]$ (see \cite[Theorem~2.b.4]{LT79}).
		Thus, using Lemma \ref{LEMMA : 4}, we get that
		\begin{equation*}
			\mathfrak{L}(\ell_2)\cembed{}\mathfrak{L}(L_p[0,1], L_q[0,1]).
		\end{equation*}
		This, in particular, means that $\mathfrak{L}(L_p[0,1], L_q[0,1])$ cannot be separable.
		
		On the other hand, by Pitt's theorem every operator from $\ell_p$ into $\ell_q$ is compact, and since $\ell_q$ has a basis, $\mathfrak{L}(\ell_p,\ell_q)$ is the closure of the finite-rank operators, hence separable.
		(In the proof of Theorem \ref{LEMMA : 6} the argument breaks down because $\mathscr{M}(\ell_p,\ell_q) = \ell_r$ with $1/r = 1/q - 1/p$, and the sequence $\{ \norm{\Phi_n(T)} \}_{n=1}^{\infty}$ in general is only bounded and does not belong to $\ell_r$.)
		\demo
	\end{remark}
	
	\begin{remark}[$\mathcal{L}_1$-spaces which are not complemented in their biduals] \label{REMARK : L1 not complemented in bidual}
		Condition \ref{AF3} in Theorem~\ref{LEMMA : 6} is not automatic for $q = 1$.
		Indeed, if a separable $\mathcal{L}_1$-space $\mathsf{Y}$ is complemented in $\mathsf{Y}^{**}$, then $\mathsf{Y}$ is isomorphic to a complemented subspace of an $L_1(\mu)$-space
		(since the dual $\mathcal{L}_{\infty}$-space $\mathsf{Y}^{*}$ is injective \cite{Lin64}, $\mathsf{Y}^{**}$ is complemented in some $L_1(\mu)$),
		and hence to a subspace of $L_1[0,1]$.
		However, there are separable $\mathcal{L}_1$-spaces which do not embed into $L_1[0,1]$ (see \cite{Bou81} and \cite{JL80}).
		For such a space $\mathsf{Y}$ and any separable $\mathcal{L}_1$-space $\mathsf{X}$, the space $\mathfrak{L}(\mathsf{X},\mathsf{Y})$ is {\bf not} isomorphic to $\mathfrak{L}(\ell_1,\ell_1)$.
		\demo
	\end{remark}
	
	\subsection{Operators from separable $\mathcal{L}_{\infty}$-spaces} \label{SUBSECTION : operators Script L_infty spaces}
	
	Theorem~\ref{LEMMA : 6} does not cover the case where $\mathsf{X}$ or $\mathsf{Y}$ is a separable $\mathcal{L}_\infty$-space.
	In what follows we fill this gap in the case where the target space is $\ell_{\infty}$.
	From a pragmatic point of view, this is not a big restriction, because this case already exceeds our future needs
	(see the proof of Theorem~\ref{THM : suma L1 isomorficzna z miarami}).
	The general picture, however, remains a mystery (see Question~\ref{QUESTION : general targets}).
	
	\begin{theorem} \label{LEMMA : 6infty}
		{\it Let $\mathsf{X}$ be a separable $\mathcal{L}_\infty$-space.
			Then the space $\mathfrak{L}(\mathsf{X}, \ell_\infty) \cong \ell_{\infty}(\mathsf{X}^{*})$ is isomorphic to the space $\mathfrak{L}(c_0,\ell_\infty) \cong \ell_{\infty}(\ell_1)$.}
	\end{theorem}

	\begin{proof}[Proof of Theorem~\ref{LEMMA : 6infty}]
		The argument is quite similar to the proof of Theorem~\ref{LEMMA : 6}, so we will focus primarily
		on the moments in which both arguments differ, only sketching the remaining fragments.
		(We believe that keeping the notation as identical as possible will make it easier for the reader to spot the differences mentioned.)
		
		The isometries in the statement follow from Remark~\ref{REMARK : isometries we need}\eqref{EQ : L(X,l_infty)}.
		Put 
		\begin{equation*}
			\mathfrak{W}_{\infty} \coloneqq \bigl( \bigoplus_{n=1}^{\infty} \mathfrak{L}(\ell_\infty^n,\ell_\infty^n) \bigr)_{\ell_{\infty}}.
		\end{equation*}
		As in the proof of Theorem~\ref{LEMMA : 6}, we show the following chain of complemented embeddings
		\begin{equation} \label{CONDITION 1infty} \tag{$\spadesuit_{\infty}$}
			\mathfrak{L}(c_0,\ell_{\infty}) \cembed{}\mathfrak{L}(\mathsf{X}, \ell_\infty) \cembed{} \mathfrak{W}_{\infty} \cembed{} \mathfrak{L}(c_0,\ell_{\infty}),
		\end{equation}
		and that $\mathfrak{L}(c_0,\ell_{\infty})$ satisfies Pe{\l}czy{\' n}ski's accordion property.
		Invoking Theorem~\ref{THM : QPDM}\ref{2C} finishes the proof.
		
		{\bf First step.}
		Let $\mathsf{X}$ be an $\mathcal{L}_{\infty,\lambda}$-space and let $\{x_n\}_{n=1}^{\infty}$ be dense in $\mathsf{X}$.
		Directly from the definition, we find an increasing sequence $\mathsf{E}_1 \subset \mathsf{E}_2 \subset \dots$ of finite-dimensional subspaces with $x_n \in \mathsf{E}_n$
		and isomorphisms $u_n \colon \ell_{\infty}^{k_n} \rightarrow \mathsf{E}_n$ with $\norm{u_n^{-1}} = 1$ and $\norm{u_n} \leqslant 2\lambda$;
		we may also assume that $k_1 < k_2 < \dots$.
		Since $\ell_{\infty}^{k_n}$ is $1$-injective, $u_n^{-1}$ extends to an operator $\widetilde{u}_n \colon \mathsf{X} \rightarrow \ell_{\infty}^{k_n}$ of norm one,
		and $P_n \coloneqq u_n \widetilde{u}_n$ is a projection of $\mathsf{X}$ onto $\mathsf{E}_n$ with $\norm{P_n} \leqslant 2\lambda$.
		As in Lemma~\ref{LEMMA : pieces}, $P_n \rightarrow \mathrm{id}_{\mathsf{X}}$ in the strong operator topology.
		Let $R_n \colon \ell_\infty \rightarrow \ell_{\infty}^{k_n}$ be the restriction to the first $k_n$ coordinates
		and $\imath_n \colon \ell_\infty^{k_n} \rightarrow \ell_{\infty}$ the corresponding inclusion.
		We define $\Phi_n$ and $\Theta_n$ by the commutative diagrams
		\[\begin{tikzcd}
			{\mathsf{X}} & \ell_\infty & {\mathsf{X}} & \ell_\infty \\
			{\ell_\infty^{k_n}} & {\ell_\infty^{k_n}} & {\ell_\infty^{k_n}} & {\ell_\infty^{k_n}}
			\arrow["{u_n}", from=2-1, to=1-1]
			\arrow["{\Phi_n T}", dashed, from=2-1, to=2-2]
			\arrow["{R_n}", from=1-2, to=2-2]
			\arrow["T", from=1-1, to=1-2]
			\arrow["{\widetilde{u}_n}"', from=1-3, to=2-3]
			\arrow["{\imath_n}"', from=2-4, to=1-4]
			\arrow["S", from=2-3, to=2-4]
			\arrow["{\Theta_n S}", dashed, from=1-3, to=1-4]
		\end{tikzcd}\]
		that is, $\Phi_n(T) \coloneqq R_n T u_n$ and $\Theta_n(S) \coloneqq \imath_n S \widetilde{u}_n$.
		Then
		\begin{equation*}
			\Phi_n\Theta_n = \mathrm{id}_{\mathfrak{L}(\ell_{\infty}^{k_n},\ell_{\infty}^{k_n})}\text{ and } \Theta_n\Phi_n(T) = \imath_n R_n T P_n
		\end{equation*}
		with $\norm{\Theta_n} \leqslant 2\lambda$.
		Moreover, for every $x \in \mathsf{X}$, the sequence $\{\Theta_n\Phi_n(T)x\}_{n=1}^{\infty}$ is bounded and converges to $Tx$ coordinatewise, because
		\begin{equation*}
			\norm{\imath_n R_n T(P_n x - x)}_{\ell_{\infty}} \leqslant \norm{T}\norm{P_n x - x} \rightarrow 0
		\end{equation*}
		and $\imath_n R_n y \rightarrow y$ coordinatewise for every $y \in \ell_{\infty}$.
		Since bounded coordinatewise convergence in $\ell_{\infty}$ is weak$^{*}$-convergence, we get
		\begin{equation} \label{CONDITION 6infty}
			\mathrm{weak}^{*}\text{-}\lim_{n\to\infty} \Theta_n\Phi_n(T)x = Tx \quad (T \in \mathfrak{L}(\mathsf{X}, \ell_\infty),\ x \in \mathsf{X}).
		\end{equation}
		
		{\bf Second step.}
		Since $\ell_{\infty}$ is a dual space, we can repeat the second step of the proof of Theorem~\ref{LEMMA : 6} without projecting from the bidual space (now using \eqref{CONDITION 6infty}).
		The operators
		$$\Phi(T) \coloneqq \{\Phi_n(T)\}_{n=1}^{\infty}\text{ and }
		\Theta(\{S_n\}_{n=1}^{\infty})x \coloneqq \mathrm{weak}^{*}\text{-}\lim_{\mathcal{U}}\Theta_n(S_n)x$$
		satisfy $\Theta\Phi = \mathrm{id}_{\mathfrak{L}(\mathsf{X}, \ell_\infty)}$, and thus
		\begin{equation*}
			\mathfrak{L}(\mathsf{X}, \ell_\infty) \cembed{} \bigl( \bigoplus_{n=1}^{\infty} \mathfrak{L}(\ell_\infty^{k_n},\ell_\infty^{k_n}) \bigr)_{\ell_{\infty}} \cembed{1} \mathfrak{W}_{\infty}.
		\end{equation*}
		The block-diagonal argument from the second step of the proof of Theorem~\ref{LEMMA : 6} (which is even simpler for the sup norm) gives
		\begin{equation*}
			\mathfrak{W}_{\infty} \cembed{1} \mathfrak{L}(c_0,\ell_{\infty}).
		\end{equation*}
		
		{\bf Third step.}
		The dual of an infinite-dimensional $\mathcal{L}_{\infty}$-space is an infinite-dimensional $\mathcal{L}_1$-space \cite[Theorem~III(a)]{LR69},
		so $\mathsf{X}^{*}$ contains a complemented subspace isomorphic to $\ell_1$ \cite[Proposition~7.3]{LP68}.
		Consequently,
		\begin{equation*}
			\mathfrak{L}(c_0,\ell_{\infty}) \cong \mathfrak{L}(\ell_1,\ell_1) \cong \ell_{\infty}(\ell_1) \cembed{} \ell_{\infty}(\mathsf{X}^{*}) \cong \mathfrak{L}(\mathsf{X},\ell_{\infty})
		\end{equation*}
		by Lemmas~\ref{LEMMA : 5} and \ref{LEMMA : 1} and \eqref{EQ : L(X,l_infty)}.
		
		{\bf Fourth step.}
		By Lemmas~\ref{LEMMA : 5} and \ref{LEMMA : 1},
		\begin{equation*}
			\ell_\infty \left(\mathfrak{L}(c_0,\ell_{\infty})\right)
			\cong \ell_\infty \left(\mathfrak{L}(\ell_1,\ell_1)\right)
			\cong \ell_\infty\left(\ell_\infty(\ell_1)\right)
			\cong \ell_\infty(\ell_1)
			\cong \mathfrak{L}(c_0,\ell_{\infty}).
		\end{equation*}
		
		{\bf Fifth step.}
		Theorem~\ref{THM : QPDM}(2), applied to $\mathfrak{L}(c_0,\ell_{\infty})$ and $\mathfrak{L}(\mathsf{X},\ell_{\infty})$, gives $\mathfrak{L}(\mathsf{X}, \ell_\infty) \approx \mathfrak{L}(c_0,\ell_{\infty})$.
	\end{proof}
	\begin{remark}[About the separability assumption in Theorem~\ref{LEMMA : 6infty}] \label{REMARK : duzo bardzo jest scrip L_infty}
		Theorem~\ref{LEMMA : 6infty} breaks down when the separability assumption is dropped.
  
		To see this, consider the simplest possible situation $\mathsf{X} = \ell_{\infty}$.
		We claim that the spaces $\mathfrak{L}(\ell_{\infty},\ell_{\infty})$ and $\mathfrak{L}(c_0,\ell_{\infty})$ are not isomorphic, because they have different cardinalities.
		On the one hand, since every operator $T \in \mathfrak{L}(c_0,\ell_{\infty})$ is determined by its values on a countable
		dense subset of $c_0$,
		\begin{equation*}
			\mathrm{card}\,\mathfrak{L}(c_0,\ell_{\infty})
			\leqslant \mathfrak{c}^{\aleph_0}
			= 2^{\aleph_0 \cdot \aleph_0}
			= \mathfrak{c},
		\end{equation*}
		while on the other hand
		\begin{equation*}
			\mathrm{card}\,\mathfrak{L}(\ell_{\infty},\ell_{\infty})
			\geqslant \mathrm{card}\,\mathfrak{L}(\ell_{\infty},\mathbb{K})
			=  2^{\mathfrak{c}}.
		\end{equation*}
		Our claim follows.
		\demo
	\end{remark}
	\section{{\bf Applications}} \label{SECTION : Applications}
	
	In this section, we first formally introduce the spaces $\ces$ and $\Ces$ (Section~\ref{SUBSECTION : aide}).
	After a brief historical outline touching on duality, we present four simple results (Claims~\ref{CLAIM A}--\ref{CLAIM D}) that will be used later. Then, using Theorem~\ref{LEMMA : 6}, we show that $\ces$ and $\Ces$ are isomorphic
	(Theorem~\ref{COR: Ces and ces isomorphic}), which re-proves \cite[Theorem~5.1]{ALM19},
	and we discuss some natural complements of this result (Remarks~\ref{REMARK : d_BM(k,K)<=} and \ref{REMARK : K^p}).
	Section~\ref{SUBSECTION : Ces_infty isomorphic to Radon measures} begins with a discussion of the density character of
	$\ell_{\infty}$-sums of non-separable $\mathcal{L}_p$-spaces (Example~\ref{EXAMPLE : density of non-separable L_p's})
	and of the space of Borel measures $\mathscr{M}[0,1]$ (Example~\ref{DIGRESSION : M as sum of L1}).
	Then, as an application of Theorem~\ref{LEMMA : 6infty}, we prove Theorem~\ref{THM : suma L1 isomorficzna z miarami},
	which can be seen as an abstract version of \cite[Theorem~5.5]{ALM19} ({\it cf.} Theorem~\ref{COROLLARY : Michau}).
	
	\subsection{Aide-m{\' e}moire} \label{SUBSECTION : aide}
	
	Let us first introduce the Banach space $\Ces = \Ces(\mathbb{I})$ and its sequence counterpart $\ces$.
	They are Banach lattices defined via the (extended) norms
	\begin{equation*}
		\norm{f}_{\Ces(\mathbb{I})} \coloneqq \sup_{x \in \mathbb{I},\, x>0} \, \frac{1}{x} \int_0^x \abs{f(t)}dt \quad (f \in \mathbb{K}^{\mathbb{I}})
	\end{equation*}
	and, respectively,
	\begin{equation*}
		\norm{x}_{\ces} \coloneqq \sup_{n \in \mathbb{N}} \, \frac{1}{n} \sum_{k=1}^n \abs{x_k} \quad (x = \{x_n\}_{n=1}^{\infty} \in \mathbb{K}^{\mathbb{N}}).
	\end{equation*}
	Here and hereafter, $\mathbb{I}$ stands either for the {\bf unit interval} $[0,1]$ or for the {\bf half-line} $[0,\infty)$,
	both endowed with the Lebesgue measure $\lambda$.
	
	As already mentioned in the introduction, the origins of these constructions can be traced back to \cite{KKL48} and \cite{Wie32}.
	For historical reasons only, and contrary to \cite{AM09}, we decided to use the names $\ces$ and $\Ces$;
	in fact, they appeared under these names already in the 1948 paper by Korenblyum, Kre{\u \i}n and Levin \cite{KKL48}
	(see \cite[p.~14]{AM14}, \cite[Exercise~6, p.~31]{BS88}, \cite[Example~1.2]{LZ63} and \cite[pp.~26 and~61]{Wnu09}; {\it cf.} \cite[pp.~330--331]{AM09}).
	Whatever we call them, they have appeared (sometimes implicitly) in many papers;
	see, for example, \cite{AA17}, \cite{ALM19}, \cite{AM09}, \cite{CR16}, \cite{DS07}, \cite{GE98}, \cite{KMS07}, \cite{KKM21}, \cite{KKT22}, \cite{KT17},
	\cite{KT24}, \cite{LM15a}, \cite{Leu15}, \cite{Mas91}, \cite{MS06}, \cite{NP11}, \cite{Si94} and \cite{Wal20}
	(see also \cite{AM14}, \cite{Be96}, \cite{FLM16}, \cite{KMP07}, \cite{ORS08} and the references therein).
	
	In a sense that can be made precise, the constructions dual to $\Ces$ and $\ces$ are the so-called {\bf Tandori spaces} $\widetilde{L}_1$ and $\widetilde{\ell}_1$
	(see \cite{Ta55}; {\it cf.} \cite{LM15a}).
	They are the Banach lattices $\widetilde{L}_1 = \widetilde{L}_1(\mathbb{I})$ and $\widetilde{\ell}_1$ defined via the (extended) norms
	\begin{equation*}
		\norm{f}_{\widetilde{L}_1(\mathbb{I})} \coloneqq \int_{\mathbb{I}} \widetilde{f}(t) dt \quad (f \in \mathbb{K}^{\mathbb{I}})
		\quad\text{and}\quad
		\norm{x}_{\widetilde{\ell}_1} \coloneqq \sum_{n=1}^{\infty} \widetilde{x}_n \quad (x = \{x_n\}_{n=1}^{\infty} \in \mathbb{K}^{\mathbb{N}}),
	\end{equation*}
	respectively.
	Here, $\widetilde{f}(x) \coloneqq \esssup_{t \in \mathbb{I},\, t \geqslant x} \abs{f(t)}$, $x \in \mathbb{I}$, is the so-called
	{\bf decreasing majorant} of $f$, while $\widetilde{x}_n \coloneqq \sup_{k \geqslant n} \abs{x_k}$, $n \in \mathbb{N}$,
	is the decreasing majorant of a sequence $x$.
	Unlike $\widetilde{\ell}_1$, the space $\widetilde{L}_1$ plays only a minor r{\^ o}le in this work; we mention it for completeness
	(this is, of course, not the case in general - consult the papers mentioned above and \cite{Beu49}).
	
	Providentially, the very appearance of $\ces$ and $\Ces$ is inextricably linked to a {\it prijsvraag} posed in 1968 by the Dutch Mathematical Society
	\cite{NAW68}, asking for a description of the K{\" o}the duals of Ces{\` a}ro spaces (see Section~\ref{SUBSECTION : AM's question}).
	Already in the late 1950s, an isometric description of the K{\" o}the dual of weighted $\ces$-spaces
	was implicitly given by Andrzej Alexiewicz in his rather obscure paper \cite{Ale57}
	(see \cite[Theorem~1]{Ale57}; {\it cf.} \cite[p.~62]{Be96} and \cite[Theorem~6]{LM15a}).
	Much better sources of information on this subject are \cite{AM09} and, especially, \cite{LM15a}.
	
	As a preparation for the next sections, let us explicitly mention the following simple results.
	
	\begin{claimletters} \label{CLAIM A}
		{\it The space $\ces$ is not isomorphic to $\ell_{\infty}$.
			More generally, neither $\ces$ nor $\Ces$ is isomorphic to $\mathcal{C}(K)$ for any compact Hausdorff space $K$.}
	\end{claimletters}
	\begin{proof}
		Even though the second part implies the first (as $\ell_{\infty} \cong \mathcal{C}(\beta\mathbb{N})$),
		this implication uses the existence of $\beta\mathbb{N}$, which is available in $\mathsf{ZFC}$ but {\bf not} in $\mathsf{ZF}$ (this distinction will be important in Section~\ref{SUBSECTION : Vath}).
		For this reason, we give separate proofs of both parts.
		
		Both spaces $\ces$ and $\Ces$ contain complemented subspaces isomorphic to $\ell_1$
		(for $\ces$ this follows from Claim~\ref{CLAIM D} below, and for $\Ces$ see \cite[p.~60]{KKM21}).
		Since $\ell_{\infty}$ is prime, that is, every infinite-dimensional complemented subspace of $\ell_{\infty}$
		is isomorphic to $\ell_{\infty}$ \cite{Lin67}, the space $\ces$ is not isomorphic to $\ell_{\infty}$.
		
		Now suppose that $\ell_1$ is isomorphic to a complemented subspace of $\mathcal{C}(K)$.
		Then $\ell_{\infty} = \ell_1^{*}$ is isomorphic to a complemented subspace of $\mathcal{C}(K)^{*} \cong \mathscr{M}(K)$.
		However, $\mathscr{M}(K)$ is weakly sequentially complete (as an abstract $L$-space; see, for example, \cite[Theorem~5.2.9]{AK06}),
		and this property passes to subspaces, whereas $\ell_{\infty}$ contains $c_0$, which is not weakly sequentially complete.
		This contradiction shows that neither $\ces$ nor $\Ces$ is isomorphic to a $\mathcal{C}(K)$-space.
	\end{proof}
	
	\begin{claimletters} \label{CLAIM B}
		{\it The space $\widetilde{\ell}_1$ is not isomorphic to $\ell_1$.}
	\end{claimletters}
	\begin{proof}
		By Alexiewicz's result \cite[Theorem~1]{Ale57} (see also \cite{KMS07} and \cite{LM15a}; {\it cf.} Lemma~\ref{LEMMA : dual to ces_infty is t_1 in ZF+DC}),
		the dual of $\widetilde{\ell}_1$ can be identified with $\ces$,
		while the dual of $\ell_1$ is $\ell_{\infty}$.
		By Claim~\ref{CLAIM A}, $\ell_{\infty}$ is not isomorphic to $\ces$.
		Consequently, $\ell_1$ is not isomorphic to $\widetilde{\ell}_1$.
	\end{proof}
	
	\begin{claimletters} \label{CLAIM C}
		{\it The spaces $\Ces[0,1]$ and $\Ces[0,\infty)$ are isomorphic.
			More precisely, both are $4$-isomorphic to $(\bigoplus_{n=1}^{\infty} L_1[0,1])_{\ell_{\infty}}$.}
	\end{claimletters}

	This is due to various results of varying generality
	(in ascending order, see \cite[Theorem~9]{AM09}, \cite[Proposition~4.2]{ALM19} and \cite[Corollary~4.e.1]{KT24}).
	To the best of our knowledge, the shortest argument is given in \cite{AM09} ({\it cf.} \cite[Theorem~7.2]{AM14});
	we reproduce it (tracking the constants precisely) for the reader's convenience.

	\begin{proof}[Proof of Claim~\ref{CLAIM C}]
		Denote $\Delta_k \coloneqq [2^k,2^{k+1})$ for $k \in \mathbb{Z}$, and $\mathbb{Z}_{-} \coloneqq \{..., -2, -1\}$.
		For $f \in \Ces[0,\infty)$ put
		\begin{equation*}
			a_k(f) \coloneqq \lambda(\Delta_k)^{-1} \int_{\Delta_k} \abs{f(t)}dt \qquad (k \in \mathbb{Z}).
		\end{equation*}
		We have
		\begin{equation*}
			\frac{1}{x} \int_0^x \abs{f(t)}dt \leqslant 2^{-k} \sum_{j \leqslant k} 2^{j} a_j(f) \leqslant 2 \sup_{j \in \mathbb{Z}} a_j(f) \qquad (k \in \mathbb{Z})
		\end{equation*}
		whenever $x \in \Delta_k$, and
		\begin{equation*}
			a_k(f) \leqslant 2^{-k} \int_0^{2^{k+1}} \abs{f(t)}dt \leqslant 2 \norm{f}_{\Ces[0,\infty)} \qquad (k \in \mathbb{Z}).
		\end{equation*}
		Hence
		\begin{equation*}
			\frac{1}{2} \sup_{k \in \mathbb{Z}} a_k(f) \leqslant \norm{f}_{\Ces[0,\infty)} \leqslant 2 \sup_{k \in \mathbb{Z}} a_k(f),
		\end{equation*}
		and, in the same way, $$\frac{1}{2} \sup_{k \in \mathbb{Z}_{-}} a_k(f) \leqslant \norm{f}_{\Ces[0,1]} \leqslant 2 \sup_{k \in \mathbb{Z}_{-}} a_k(f).$$
		The functionals $f \mapsto \sup_k a_k(f)$ are the norms of the spaces $(\bigoplus_{k} L_1(\Delta_k, \lambda(\Delta_k)^{-1}d\lambda))_{\ell_{\infty}}$
		over $k \in \mathbb{Z}$ and $k \in \mathbb{Z}_{-}$, respectively, and each of these spaces is isometric to $(\bigoplus_{n=1}^{\infty} L_1[0,1])_{\ell_{\infty}}$
		(just rescale each $\Delta_k$ to $[0,1]$).
		The result follows.
	\end{proof}
	
	\begin{claimletters} \label{CLAIM D}
		{\it The space $\ces$ is $4$-isomorphic to $(\bigoplus_{n=0}^{\infty} \ell_1^{2^n})_{\ell_{\infty}}$, and hence $16$-isomorphic to $(\bigoplus_{n=1}^{\infty} \ell_1^{n})_{\ell_{\infty}}$.}
	\end{claimletters}
	
	Again, this result has many protagonists.
	We give priority to Theorem~4.1 in Grosse-Erdmann's Lecture Notes \cite{GE98}
	(see also \cite[Corollary~3.5]{ALM19}, \cite[Corollary~4.e.1]{KT24} and \cite[Corollary~3.4]{Wal20}).
	In a roundabout way, it can also be deduced from Lemmas~\ref{LEMMA : dual to ces_infty is t_1 in ZF+DC} and \ref{LEMMA : tandori_1 = direct sum}.
	
	\begin{proof}[Proof of Claim~\ref{CLAIM D}]
		Let $\Delta_n \coloneqq \{2^n, \dots, 2^{n+1}-1\}$ for $n \in \mathbb{N}_0$, so that $\#(\Delta_n) = 2^n$, and put $a_n(x) \coloneqq 2^{-n}\sum_{k \in \Delta_n} \abs{x_k}$.
		Exactly as in the proof of Claim~\ref{CLAIM C}, for $m \in \Delta_n$ we have
		$$\frac{1}{m}\sum_{k=1}^{m}\abs{x_k} \leqslant 2^{-n}\sum_{j=0}^{n} 2^{j}a_j(x) \leqslant 2\sup_{j} a_j(x),$$
		and $$a_n(x) \leqslant 2^{-n}\sum_{k=1}^{2^{n+1}-1}\abs{x_k} \leqslant 2\norm{x}_{\ces}.$$
		Hence $\ces \approx_4 (\bigoplus_{n=0}^{\infty} \ell_1^{2^n})_{\ell_{\infty}}$ (with the rescaled norms $2^{-n}\norm{\cdot}_{\ell_1^{2^n}}$ on the blocks).
		The second part follows from Lemma~\ref{LEMMA : 3} with $p = \infty$ and $q = 1$.
	\end{proof}
	
	\subsection{The spaces $\ces$ and $\Ces$ are isomorphic, still} \label{SUBSECTION : isomorphism Ces_infty and ces_infty}
	
	Given a measure space $(\Omega,\Sigma,\mu)$, let $\mathrm{Fin}(\Sigma)$ denote the collection of all $A \in \Sigma$
	with $\mu(A) < \infty$
	(as usual, we identify sets which differ by a set of measure zero).
	Then $(\mathrm{Fin}(\Sigma),d_{\mu})$, where $d_{\mu}(A,B) \coloneqq \mu(A \ominus B)$
	and $A \ominus B$ is the symmetric difference of $A$ and $B$, is a complete metric space.
	A measure space $(\Omega,\Sigma,\mu)$ is {\bf separable} if the metric space $(\mathrm{Fin}(\Sigma),d_{\mu})$ is separable.
	
	The time has come to put some pieces together. The main result of this section is the following.
	
	\begin{theorem} \label{THM : isomorphism direct sums}
		{\it Let $1 \leqslant p < \infty$, and let $\Omega = (\Omega,\Sigma,\mu)$ be a separable measure space such that $L_p(\Omega)$ is infinite-dimensional.
			Then the three spaces $( \bigoplus_{n=1}^{\infty} \ell_p^n )_{\ell_{\infty}}$, $( \bigoplus_{n=1}^{\infty} \ell_p )_{\ell_{\infty}}$
			and $( \bigoplus_{n=1}^{\infty} L_p(\Omega) )_{\ell_{\infty}}$ are all isomorphic.}
	\end{theorem}
	\begin{proof}
		By Lemmas~\ref{LEMMA : 1} and \ref{LEMMA : 2},
		\begin{equation*}
			( \bigoplus_{n=1}^{\infty} \ell_p^n )_{\ell_{\infty}}
			\approx ( \bigoplus_{n=1}^{\infty} \ell_p )_{\ell_{\infty}}
			\cong \mathfrak{L}(\ell_1,\ell_p)
			\quad\text{and}\quad
			( \bigoplus_{n=1}^{\infty} L_p(\Omega) )_{\ell_{\infty}} \cong \mathfrak{L}(\ell_1,L_p(\Omega)).
		\end{equation*}
		The space $L_p(\Omega)$ is a separable $\mathcal{L}_p$-space (this is where we need the separability of $(\Omega,\Sigma,\mu)$; see \cite[Chapter~I, Theorem~5.5]{BS88}),
		and it is complemented in its bidual (it is reflexive for $p > 1$, and an $L$-summand in its bidual for $p = 1$).
		Therefore, Theorem~\ref{LEMMA : 6}, applied to the pairs $(\ell_1, \ell_p)$ and $(\ell_1, L_p(\Omega))$, gives
		$\mathfrak{L}(\ell_1,\ell_p) \approx \mathfrak{L}(\ell_1,L_p(\Omega))$.
	\end{proof}
	
	With Theorem~\ref{THM : isomorphism direct sums} at hand, we can effortlessly reproduce the main result of \cite{ALM19} (see Theorem~5.1 {\it ibidem}).
	
	\begin{theorem}[Astashkin, Le{\' s}nik and Maligranda, 2019] \label{COR: Ces and ces isomorphic}
		{\it The spaces $\ces$ and $\Ces$ are isomorphic.}
	\end{theorem}
	\begin{proof}
		By Claims~\ref{CLAIM C} and \ref{CLAIM D}, the spaces $\ces$ and $\Ces$ are isomorphic to
		$( \bigoplus_{n=1}^{\infty} \ell_1^n )_{\ell_{\infty}}$ and $( \bigoplus_{n=1}^{\infty} L_1[0,1] )_{\ell_{\infty}}$, respectively.
		It is now enough to apply Theorem~\ref{THM : isomorphism direct sums}.
	\end{proof}
	
	\begin{remark}[Quantitative version of Theorem~\ref{COR: Ces and ces isomorphic}] \label{REMARK : d_BM(k,K)<=}
		We have all the tools to estimate the Banach--Mazur distance between $\ces$ and $\Ces$.
		Let $\mathsf{X} \coloneqq (\bigoplus_{n=1}^{\infty} L_1[0,1])_{\ell_{\infty}}$ and $\mathsf{Y} \coloneqq (\bigoplus_{n=0}^{\infty} \ell_1^{2^n})_{\ell_{\infty}}$.
		\begin{enumerate}
			\item[$\bullet$] By Claims~\ref{CLAIM C} and \ref{CLAIM D}, $\Ces \approx_4 \mathsf{X}$ and $\ces \approx_4 \mathsf{Y}$.
			\item[$\bullet$] Embedding $\ell_1^{2^n}$ into the $n^{\mathrm{th}}$ copy of $L_1[0,1]$ as the span of the characteristic functions of the dyadic intervals of length $2^{-n}$
			(which is the range of a conditional expectation, a norm one projection), we get $\mathsf{Y} \cembed{1} \mathsf{X}$.
			\item[$\bullet$] By \cite[Theorem~3.7]{ALM19} ({\it cf.} \cite{CKL77}), $L_1[0,1] \cembed{1} \mathsf{Y}$.
			Moreover, since $\ell_1^{2^m} \cembed{1} \ell_1^{2^n}$ for $m \leqslant n$ (again via conditional expectations), we have $\mathsf{Y} \cembed{1} (\bigoplus_{n \in \mathcal{N}} \ell_1^{2^n})_{\ell_{\infty}}$
			for every infinite set $\mathcal{N} \subset \mathbb{N}_0$.
			Splitting $\mathbb{N}_0$ into infinitely many infinite sets, we conclude that $\mathsf{X} = \ell_{\infty}(L_1[0,1]) \cembed{1} \mathsf{Y}$.
			\item[$\bullet$] $\mathsf{X} \cong \ell_{\infty}(\mathsf{X})$.
		\end{enumerate}
		Therefore, Theorem~\ref{THM : QPDM}\ref{2C} with $\lambda = \mu = h = \kappa = 1$ gives $d_{\mathrm{BM}}(\mathsf{X},\mathsf{Y}) \leqslant 81$, and hence
		\begin{equation*}
			d_{\mathrm{BM}}(\ces,\Ces) \leqslant 4 \cdot 81 \cdot 4 = 1296.
		\end{equation*}
		This is closely related to the problem of computing $d_{\mathrm{BM}}(\ell_{\infty},L_{\infty})$,
		which still seems to be open despite several recent attacks
		(see \cite{KL26}, \cite{KP25} and the references therein; we return to this in Section~\ref{SUBSECTION : Banach--Mazur distance}).
		\demo
	\end{remark}
	
	\begin{remark}[Extension of Theorem~\ref{COR: Ces and ces isomorphic} by $p$-convexification] \label{REMARK : K^p}
		{\it En route} to Theorem~\ref{COR: Ces and ces isomorphic} we used only the case $p = 1$ of Theorem~\ref{THM : isomorphism direct sums}.
		For a Banach lattice $\mathsf{X}$ and $1 \leqslant p < \infty$, the
		{\bf $p$-convexification} $\mathsf{X}^{(p)}$ of $\mathsf{X}$ is the space $\{f : \abs{f}^p \in \mathsf{X}\}$ endowed
		with the norm $\lVert f \rVert_{\mathsf{X}^{(p)}} \coloneqq \lVert \abs{f}^{p} \rVert^{1/p}_{\mathsf{X}}$.
		With this construction at hand, it is routine to introduce the spaces $\ces_p \coloneqq \ces^{(p)}$ and $\Ces_p \coloneqq \Ces^{(p)}$
		and to show, exactly as above, that $\ces_p$ is isomorphic to $\Ces_p$.
		We leave the easy details to the interested reader.
		
		Both scales $\{\ces_p\}_{1 \leqslant p < \infty}$ and $\{\Ces_p\}_{1 \leqslant p < \infty}$ are not just a formal whim;
		for example, they were used in \cite{Be96} and \cite{AM09} in factorisation results ({\it cf.} Example~1.c.5 and Corollary~7.b.5 in \cite{KT24}).
		Moreover, the behaviour of the Rademacher functions in $\Ces_p$ was investigated in \cite{AM15}.
		\demo
	\end{remark}
	
	\subsection{On $\Ces$ and spaces of regular Borel measures} \label{SUBSECTION : Ces_infty isomorphic to Radon measures}
	
	Recall that the {\bf density character} $\mathfrak{d}(\mathsf{X})$ of a Banach space $\mathsf{X}$ is the least cardinality of a dense subset of $\mathsf{X}$.
	Thus, $\mathsf{X}$ is separable if, and only if, $\mathfrak{d}(\mathsf{X}) = \aleph_0$.
	
	\begin{example}[Density character of $\ell_{\infty}$-sums of non-separable $\mathcal{L}_p$-spaces] \label{EXAMPLE : density of non-separable L_p's}
		Fix a cardinal $\kappa > \mathfrak{c}$ and $1 \leqslant p < \infty$, and consider the space $L_p(\{-1,1\}^{\kappa})$ with the Haar measure
		on the Cantor cube $\{-1,1\}^{\kappa}$.
		Since $\mathfrak{d}(L_p(\{-1,1\}^{\kappa})) = \kappa$, this is a non-separable $\mathcal{L}_p$-space.
		For $\mathfrak{X} = (\bigoplus_{n=1}^{\infty} \mathsf{X}_n)_{\ell_{\infty}}$, where $\{\mathsf{X}_n\}_{n=1}^{\infty}$ is a sequence
		of Banach spaces, it is readily seen that
		\begin{equation*}
			\mathfrak{d}\left( \mathfrak{X} \right)
			= \max \left\{ \mathfrak{c}, \,\, \sup_{n \in \mathbb{N}} \mathfrak{d}(\mathsf{X}_n) \right\}
		\end{equation*}
		(the first term appears because the density character of $\ell_{\infty}$ is $\mathfrak{c}$).
		Consequently, the space $(\bigoplus_{n=1}^{\infty} L_p(\{-1,1\}^{\kappa}) )_{\ell_{\infty}}$, whose density character is $\kappa$,
		is isomorphic neither to $(\bigoplus_{n=1}^{\infty} L_p[0,1] )_{\ell_{\infty}}$ nor to $( \bigoplus_{n=1}^{\infty} \ell_p )_{\ell_{\infty}}$,
		both of which have density character $\mathfrak{c}$.
		\demo
	\end{example}
	
	Let $K$ be a compact Hausdorff space.
	By $\mathscr{M}(K)$ we denote the space of all regular Borel measures on $K$ with finite variation,
	and by $\mathscr{C}(K)$ the Banach space of continuous functions
	on $K$ with the sup norm.
	By the Riesz--Markov--Kakutani representation theorem (see, for example, \cite[Theorem~4.1.1]{AK06}),
	$\mathscr{C}(K)^*$ can be naturally identified with $\mathscr{M}(K)$.
	
	\begin{example}[Density character of $\mathscr{M}(K)$] \label{DIGRESSION : M as sum of L1}
		Let $K$ be a compact Hausdorff space.
		It is known that $\mathscr{M}(K)$ is isometric to $( \bigoplus_{\gamma \in \Gamma} L_1(\mu_{\gamma}) )_{\ell_{1}}$ via
		$\mu \mapsto \{ d\mu/d\mu_{\gamma} \}_{\gamma \in \Gamma}$,
		where $\{\mu_{\gamma}\}_{\gamma \in \Gamma}$ is a maximal family of mutually singular probability measures on $K$
		(the existence of such a family relies on Zorn's lemma;
		see, for example, \cite[Proposition~4.3.8(iii), pp.~84--85]{AK06} and \cite[p.~27]{DLS11}).
		In particular, by \cite[Proposition~2.13]{DLS11},
		\begin{equation*}
			\mathscr{M}[0,1]
			\approx ( \bigoplus_{\gamma \in \mathbb{R}} L_1[0,1] )_{\ell_{1}}
			\approx L_1(\{-1,1\}^{\mathfrak{c}}),
		\end{equation*}
		and the density character of $\mathscr{M}[0,1]$ is $\mathfrak{c}$.
		\demo
	\end{example}
	
	It is known that $\mathscr{C}(K)$ is an $\mathcal{L}_{\infty,\lambda}$-space for every $\lambda > 1$, and thus
	its dual $\mathscr{M}(K)$ is an $\mathcal{L}_{1,\lambda}$-space (see, for example, \cite[p.~51]{Rya02}).
	However, as Example~\ref{DIGRESSION : M as sum of L1} shows, $\mathscr{M}(K)$ need not be separable, and
	Theorem~\ref{LEMMA : 6} tells us nothing about the isomorphism class of $( \bigoplus_{n=1}^{\infty} \mathscr{M}(K) )_{\ell_{\infty}} \cong \mathfrak{L}(\ell_1,\mathscr{M}(K))$.
	In the light of Example~\ref{EXAMPLE : density of non-separable L_p's}, the next result may come as a bit of a surprise.
	
	\begin{theorem}[Isomorphism classes of $\ell_{\infty}$-sums of $\mathscr{M}(K)$'s] \label{THM : suma L1 isomorficzna z miarami}
		{\it Let $K$ be an infinite compact metrizable space, and let $\mu$ be a finite regular Borel measure on $K$ such that $L_1(\mu)$ is infinite-dimensional.
			Then the spaces
			$( \bigoplus_{n=1}^{\infty} L_1(\mu) )_{\ell_{\infty}}$,
			$( \bigoplus_{n=1}^{\infty} \mathscr{M}(K) )_{\ell_{\infty}}$ and $\ces$
			are isomorphic.}
	\end{theorem}
	\begin{proof}
		The space $\mathcal{C}(K)$ is an infinite-dimensional $\mathcal{L}_\infty$-space ({\it cf.} \cite[Section~3.2]{Rya02}),
		which is separable because $K$ is metrizable (see, for example, \cite[Proposition~2.3(i)]{DLS11}).
		By Lemmas~\ref{LEMMA : 1} and \ref{LEMMA : 5} and Theorem~\ref{LEMMA : 6infty},
		\begin{align*}
			( \bigoplus_{n=1}^{\infty} \mathscr{M}(K) )_{\ell_{\infty}}
			& \cong \mathfrak{L}(\ell_1,\mathscr{M}(K))
			\cong \mathfrak{L}(\mathcal{C}(K),\ell_\infty)
			\approx \mathfrak{L}(c_0,\ell_\infty) \\
			& \cong \mathfrak{L}(\ell_1,\ell_1)
			\approx \mathfrak{L}(\ell_1,L_1(\mu))
			\cong (\bigoplus_{n=1}^{\infty} L_1(\mu) )_{\ell_{\infty}},
		\end{align*}
		where the penultimate isomorphism follows from Theorem~\ref{LEMMA : 6}, since
		$L_1(\mu)$ is a separable $\mathcal{L}_1$-space (see, for example, \cite[Proposition~2.21]{Rya02}) complemented in its bidual.
		Finally,
		\begin{equation*}
			\mathfrak{L}(\ell_1,\ell_1) \cong \ell_{\infty}(\ell_1) \approx \ces
		\end{equation*}
		by Lemmas~\ref{LEMMA : 1} and \ref{LEMMA : 2}, and Claim~\ref{CLAIM D}.
	\end{proof}
	
	The above result immediately gives the following (see \cite[Theorem~5.5]{ALM19}).
	
	\begin{theorem}[Astashkin, Le{\' s}nik and Maligranda, 2019] \label{COROLLARY : Michau}
		{\it The space $\Ces$ is isomorphic to $( \bigoplus_{n=1}^{\infty} \mathscr{M}[0,1] )_{\ell_{\infty}}$.}
	\end{theorem}
	\begin{proof}
		Recall that $\Ces \approx ( \bigoplus_{n=1}^{\infty} L_1[0,1] )_{\ell_{\infty}}$ (Claim~\ref{CLAIM C})
		and call Theorem~\ref{THM : suma L1 isomorficzna z miarami} on stage.
	\end{proof}
	
	\section{{\bf Complementary material}} \label{SECTION : complementary material}
	
	The results collected in this section complement those of Section~\ref{SECTION : Applications}.
	They exemplify the paradigm outlined in the introduction, that is, the search for analogies between the pairs $(\ces,\Ces)$ and $(\ell_{\infty},L_{\infty})$.
	First, in Section~\ref{SUBSECTION : Ces_infty and ces_infty are NOT isometric}, we show that the spaces $\ces$ and $\Ces$, although
	isomorphic (Theorem~\ref{COR: Ces and ces isomorphic}), are not isometric (Theorem~\ref{THEOREM : ces and Ces NOT isometric}).
	In Section~\ref{SUBSECTION : ces_infty preduals}, we show that $\ces$ has a unique isometric predual (Theorem~\ref{THEOREM : ces_infty has UNIQUE predual})
	and at least continuum many pairwise non-isomorphic preduals (Theorem~\ref{PROP : ces has uncountably many preduals} and Corollary~\ref{COROLLARY : c many preduals}).
	Section~\ref{SUBSECTION : Vath}, which concludes the main body of the paper, is a kind of culmination of the ideas contained therein.
	A number of independently interesting results (Lemmas~\ref{LEMMA : dual to ces_infty is t_1 in ZF+DC}, \ref{LEMMA : tandori_1 = direct sum} and \ref{PROPOSITION: Schur Tandori l1 ZF})
	lead us to the conclusion that in some models of set theory without the Axiom of Choice (for example, in Solovay's model) the spaces
	$\ces$ and $\Ces$ are no longer isomorphic (Theorem~\ref{THM : In ZF ces_infty and Ces_infty are NOT isomorphic}).
	
	\subsection{The spaces $\ces$ and $\Ces$ are not isometric} \label{SUBSECTION : Ces_infty and ces_infty are NOT isometric}
	
	As already mentioned, Pe{\l}czy{\' n}ski's celebrated result shows that the spaces $\ell_{\infty}$ and $L_{\infty}[0,1]$
	are isomorphic (see \cite[Proposition~4]{Pel60}; {\it cf.} \cite[Proposition~4.3.10]{AK06}).
	They are, however, {\it not} isometric.
	Indeed, $\ell_{\infty}$ is isometric to $\mathscr{C}(\beta\mathbb{N})$, where $\beta\mathbb{N}$
	is the {\v C}ech--Stone compactification of $\mathbb{N}$, while $L_{\infty}[0,1]$ is isometric to $\mathscr{C}(\mathfrak{L})$,
	where $\mathfrak{L}$ is the Stone space of the measure algebra of $[0,1]$
	(see, for example, Johnstone's monograph \cite{Joh82}).
	By the Banach--Stone theorem (see \cite[Problem~4.2, p.~98]{AK06}), a putative isometry between $\ell_{\infty}$ and $L_{\infty}[0,1]$
	would induce a homeomorphism between $\beta\mathbb{N}$ and $\mathfrak{L}$, which is impossible, because $\beta\mathbb{N}$ is separable, while $\mathfrak{L}$ is not.
	
	All this is folklore.
	Following the {\it modus operandi} of this section, we now show that the spaces $\ces$ and $\Ces$
	are not isometric either.
	Here, we work with real scalars.
	
	Recall that the norm of a Banach space $\mathsf{X}$ is {\bf Fr{\' e}chet differentiable} at $x_0 \in S(\mathsf{X})$ if there is $\varphi \in \mathsf{X}^{*}$ such that
	\begin{equation*}
		\lim_{h \rightarrow 0} \, \frac{\norm{x_0 + h} - \norm{x_0} - \varphi(h)}{\norm{h}} = 0.
	\end{equation*}
	Since the set of points of Fr{\' e}chet differentiability of the norm is preserved by isometries, Theorem~\ref{THEOREM : ces and Ces NOT isometric} below follows from the next two lemmas.
	
	The idea of looking at differentiability of norms comes from Exercises~8.30 and 8.31 on p.~268 of \cite{FHHMSPZ01}.
	
	As before, we write $\mathscr{H}$ for the {\bf Hardy operator}, which,
	depending on the context, is defined either as $\mathscr{H}x(n) \coloneqq \frac{1}{n} \sum_{i=1}^n x_i$ for sequences $x = \{x_n\}_{n=1}^{\infty} \in \mathbb{K}^{\mathbb{N}}$
	and $n\in\mathbb{N}$ or as $\mathscr{H}f(t) \coloneqq \frac{1}{t} \int_0^t f(s) ds$ for measurable functions $f$ and $t > 0$.
	
	\begin{lemma} \label{LEMMA : Frechet points of k}
		{\it 
			There is a dense set of Fr{\' e}chet differentiable points in $S(\ces)$.}
	\end{lemma}
	\begin{proof}
		Let us start with defining the desired set. For $i \in \mathbb{N}$ and $\delta > 0$, denote by $A(i, \delta)$ the collection
		of all positive functions $x\in\ces$ such that $\mathscr{H} x (i) = 1$ and $\mathscr{H} x (n) < 1 - \delta$ for all $n \neq i$.
		It is straightforward to see, that the set $\bigcup_{i \in \mathbb{N}} \bigcup_{\delta > 0} \pm A(i,\delta)$ is dense in $\text{Ball}(\ces)$.
		
		We will show that all elements from each of $A(i, \delta)$ are Fr{\' e}chet differentiable.  To do this, fix $i_0 \in \mathbb{N}$, $\varepsilon > 0$ and $\delta > 0$.
		Take $x \in A(i,\delta)$ and $h \in \ces$ with $\varepsilon < \norm{h}_{\ces_{\infty}} < \delta/2$.
		We claim that
		\begin{equation} \tag{$\spadesuit$} \label{EQ : addytywnosc na H}
			\norm{x + h}_{\ces} = \mathscr{H}x(i_0) + \mathscr{H}h(i_0).
		\end{equation}
		Indeed, using the reverse $\triangle$-inequality,
		\begin{equation*}
			\mathscr{H}\abs{x + h}(i_0) \geqslant \mathscr{H}(\abs{x} + \abs{h})(i_0) \geqslant 1 - \frac{\delta}{2}.
		\end{equation*}
		Moreover, for $n \neq i_0$,
		\begin{equation*}
			\mathscr{H}\abs{x + h}(n)
			\leqslant \mathscr{H}\abs{x}(n) + \mathscr{H}\abs{h}(n)
			\leqslant 1 - \delta + \frac{\delta}{2}
			= 1 - \frac{\delta}{2}.
		\end{equation*}
		Thus, in view of the above inequalities, we infer that
		\begin{equation*}
			\norm{x + h}_{\ces} = \mathscr{H}\abs{x + h}(n) = \mathscr{H}(x + h)(i_0) = \mathscr{H}x(i_0) + \mathscr{H}h(i_0).
		\end{equation*}
		In consequence, our claim \eqref{EQ : addytywnosc na H} follows.
		Now, we have
		\begin{align*}
			\frac{\abs{\norm{x + h}_{\ces} - \norm{x}_{\ces} - (\boldsymbol{e}^*_i \circ \mathscr{H})(h)}}{\norm{h}_{\ces}}
			& \leqslant \frac{\mathscr{H}x(i_0) + \mathscr{H}h(i_0) - \mathscr{H}x(i_0) - \mathscr{H}h(i_0)}{\varepsilon} = 0.
		\end{align*}
		This means that Fr{\' e}chet derivative of the norm function $\norm{\cdot}_{\ces} \colon \ces \rightarrow \mathbb{R}$
		at a point $x$ is equal to $\boldsymbol{e}^*_i \circ \mathscr{H}$.

	\end{proof}
	
	\begin{lemma} \label{LEMMA : K nowhere Frechet}
		{\it The norm of $\Ces = \Ces(\mathbb{I})$ is nowhere Fr{\' e}chet differentiable.}
	\end{lemma}
	\begin{proof}
		Take a positive, norm one function $f$ from $\Ces$.
		Our plan is to show that $f$ has at least two different supporting hyperplanes.
		To this end, let
		\begin{equation*}
			\Omega_n \coloneq \left\{ t > 0 : \frac{1}{n+1} < \abs{\mathscr{H}f(t) - 1} \leqslant \frac{1}{n} \right\} \quad (n \in \mathbb{N}),
		\end{equation*}
		with the proviso that if it happened that one of the $\Omega_n$'s has infinite measure, then instead of $\Omega_n$ we take any subset of it with measure one.
		Further, let us consider the functional $\mathcal{A}_{n} \in \Ces^{*}$ defined as
		\begin{equation*}
			\mathcal{A}_{n}(g) \coloneq \frac{1}{\lambda(\Omega_n)} \int_{\Omega_n} \mathscr{H}g(t)dt \quad (g \in \Ces, n \in \mathbb{N}).
		\end{equation*}
		Note that all $\mathcal{A}_{n}$'s are well-defined because we integrate bounded functions over sets of finite measure.
		It is clear that $\mathcal{A}_{n} \in \mathrm{Ball}(\Ces^{*})$ for every $n \in \mathbb{N}$.
		Next, due to Banach--Alaoglu's theorem, $\mathrm{Ball}(\Ces^{*})$ is compact in the weak$^*$ topology,
		so both sequences $\{\mathcal{A}_{2n}\}_{n=1}^{\infty}$ and $\{\mathcal{A}_{2n+1}\}_{n=1}^{\infty}$ have weak$^*$-cluster points,
		say $\mathcal{A}_{\mathrm{even}}$ and $\mathcal{A}_{\mathrm{odd}}$, respectively.
		Moreover, since $\lim_{n \rightarrow \infty} \mathcal{A}_{2n}(f) = 1$ and $\lim_{n \rightarrow \infty} \mathcal{A}_{2n+1}(f) = 1$, so
		\begin{equation*}
			\mathcal{A}_{\mathrm{even}}(f) = \mathcal{A}_{\mathrm{odd}}(f) = 1.
		\end{equation*}
		Thus, it is enough to show that the supporting hyperplanes generated by $\mathcal{A}_{\mathrm{even}}$ and $\mathcal{A}_{\mathrm{odd}}$ are distinct.
		In fact, by putting $\Omega \coloneq \bigcup_{n \in \mathbb{N}} \Omega_{2n}$, we see that $\mathcal{A}_{\mathrm{even}}(\chi_{\Omega}) = 1$,
		but $\mathcal{A}_{\mathrm{odd}}(\chi_{\Omega}) = 0$ (just note that the characteristic function $\chi_{\Omega}$ of $\Omega$
		belongs to $L_{\infty}$, and therefore even more so to $\Ces$).
		This observation finishes the proof.
	\end{proof}

	\begin{theorem} \label{THEOREM : ces and Ces NOT isometric}
		{\it The spaces $\ces$ and $\Ces$ are not isometric. }
	\end{theorem}
	\begin{proof}
		By Lemma~\ref{LEMMA : Frechet points of k}, the norm of $\ces$ is Fr{\' e}chet differentiable at some points of $S(\ces)$
		(even more, $\norm{\cdot}_{\ces}$ is Fr{\' e}chet differentiable on a dense subset of $S(\ces)$),
		while by Lemma~\ref{LEMMA : K nowhere Frechet} the norm of $\Ces$ is nowhere Fr{\' e}chet differentiable.
		The proof works in the same way for $\mathbb{I} = [0,1]$ and $\mathbb{I}=[0,\infty)$
		(however, see Question~\ref{QUESTION : Are K[0,1] and K[0,infty) ISOMETRIC?}).
	\end{proof}
	
	\subsection{Briefly about preduals of $\ces$} \label{SUBSECTION : ces_infty preduals}
	
	A Banach space $\mathsf{Y}$ is an {\bf isometric predual} of a Banach space $\mathsf{X}$ if $\mathsf{Y}^{*}$
	is isometric to $\mathsf{X}$, and $\mathsf{X}$ has a {\bf unique predual} if it has an isometric predual and all its isometric preduals
	are mutually isometric.
	
	In general, an isometric predual is not unique;
	for example, both $c_0$ and $c$ are isometric preduals of $\ell_1$, but they are not\footnote{To be pedantically precise,
	in the late 1960s Michael Cambern \cite{Cam68} showed that the Banach--Mazur distance between $c_0$ and $c$ is exactly $3$.} isometric.
	On the other hand, in the mid-1950s Alexander Grothendieck \cite[Remark~4]{Gro55} ({\it cf.} \cite[pp.~59--60]{BI80})
	noted that the space $\ell_{\infty}$ has a unique predual, namely $\ell_1$.
	As the next result shows, a similar situation occurs for $\ces$.
	
	\begin{theorem} \label{THEOREM : ces_infty has UNIQUE predual}
		{\it The space $\ces$ has a unique isometric predual, namely the space $\widetilde{\ell}_1$.}
	\end{theorem}
	\begin{proof}
		By a theorem of Andrzej Alexiewicz \cite{Ale57}, the natural isometric predual of $\ces$ can be identified with $\widetilde{\ell}_1$
		(see Theorem~3 {\it ibidem} or Lemma~\ref{LEMMA : dual to ces_infty is t_1 in ZF+DC} below; see also \cite{LM15a} for some generalisations).
		By Lemma~\ref{LEMMA : Frechet points of k}, the norm of $\ces$ is Fr{\' e}chet differentiable on a dense subset of $S(\ces)$.
		By Theorem~II.1 together with Example~II in Godefroy's survey \cite{God89}, $\ces$ has a unique predual.
	\end{proof}
	
	It seems natural --- although somewhat ambitious --- to estimate the number of isomorphic preduals of $\ces$.
	Two important ingredients of this endeavour are the class of Lindenstrauss spaces and the Szlenk index,
	which we now briefly recall.
	
	A {\bf Lindenstrauss space} is a Banach space whose dual is isometric to $L_1(\mu)$ for some measure $\mu$.
	In particular, a Banach space $\mathsf{X}$ is an {\bf isometric $\ell_1$-predual} if $\mathsf{X}^{*}$ is isometric to $\ell_1$.
	The space $\ell_1$ admits a multitude of non-isomorphic isometric preduals
	(see, for example, \cite{AH11}).
	The class of Lindenstrauss spaces coincides with the class of $\mathscr{L}_{\infty,\lambda}$-spaces for all $\lambda > 1$
	(see \cite{LL66} and \cite{Lin64}).
	For a comprehensive study of $L_1$-preduals we refer to Lacey's monograph \cite[Chapter~7]{EL74} and the references therein.
	
	Without going into details, the {\bf Szlenk index} $\textbf{Sz}$ is an ordinal-valued isomorphic invariant of Banach spaces
	introduced in 1968 by Wies{\l}aw Szlenk \cite{Szl68} to show that the class of separable reflexive Banach spaces has
	no universal element.
	(Parenthetically, Szlenk's discovery solved an old problem posed by Stefan Banach and Stanis{\l}aw Mazur in the Scottish Book.)
	For our modest needs, it is enough to know that $\textbf{Sz}$ describes the isomorphic classification of
	$\mathcal{C}(K)$-spaces for countable compact spaces $K$\footnote{Recall that, by the classical result of Mazurkiewicz and Sierpi{\' n}ski
		\cite{MS20}, every countable compact Hausdorff space is homeomorphic to $\alpha + 1$ for some $\alpha < \omega_1$.}
	(see Brooker's paper \cite{Bro13}, where the Szlenk index of $\mathcal{C}(\alpha + 1)$ is computed for every ordinal $\alpha \geqslant \omega$;
	{\it cf.} \cite{Sam83} and \cite{HL07}).
	The Szlenk index has much more to offer in the geometry of Banach spaces;
	see Lancien's survey \cite{Lan06}, Odell's beautiful paper \cite{Ode04}, and also \cite{Bro13} and \cite{DK16}.
	
	\begin{theorem} \label{PROP : ces has uncountably many preduals}
		{\it The space $\ces$ has at least $\aleph_1$ mutually non-isomorphic preduals.}
	\end{theorem}
	\begin{proof}
		We first show that every space of the form $( \bigoplus_{n=1}^{\infty} \boldsymbol{L} )_{\ell_{1}}$,
		where $\boldsymbol{L}$ is an isometric $\ell_1$-predual, is a predual of $\ces$.
		Recall the following \enquote{tensor-hom adjunction}:
		For Banach spaces $\mathsf{X}$ and $\mathsf{Y}$ one has $(\mathsf{X} \widetilde{\otimes}_{\pi} \mathsf{Y})^* \cong \mathfrak{L}(\mathsf{X},\mathsf{Y}^*)$
		(see, for example, \cite[Theorem~2.9, p.~22]{Rya02}).
		Using Lemmas~\ref{LEMMA : 1}, \ref{LEMMA : 3} and \ref{LEMMA : 2} together with Claim~\ref{CLAIM D}, we get
		\begin{equation*}
			\ces
			\approx (\bigoplus_{n=1}^{\infty} \ell_1^n)_{\ell_{\infty}}
			\approx \ell_{\infty}(\ell_1)
			\cong \mathfrak{L}(\ell_1,\ell_1)
			\cong \mathfrak{L}(\ell_1,\boldsymbol{L}^{*})
			\cong (\ell_1 \widetilde{\otimes}_{\pi} \boldsymbol{L})^*.
		\end{equation*}
		Moreover, $\ell_1 \widetilde{\otimes}_{\pi} \boldsymbol{L} \cong ( \bigoplus_{n=1}^{\infty} \boldsymbol{L} )_{\ell_{1}}$
		(see \cite[Example~2.6, p.~19]{Rya02}).
		In other words, $\ces \approx ( \bigoplus_{n=1}^{\infty} \boldsymbol{L} )_{\ell_{1}}^{*}$, as claimed.
		
		It remains to show that the family of spaces $(\bigoplus_{n=1}^{\infty} \boldsymbol{L} )_{\ell_{1}}$, where $\boldsymbol{L}$
		runs through the isometric $\ell_1$-preduals, contains at least $\aleph_1$ isomorphism classes.
		For every countable ordinal $\alpha$ endowed with the order topology, $\mathcal{C}(\alpha+1)^* \cong \ell_1$.
		By Bessaga and Pe{\l}czy{\' n}ski \cite[Theorem~1]{BP60}, for $0 \leqslant \alpha,\beta < \omega_1$ the spaces $\mathcal{C}(\omega^{{\omega}^{\alpha}}+1)$
		and $\mathcal{C}(\omega^{{\omega}^{\beta}}+1)$ are isomorphic if, and only if, $\alpha = \beta$.
		For $0 \leqslant \gamma < \omega_1$, let $\mathfrak{X}_{\gamma} \coloneqq ( \bigoplus_{n=1}^{\infty} \mathcal{C}(\omega^{{\omega}^{\gamma}}+1) )_{\ell_{1}}$,
		and fix $0 \leqslant \alpha < \beta < \omega_1$.
		Suppose, towards a contradiction, that $\mathfrak{X}_{\alpha}$ and $\mathfrak{X}_{\beta}$ are isomorphic.
		Since $\mathfrak{X}_{\beta}$ has a subspace isomorphic to $\mathcal{C}(\omega^{{\omega}^{\beta}}+1)$,
		so does $\mathfrak{X}_{\alpha}$.
		Since $\mathcal{C}(\omega^{{\omega}^{\beta}}+1)$ has no\footnote{Otherwise, $\mathcal{C}(\omega^{{\omega}^{\beta}}+1)^{*} \cong \ell_1$
			would have a quotient isomorphic to $\ell_{\infty} = \ell_1^{*}$, which is impossible, because $\ell_1$ is separable.} subspace isomorphic to $\ell_1$, Cembranos and Mendoza's result \cite[Corollary~2.2]{CM08}
		shows that $\mathcal{C}(\omega^{{\omega}^{\alpha}}+1)$ contains a subspace isomorphic to $\mathcal{C}(\omega^{{\omega}^{\beta}}+1)$.
		Thus
		\begin{equation*}
			\textbf{Sz}( \mathcal{C}(\omega^{{\omega}^{\beta}}+1) ) \leqslant \textbf{Sz}( \mathcal{C}(\omega^{{\omega}^{\alpha}}+1) ).
		\end{equation*}
		On the other hand, by Samuel's result \cite{Sam83}, the Szlenk index of $\mathcal{C}(\omega^{{\omega}^{\delta}}+1)$
		with $0 \leqslant \delta < \omega_1$ equals $\omega^{\delta +1}$
		(see \cite[Theorem~3.2]{HL07} for a direct proof, and \cite[Theorem~2.6]{Bro13} for some generalisations).
		Therefore,
		\begin{equation*}
			\textbf{Sz}( \mathcal{C}(\omega^{{\omega}^{\beta}}+1) ) = \omega^{\beta + 1} > \omega^{\alpha + 1} = \textbf{Sz}( \mathcal{C}(\omega^{{\omega}^{\alpha}}+1) ),
		\end{equation*}
		a contradiction.
	\end{proof}
	
	The above result can be improved with a slightly heavier arsenal.
	
	\begin{corollary} \label{COROLLARY : c many preduals}
		{\it The space $\ces$ has at least $\mathfrak{c}$ mutually non-isomorphic preduals.}
	\end{corollary}
	\begin{proof}[Sketch of the proof]
		Bourgain and Delbaen \cite[Corollary~5.11]{BD80} constructed continuum\footnote{Let us note,
			{\it en passant}, that to the best of our knowledge the question whether there is a continuum of non-isomorphic isometric
			$\ell_1$-preduals is open.} many mutually non-isomorphic $\ell_1$-preduals;
		in the notation of \cite{BD80}, we denote them by $\mathsf{Y}_{a,b}$, where $0 < a,b < 1$ are real numbers satisfying
		certain technical conditions (see \cite[p.~166]{BD80}).
		Fix an admissible $a$ and consider the family
		\begin{equation*}
			\mathcal{Y} = \left\{ ( \bigoplus_{n=1}^{\infty} \mathsf{Y}_{a,b} )_{\ell_{1}} : b \text{ admissible} \right\}.
		\end{equation*}
		Since each $\mathsf{Y}_{a,b}$ is a separable $\mathcal{L}_{\infty}$-space with a separable dual, Lewis and Stegall's result
		\cite[Theorem~2]{LS73} shows that $\mathsf{Y}_{a,b}^{*} \approx \ell_1$, and hence, by Lemma~\ref{LEMMA : 1} and Claim~\ref{CLAIM D},
		\begin{equation*}
			( \bigoplus_{n=1}^{\infty} \mathsf{Y}_{a,b} )_{\ell_{1}}^{*} \cong \ell_{\infty}(\mathsf{Y}_{a,b}^{*}) \approx \ell_{\infty}(\ell_1) \approx \ces.
		\end{equation*}
		In particular, no $\mathsf{Y}_{a,b}$ contains a subspace isomorphic to $\ell_1$.
		The rest of the argument is the same as in the proof of Theorem~\ref{PROP : ces has uncountably many preduals}, with the Szlenk index replaced by
		the fact that $\mathsf{Y}_{a,c}$ does not embed into $\mathsf{Y}_{a,b}$ for $c > b$ \cite[Theorem~5.10]{BD80}.
	\end{proof}
	
	\subsection{Space oddities, after Martin V{\" a}th} \label{SUBSECTION : Vath}
	
	Recall that the {\bf Axiom of Dependent Choice} (briefly, $\mathsf{DC}$) asserts
	that if $R$ is a total relation on a non-empty set $S$ (that is, for every $x \in S$ there exists $y \in S$
	with $x R y$), then there exists a sequence $\{x_n\}_{n=1}^{\infty}$ in $S$ such that\footnote{Hence the name,
		as the available choices of $x_{n+1}$ depend on the previous choice of $x_n$.} $x_n R x_{n+1}$ for all $n \in \mathbb{N}$.
	Alternatively, one can define $\mathsf{DC}$ in terms of strategies for infinite games (see, for example, Sections~567A--567C in \cite{Fre15}).
	
	It is well known that $\mathsf{DC}$, although weaker than $\mathsf{AC}$, is stronger than the Axiom of Countable Choice $\mathsf{AC}_{\omega}$
	(roughly speaking, because $\mathsf{DC}$ allows one to choose elements recursively; {\it cf.} \cite[Theorem~5]{Kar20}).
	In $\mathsf{ZF}$, $\mathsf{DC}$ is equivalent to the Baire category theorem for complete metric spaces (see \cite[Theorem~18]{Kar20}).
	In some sense, $\mathsf{DC}$ is the main alternative to $\mathsf{AC}$:
	Starting from a model of $\mathsf{ZFC}$ with an inaccessible cardinal, Robert Solovay \cite{Sol70} constructed a model of $\mathsf{ZF}+\mathsf{DC}$ in which every set
	of reals is Lebesgue measurable (see also \cite[567L]{Fre15}).
	
	\begin{scholium} \label{SCHOLIUM : reinventing the wheel}
		To the erudite reader acquainted with the literature on Ces{\` a}ro spaces,
		Lemmas~\ref{LEMMA : dual to ces_infty is t_1 in ZF+DC},~\ref{LEMMA : tandori_1 = direct sum} and~\ref{PROPOSITION: Schur Tandori l1 ZF}
		may look like reinventing the wheel.
		They are not, because in this section we work {\bf without} the Axiom of Choice, and most of the available
		results concerning Ces{\` a}ro spaces depend\footnote{Sometimes it is a non-trivial task to decide whether
			a given argument uses some form of $\mathsf{AC}$ or not (we experienced this painfully while reading, for example, \cite{LM15a}).}
		on $\mathsf{AC}$ in one way or another.
		Apart from relying on the authors' word,
		we see no better way to show that an argument works without $\mathsf{AC}$ than to write out all the details.
		\demo
	\end{scholium}
	
	\begin{remark}[Measure theory without choice] \label{REMARK : measure theory without choice}
		Along the way, we need some measure theory without $\mathsf{AC}$.
		Chapter~56 \enquote{Choice and determinacy} of Fremlin's monograph \cite{Fre15} seems designed for this very purpose,
		so we stick to it (and refer there for detailed explanations).
		In short, $\mathsf{AC}_{\omega}$ is enough to make the Lebesgue measure behave neatly;
		in fact, \cite[Section~565]{Fre15} shows how to construct the Lebesgue measure without $\mathsf{AC}$.
		According to \cite[Definition~565D]{Fre15}, the Lebesgue measure on $\mathbb{R}^d$
		is then understood as the completion of a unique Borel-coded measure constructed in \S563H. Let us stop here.
		\demo
	\end{remark}
	
	\begin{lemma}[Duality] \label{LEMMA : dual to ces_infty is t_1 in ZF+DC}
		{\it Assume $\mathsf{ZF}$.
			Then $\ces^{\times} \equiv \widetilde{\ell}_1$.}
	\end{lemma}
	
	\begin{remark}
		A much more general approach to the duality of Ces{\` a}ro spaces can be found in Le{\' s}nik and Maligranda's paper \cite{LM15a}
		(which, at least for rearrangement-invariant spaces, agrees with Sinnamon's results \cite{Si94}).
		In particular, a special case of \cite[Theorem~6]{LM15a} shows that the K{\" o}the dual of $\ces$ coincides,
		up to the equivalence of norms, with $\widetilde{\ell}_1$.
		Anyway, the argument below is elementary, choice-free, and gives the equality of norms.
		\demo
	\end{remark}
	
	\begin{proof}[Proof of Lemma~\ref{LEMMA : dual to ces_infty is t_1 in ZF+DC}]
		The argument presented below is almost verbatim the same as Aleksiewicz's proof of Theorem~1 from \cite{Ale57};
		in fact, most of the adjustments being adaptations to our notation.
		We present all the details mainly for the reader's convenience (but see Scholium~\ref{SCHOLIUM : reinventing the wheel}).
		
		Our goal is to show that $(\widetilde{\ell_1})^{\times} \equiv \ces$.
		Knowing this, and using the Fatou property of the space $\widetilde{\ell_1}$ (see, for example, \cite[Theorem~1]{LM15a}), we immediately get
		\begin{equation*}
			\ces^{\times}
			\equiv (\widetilde{\ell}_1)^{\times\times}
			\equiv \widetilde{\ell}_1.
		\end{equation*}
		This yields the desired conclusion.
		
		We will start by showing that for any function $f$ from $(\widetilde{\ell_1})^\times$ such that
		$\supp(f) \subset [1,N]$ for some $N\in\mathbb{N}$ it follows that
		\begin{equation} \tag{$\spadesuit$} \label{EQ : dual do tandori 1 = ces_infty}
			\norm{f}_{(\widetilde{\ell}_1)^\times} = \norm{f}_{\ces}.
		\end{equation}
		Fix $N \in \mathbb{N}$ and take $f = \sum_{n=1}^N f(n) \boldsymbol{e}_n$.
		Of course, we can - and we will - assume that $f$ is positive.
		Thus, we have
		\begin{align} \label{lemma_dual_of_tandoriDC_norm}
			\norm{f}_{(\widetilde{\ell}_1)^\times}
			& = \sup \left\{ \abs{\sum_{n=1}^N f(n)g(n)} : \norm{g}_{\widetilde{\ell}_1} = 1 \right\} \nonumber \\
			& = \sup \left\{ \sum_{n=1}^N \abs{f(n)g(n)} : \norm{g}_{\widetilde{\ell}_1} = 1 \right\} \nonumber \\
			& = \sup \left\{ \sum_{n=1}^N f(n) \widetilde{g}(n) : \norm{g}_{\widetilde{\ell}_1} = 1 \right\}.
		\end{align}
		Set
		\begin{equation*}
			\Omega \coloneqq \left\{ x = \sum_{n=1}^N x(n) \boldsymbol{e}_n
			: x(1) \geqslant x(2) \geqslant \dots \geqslant x(N) \geqslant 0 \text{ and } \sum_{n=1}^N x(n) = 1 \right\}.
		\end{equation*}
		Clearly, $\Omega$ is a closed and convex subset of $\mathbb{R}^N$.
		Hence, remembering about \eqref{lemma_dual_of_tandoriDC_norm},
		\begin{equation*}
			\norm{f}_{(\widetilde{\ell}_1)^\times} = \sup \left\{ \sum_{n=1}^N f(n)g(n) : g \in \Omega \right\}.
		\end{equation*}
		We will show that 
		\begin{equation}\label{lemma_dual_of_tandoriDC_ext}
			\text{Ext}(\Omega) = \left\{ \delta_k : 1 \leqslant k \leqslant N \right\},
		\end{equation}
		where
		\begin{equation*}
			\delta_k \coloneqq \frac{1}{k} \, \sum_{n=1}^k \boldsymbol{e}_n \in \Omega, \quad (k \in \mathbb{N}).
		\end{equation*}
		For $x = \sum_{n=1}^k x(k) \in \Omega$, we have
		\begin{equation*}
			x = \frac{x(k)}{k} \, \delta_k+\frac{x(k-1)-x(k)}{k-1} \, \delta_{k-1}+\dots+\frac{x_{1}-x_2}{1} \, \delta_{1}.
		\end{equation*}
		Observe also that
		\begin{equation*}
			\frac{x(k)}{k}+\frac{x(k-1)-x(k)}{k-1}+\dots+\frac{x(1)-x(2)}{1}=x(k)+x(k-1)+\dots+x(1)=1.
		\end{equation*}
		This, in turn, imply that for any $x\in\Omega$ there is $k = 1, 2, ..., N$ such that
		\begin{equation*}
			x = \sum_{n=1}^k \lambda(n)\delta_n \quad \left( \sum_{n=1}^k\lambda(n)=1 \right).
		\end{equation*}
		Since all $\delta_n$'s are linearly independent this means that the equality \eqref{lemma_dual_of_tandoriDC_ext} holds.
		Moving on, since the norm of a functional acting on a finite-dimensional Banach space is attained on a set of extreme points
		of its unit ball, we infer that
		\begin{align*}
			\norm{f}_{(\widetilde{\ell}_1)^\times}
			& = \sup \left\{ \sum_{n=1}^N f(n)g(n) : g \in \Omega \right\}
			= \sup \left\{ \sum_{n=1}^N f(n)g(n) : g \in \text{Ext}(\Omega) \right\} \\
			& = \max \left\{ \sum_{n=1}^k \frac{f(n)}{k} : 1 \leqslant k \leqslant N \right\}
			= \norm{f}_{\ces}.
		\end{align*}
		In consequence, \eqref{EQ : dual do tandori 1 = ces_infty} follows.
		
		It remains to explain how to improve \eqref{EQ : dual do tandori 1 = ces_infty} to cover the case of any function,
		that is, without any additional restrictions on support.
		To this end, take $f \in (\widetilde{\ell}_1)^\times$ and let $f_n \coloneqq \abs{f}\chi_{[1,n]}$ for $n \in \mathbb{N}$.
		Plainly, $f$ is a pointwise limit of a sequence $\{f_n\}_{n=1}^{\infty}$.
		Thus, it follows from the Fatou property of both spaces $(\widetilde{\ell_1})^\times$ and $\ces$
		in tandem with the previously proven \eqref{EQ : dual do tandori 1 = ces_infty} that
		\begin{equation*}
			\norm{f}_{(\widetilde{\ell}_1)^\times}
			= \sup \left\{ \norm{f_n}_{(\widetilde{\ell}_1)^\times} : n \in \mathbb{N} \right\}
			= \sup \left\{ \norm{f_n}_{\ces} : n \in \mathbb{N} \right\}
			= \norm{f}_{\ces}.
		\end{equation*}
		The proof is hereby completed.
	\end{proof}
	
	\begin{lemma}[Blocking technique] \label{LEMMA : tandori_1 = direct sum}
		{\it Assume $\mathsf{ZF}$.
			The space $\widetilde{\ell}_1$ can be renormed so that $\widetilde{\ell}_1 = (\bigoplus_{n=0}^{\infty} \ell_{\infty}^{\, 2^n})_{\ell_1(w)}$,
			where $w(n) \coloneqq 2^n$ for $n \in \mathbb{N}_0$; more precisely, $\widetilde{\ell}_1$ is $4$-isomorphic to $(\bigoplus_{n=0}^{\infty} \ell_{\infty}^{\, 2^n})_{\ell_1(w)}$.}
	\end{lemma}
	
	\begin{remark} \label{REMARK : History of Block Form representations}
		The above result can be deduced from \cite[Corollary~3.5]{ALM19}, \cite[Theorem~3.2]{GE98}, \cite[Corollary~4.e.1]{KT24} and \cite[Corollary~3.2]{Wal20}.
		Below, we modify our own argument from \cite{KT24}, mainly because it gives the best\footnote{We do not know whether the constant $4$ can be improved.} known constant.
		The transformation of $\widetilde{\ell}_1$-like norms (or, more generally, of norms in \enquote{section form})
		into \enquote{direct sum} norms falls within the scope of the so-called {\bf blocking technique}
		(the encyclopaedic source on this subject is \cite{GE98}; {\it cf.} \cite{KT24}).
		\demo
	\end{remark}
	
	\begin{proof}[Proof of Lemma~\ref{LEMMA : tandori_1 = direct sum}]
		Essentially, we will present here a trimmed version of the proof of Theorem~4.d.1 from \cite{KT24}
		(the original argument maneuvers between some non-trivial facts from interpolation theory the use of which we completely bypass below).
		
		Let us define an (extended) norm $\vertiii{\cdot} \colon \mathbb{K}^{\mathbb{N}} \rightarrow \mathbb{R}$ as
		\begin{equation*}
			\vertiii{f}
			\coloneqq \sum_{n=0}^{\infty} 2^n M(f,n) \quad (f = \{ f(n) \}_{n=1}^{\infty} \in \mathbb{K}^{\mathbb{N}}),
		\end{equation*}
		where
		\begin{equation*}
			M(f,n) \coloneqq \sup \{ \abs{f(k)} : k \in \Delta_n \} \quad (n \in \mathbb{N}_0),
		\end{equation*}
		and $\Delta_n$ denote the $n^{\text{th}}$ dyadic interval of $\mathbb{N}$, that is, $\Delta_n \coloneqq [2^n,\dots,2^{n+1})$ for $n \in \mathbb{N}_0$.
		Of course,
		\begin{equation*}
			\left\{ f\in \mathbb{K}^{\mathbb{N}} : \vertiii{f} < \infty \right\}
			= (\bigoplus_{n=0}^{\infty} \ell_{\infty}^{\, 2^n})_{\ell_1(w)},
		\end{equation*}
		where $w(n) = 2^n$ for $n \in \mathbb{N}_0$.
		That said, all we need to show is that the above functional introduces on the space $\widetilde{\ell}_1$ an equivalent norm.
		
		To see this, take $f = \{ f(n) \}_{n=1}^{\infty}$ from $\widetilde{\ell}_1$.
		To proceed further, the essentially trivial observation that $\vertiii{f} = \norm{ \sum_{n=0}^{\infty} M(f,n) \chi_{\Delta_n}}_{\ell_1}$
		will be vital.
		Keeping this in mind, we have
		\begin{align*}
			\vertiii{f}
			& = \sum_{n=0}^{\infty} 2^n M(f,n)
			= \abs{f(1)} + \sum_{n=1}^{\infty} 2^n M(f,n) \\
			& \leqslant \norm{f}_{\ell_\infty}
			+ \lVert D_{1/2} \colon \ell_1 \rightarrow \ell_1 \rVert
			\cdot \sum_{n=1}^{\infty} 2^{n-1} M(f,n) \\
			& \leqslant  \lVert \text{id} \colon \widetilde{{\ell}_1} \rightarrow \ell_{\infty} \rVert \, \norm{f}_{\widetilde{\ell}_1}
			+ \sum_{n=1}^{\infty} \sup\limits_{k \geqslant n} \abs{f(k)}
			= 2 \norm{f}_{\widetilde{\ell}_1},
		\end{align*}
		where $D_{1/2}$ denote the dilation operator.
		This gives us the first half of our claim.
		
		Now, let us turn to the proof of the second inequality.
		We have
		\begin{align*}
			\norm{f}_{\widetilde{\ell_1(w)}}
			& = \sum_{n=0}^\infty \left(\sup\limits_{k \geqslant n} \abs{f(k)}\right) 2^{n}
			\leqslant \sum_{n=0}^\infty \left(\sum_{k=n}^\infty \abs{f(k)}\right) 2^{n}
			\leqslant \sum_{n=0}^\infty \abs{f(n)} \sum_{k=0}^n 2^{k} \\
			& = \sum_{n=0}^\infty \abs{f(n)}2^{n} \sum_{k=0}^n 2^{-n} 2^{k}
			= \sum_{n=0}^\infty \abs{f(n)} 2^{n} \sum_{k=0}^n 2^{-k+1} \\
			& \leqslant \left( \sum_{k=0}^\infty 2^{-k+1} \right)\left( \sum_{n=0}^\infty \abs{f(n)} 2^{n} \right)
			\leqslant 2 \sum_{n=0}^\infty \abs{f(n)} 2^{n}
			= 2\norm{f}_{\ell_1(w)}.
		\end{align*}
		This, of course, means that
		\begin{equation*}
			\ell_1(w) \overset{2}{\hookrightarrow} \widetilde{\ell_1(w)}.
		\end{equation*}
		Thus,
		\begin{align*}
			\norm{f}_{\widetilde{\ell_1}}
			& \leqslant \norm{ \sum_{n=0}^{\infty} M(f,n) \chi_{\Delta_n} }_{\widetilde{\ell_1}}
			\leqslant \norm{ \sum_{n=0}^{\infty} M(f,n) \boldsymbol{e}_n }_{\widetilde{\ell_1(w)}} \\
			& \leqslant \lVert \text{id} \colon \ell_1(w) \rightarrow \widetilde{\ell_1(w)} \rVert \, \norm{ \sum_{n=0}^{\infty} M(f,n) \boldsymbol{e}_n }_{\ell_1(w)} \\
			& \leqslant  2 \sum_{n=0}^{\infty} 2^n M(f,n)
			= 2 \vertiii{f}.
		\end{align*}
		The above inequality completes the proof.
	\end{proof}
	
	Recall that a Banach space $\mathsf{X}$ has the {\bf Schur property} if every weakly null sequence in $\mathsf{X}$ is norm null.
	We shall be interested in a quantitative variant of this property (implicitly introduced in Johnson and Odell's paper \cite{JO74};
	see \cite[before Theorem~3.3]{GKL96} for the explicit formulation):
	A Banach space $\mathsf{X}$ has the {\bf $c$-strong Schur property} if
	for every $\delta > 0$ and every $\varepsilon > 0$, every normalised and $\delta$-separated\footnote{A bounded sequence $\{x_n\}_{n=1}^{\infty}$
		in a Banach space $\mathsf{X}$ is {\bf $\delta$-separated} if $\norm{x_k - x_{\ell}}_{\mathsf{X}} > \delta$ for all $k \neq \ell$.}
	sequence in $\mathsf{X}$ has a subsequence $\{x_i\}_{i=1}^{\infty}$ such that
	\begin{equation*}
		\norm{\sum_{i=1}^n a_i x_i}_{\mathsf{X}}
		\geqslant \left( \frac{2c}{\delta} + \varepsilon \right)^{-1} \, \sum_{i=1}^n \abs{a_i}
	\end{equation*}
	for all $n \in \mathbb{N}$ and all scalars $a_1, ..., a_n$.
	We say that $\mathsf{X}$ has the {\bf strong Schur property} if it has the $c$-strong Schur property for some $c$.
	
	Every Banach space with the strong Schur property has the Schur property.
	Indeed, by Rosenthal's $\ell_1$-theorem\footnote{Which is a choiceless result (see \cite[Chapter~II, Section~19.E]{Kec95}).} (see, for example, \cite[Theorem~10.2.1]{AK06}),
	the Schur property of $\mathsf{X}$ is equivalent to the following:
	for every $\delta > 0$, every $\delta$-separated sequence in $\text{Ball}(\mathsf{X})$ has a subsequence equivalent
	to the unit vector basis of $\ell_1$.
	The converse fails (see \cite[Remark~4.3]{KO89}; see also \cite[p.~436]{Wnu09} for a more elaborate explanation).
	Note also that if $\mathsf{Y} \approx_d \mathsf{X}$ and $\mathsf{X}$ has the $c$-strong Schur property, then $\mathsf{Y}$ has the $cd^2$-strong Schur property.
	
	The next lemma describes a natural permanence property of unconditional sums of finite-dimensional Banach spaces.
	
	\begin{lemma}[Strong Schur property] \label{PROPOSITION: Schur Tandori l1 ZF}
		{\it Assume $\mathsf{ZF}$ and $\mathsf{DC}$.
			Let $\mathsf{X}$ be a Banach sequence space with the $c$-strong Schur property and the Fatou property,
			and let $\{\mathsf{F}_n\}_{n=1}^{\infty}$ be a sequence of finite-dimensional Banach spaces.
			Then the space $(\bigoplus_{n=1}^{\infty} \mathsf{F}_n)_{\mathsf{X}}$ has the $c$-strong Schur property.}
		
		{\it In particular, the space $\widetilde{\ell}_1$ has the $16$-strong Schur property.}
	\end{lemma}
	
	\begin{remark} \label{REMARK : Schur}
		It is routine to show that an $\ell_1$-sum of finite-dimensional Banach spaces has the Schur property
		(see \cite[p.~19]{Bou81}, \cite[pp.~60--61]{CI90} and \cite[p.~439]{Wnu09}).
		In $\mathsf{ZFC}$ one can even show that $(\bigoplus_{n=1}^{\infty} \mathsf{X}_n)_{\ell_1}$ has the Schur property if,
		and only if, each $\mathsf{X}_n$ has the Schur property (see, for example, \cite{Tan98}; we do not know whether this can be done without $\mathsf{AC}$).
		
		It was proved in \cite[Theorem~3.1]{ALM19} that $\widetilde{\ell}_1$ has the Schur property.
		The \enquote{in particular} part of Lemma~\ref{PROPOSITION: Schur Tandori l1 ZF} strengthens this result,
		also by removing additional set-theoretic assumptions
		(the argument in \cite{ALM19} makes significant use of the Banach--Alaoglu theorem which,
		although weaker than the Axiom of Choice, is equivalent in $\mathsf{ZF}$ to the ultrafilter lemma).
		\demo
	\end{remark}
	
	\begin{proof}[Proof of Lemma~\ref{PROPOSITION: Schur Tandori l1 ZF}]
		The argument is a variation on the well-known \enquote{gliding hump} theme;
		our reading of Kalenda's paper \cite{Kal25} influenced some parts of the proof.
		
		For $A \subset \mathbb{N}$, let $P_A {\bf f} \coloneqq {\bf f}\chi_A$ for ${\bf f} \in (\bigoplus_{n=1}^{\infty} \mathsf{F}_n)_{\mathsf{X}}$,
		and write $\norm{\cdot}$ for the norm of $(\bigoplus_{n=1}^{\infty} \mathsf{F}_n)_{\mathsf{X}}$.
		
		{\bf First step: Diagonal argument.}
		Fix $\delta \in (0,2]$ and $\varepsilon > 0$, and
		let\footnote{We use bold font to emphasise that the ${\bf f}_i$'s are vector-valued.} $\{ {\bf f}_i\}_{i=1}^{\infty}$
		be a normalised $\delta$-separated sequence in $(\bigoplus_{n=1}^{\infty} \mathsf{F}_n)_{\mathsf{X}}$.
		By Cantor's diagonal argument\footnote{Which does not require $\mathsf{AC}$ here, because all $\mathsf{F}_n$'s are finite-dimensional.}, a subsequence
		of $\{ {\bf f}_i\}_{i=1}^{\infty}$ converges pointwise
		to some ${\bf f}_{\bullet} \in \prod_{n=1}^{\infty} \mathsf{F}_n$.
		Since $\mathsf{X}$ has the Fatou property, ${\bf f}_{\bullet} \in \mathrm{Ball}((\bigoplus_{n=1}^{\infty} \mathsf{F}_n)_{\mathsf{X}})$.
		We still denote this subsequence by $\{ {\bf f}_i\}_{i=1}^{\infty}$.
		
		{\bf Second step: Gliding hump.}
		We start with the following simple claim.
		
		\begin{claimletters} \label{CLAIM : separable}
			{\it Let $\mathsf{Z}$ be a Banach sequence space with the Schur property and the Fatou property.
				Then $\mathsf{Z}$ is order continuous.}
		\end{claimletters}
		\begin{proof}
			Suppose not.
			Then there are $\delta > 0$ and a positive $z \in \mathsf{Z}$ such that
			$\lVert z\chi_{[n,\infty)} \rVert_{\mathsf{Z}} > 2\delta$ for all $n \in \mathbb{N}$.
			By the Fatou property, $\lVert z\chi_{[n,m]} \rVert_{\mathsf{Z}} \rightarrow \lVert z\chi_{[n,\infty)} \rVert_{\mathsf{Z}}$ as $m \rightarrow \infty$,
			so we can construct (using $\mathsf{DC}$) an increasing sequence $\{ n_i \}_{i=1}^{\infty}$ with
			$\norm{z_i}_{\mathsf{Z}} \geqslant \delta$, where $z_i \coloneqq z\chi_{[n_i+1, n_{i+1}]}$.
			The sequence $\{ z_i \}_{i=1}^{\infty}$ is weakly null:
			For any positive functional $\varphi \in \mathsf{Z}^{*}$ we have $\sum_{i=1}^{\infty} \langle z_i, \varphi \rangle \leqslant \langle z, \varphi \rangle$,
			and every functional in $\mathsf{Z}^{*}$ is a difference of two positive ones.
			This contradicts the Schur property.
		\end{proof}
		
		Clearly, $({\bf f}_i - {\bf f}_{\bullet})(n) \rightarrow 0$ as $i \rightarrow \infty$ for every $n \in \mathbb{N}$.
		By Claim~\ref{CLAIM : separable}, the elements ${\bf f}_{\bullet}$ and ${\bf f}_i - {\bf f}_{\bullet}$ have order continuous norms.
		Hence, for any $\eta > 0$, passing to a further subsequence (still denoted by $\{ {\bf f}_i\}_{i=1}^{\infty}$),
		we find intervals $I_0 < I_1 < I_2 < \dots$ of $\mathbb{N}$ such that
		\begin{equation} \label{EQ : f bullet < eta}
			\norm{{\bf f}_{\bullet} - P_{I_0}{\bf f}_{\bullet}} < \eta
			\quad\text{and}\quad
			\norm{({\bf f}_i - {\bf f}_{\bullet}) - P_{I_i}({\bf f}_i - {\bf f}_{\bullet})} < \eta \quad (i \in \mathbb{N}).
		\end{equation}
		(For the second estimate, use pointwise convergence on the finite set $\{1,\dots,\max I_{i-1}\}$ and order continuity for the tail.)
		
		{\bf Third step: Projecting onto the base space.}
		Put
		\begin{equation*}
			{\bf c} \coloneqq P_{I_0}{\bf f}_{\bullet},\ {\bf g}_i \coloneqq P_{I_i}({\bf f}_i - {\bf f}_{\bullet})\text{ and }{\bf r}_i \coloneqq {\bf f}_i - {\bf c} - {\bf g}_i.
		\end{equation*}
		Note that $\norm{{\bf r}_i} < 2\eta$ by \eqref{EQ : f bullet < eta}. Let
		\begin{equation*}
			h_i \coloneqq \lfloor {\bf c} \rceil + \lfloor {\bf g}_i \rceil \in \mathsf{X},
			\quad\text{where}\quad
			\lfloor {\bf f} \rceil \coloneqq \sum_{n=1}^{\infty} \norm{{\bf f}(n)}_{\mathsf{F}_n} \boldsymbol{e}_n.
		\end{equation*}
		Since ${\bf c}$ and all ${\bf g}_i$'s are mutually disjoint, $\norm{h_i}_{\mathsf{X}} = \norm{{\bf c} + {\bf g}_i}$, and hence
		\begin{equation} \label{EQ : norm h_i}
			1 - 2\eta \leqslant \norm{h_i}_{\mathsf{X}} \leqslant 1 + 2\eta \quad (i \in \mathbb{N}).
		\end{equation}
		Moreover, for $i \neq j$,
		\begin{equation*}
			\norm{h_i - h_j}_{\mathsf{X}}
			= \norm{{\bf g}_i - {\bf g}_j}
			= \norm{({\bf f}_i - {\bf f}_j) - {\bf r}_i + {\bf r}_j}
			> \delta - 4\eta.
		\end{equation*}
		Finally, for all scalars $a_1,\dots,a_N$, disjointness gives
		\begin{equation} \label{EQ : h vs f}
			\norm{\sum_{i=1}^{N} a_i h_i}_{\mathsf{X}} = \norm{\sum_{i=1}^{N} a_i ({\bf c} + {\bf g}_i)}.
		\end{equation}
		
		{\bf Fourth step: The price of normalisation.}
		Passing to a subsequence, we may assume that
		\begin{equation*}
			\norm{h_i}_{\mathsf{X}} \rightarrow \gamma \in [1-2\eta,1+2\eta]\text{ and }\abs{\norm{h_i}_{\mathsf{X}} - \gamma} < \eta \qquad (i \in \mathbb{N}).
		\end{equation*}
		We have
		\begin{equation*}
			\norm{\frac{h_i}{\norm{h_i}_{\mathsf{X}}} - \frac{h_j}{\norm{h_j}_{\mathsf{X}}}}_{\mathsf{X}}
			\geqslant \frac{\norm{h_i - h_j}_{\mathsf{X}}}{\gamma} - \sum_{k \in \{i,j\}} \frac{\abs{\gamma - \norm{h_k}_{\mathsf{X}}}}{\gamma}
			> \frac{\delta - 6\eta}{1+2\eta} \eqqcolon S(\delta,\eta) \quad (i \neq j).
		\end{equation*}
		
		{\bf Fifth step: Punchline.}
		By the $c$-strong Schur property of $\mathsf{X}$, a further subsequence satisfies
		\begin{equation*}
			\norm{\sum_{i=1}^{N} b_i \frac{h_i}{\norm{h_i}_{\mathsf{X}}}}_{\mathsf{X}} \geqslant \left( \frac{2c}{S(\delta,\eta)} + \frac{\varepsilon}{2} \right)^{-1} \sum_{i=1}^{N} \abs{b_i}
		\end{equation*}
		for all scalars $\{b_n\}_{n=1}^{\infty}$.
		Using the above inequality for $b_i = a_i\norm{h_i}_{\mathsf{X}}$ along with \eqref{EQ : norm h_i} and \eqref{EQ : h vs f}, we get
		\begin{equation*}
			\norm{\sum_{i=1}^{N} a_i {\bf f}_i}
			\geqslant \norm{\sum_{i=1}^{N} a_i ({\bf c} + {\bf g}_i)} - 2\eta \sum_{i=1}^{N} \abs{a_i}
			\geqslant \left( \frac{1-2\eta}{\frac{2c}{S(\delta,\eta)}+\frac{\varepsilon}{2}} - 2\eta \right) \sum_{i=1}^{N} \abs{a_i}.
		\end{equation*}
		As $\eta \rightarrow 0$, the constant in brackets tends to $( \frac{2c}{\delta}+\frac{\varepsilon}{2} )^{-1} > ( \frac{2c}{\delta}+\varepsilon )^{-1}$,
		so choosing $\eta$ small enough, we obtain the required estimate.
		
		For the \enquote{in particular} part, note that $\ell_1(w) \cong \ell_1$ isometrically (via the multiplication map $x \mapsto xw$), so Lemma~\ref{LEMMA : tandori_1 = direct sum} shows that
		$\widetilde{\ell}_1$ is $4$-isomorphic to $(\bigoplus_{n=0}^{\infty} \ell_{\infty}^{\, 2^n})_{\ell_1}$.
		Since $\ell_1$ has the $1$-strong Schur property (see \cite[Example~5.9]{Kal25}), the first part shows that $(\bigoplus_{n=0}^{\infty} \ell_{\infty}^{\, 2^n})_{\ell_1}$ has the $1$-strong Schur property,
		and hence $\widetilde{\ell}_1$ has the $16$-strong Schur property.
	\end{proof}
	
	With the triptych of the above lemmas in hand, we can finally prove the main result of this section.
	In what follows, $\mathsf{PM}_{\omega}$ is the statement \enquote{there is no finitely additive probability measure
		on the power set of the natural numbers which vanishes on singletons}
	(we refer to V{\" a}th's paper \cite{Vat98} for more background and references; see also \cite[Theorem~36]{Kar20}).
	It holds, for example, in Solovay's model \cite{Sol70}.
	
	\begin{theorem} \label{THM : In ZF ces_infty and Ces_infty are NOT isomorphic}
		{\it Assume $\mathsf{ZF}$, $\mathsf{DC}$ and $\mathsf{PM}_{\omega}$.
			Then the spaces $\ces$ and $\Ces$ are not isomorphic.}
	\end{theorem}
	\begin{proof}
		Suppose that $\ces$ and $\Ces$ are isomorphic.
		Then their duals $\ces^{*}$ and $\Ces^{*}$ are isomorphic as well.
		By V{\" a}th's result \cite[Theorem~3.1]{Vat98}, under $\mathsf{ZF}$, $\mathsf{DC}$ and $\mathsf{PM}_{\omega}$
		there are no singular functionals on Banach function spaces\footnote{
		Recall that the dual $\mathsf{L}^*$ of a Banach lattice $\mathsf{L}$ decomposes into a direct sum $\mathsf{L}^{\times} \oplus \mathsf{L}^s$,
		where $\mathsf{L}^s \coloneqq \{ S \in \mathsf{L}^* \colon \langle x, S \rangle = 0 \text{ for all } x \in \mathsf{L}_o \}$ is the space of singular functionals and $\mathsf{L}_o$ is the order continuous part of $\mathsf{L}$.}.
		Hence $\ces^{*} = \ces^{\times}$ and $\Ces^{*} = \Ces^{\times}$, and so $\ces^{\times}$ and $\Ces^{\times}$ are isomorphic.
		We show that this is impossible by means of the Schur property.
		
		That $\ces^{\times}$ has the Schur property follows from Lemmas~\ref{LEMMA : dual to ces_infty is t_1 in ZF+DC}, \ref{LEMMA : tandori_1 = direct sum} and \ref{PROPOSITION: Schur Tandori l1 ZF}.
		To see that $\Ces^{\times}$ fails the Schur property, consider $f_n \coloneqq \chi_{[1-1/n,1]}$ for $n \in \mathbb{N}$.
		We claim that:
		\begin{enumerate}[label={\normalfont{(\Alph*)}}]
			\item\label{EQ : brak Schura 1} all $f_n$'s are normalised in $\Ces^{\times}$;
			\item\label{EQ : brak Schura 2} the sequence $\{f_n\}_{n=1}^{\infty}$ is weakly null in $\Ces^{\times}$.
		\end{enumerate}
	
		Indeed, $\norm{f_n}_{\Ces^{\times}} \leqslant \lVert \chi_{[0,1]} \rVert_{\Ces^{\times}} = 1$.
		Moreover, $g_n \coloneqq n \chi_{[1-1/n,1]}$ belongs to $\text{Ball}(\Ces)$, and $\int_0^1 f_n g_n\, d\lambda = 1$, so $\norm{f_n}_{\Ces^{\times}} = 1$ and \ref{EQ : brak Schura 1} follows.
		
		Next, let $\varphi \in (\Ces^{\times})^{*}$.
		Applying \cite[Theorem~3.1]{Vat98} to the Banach function space $\Ces^{\times}$ and using \cite{Vat97}, we get
		$(\Ces^{\times})^{*} = \Ces^{\times \times} = \Ces$,
		so there is $g \in \Ces$ with $\langle f, \varphi \rangle = \int f g\, d\lambda$ for $f \in \Ces^{\times}$.
		Therefore,
		\begin{equation} \label{EQ : weakly null part 1}
			\abs{\langle f_n, \varphi \rangle}
			\leqslant \int_{1-1/n}^1 \abs{g(t)} dt.
		\end{equation}
		Since $\left\{ h \in \Ces : \supp h \subset [0,1] \right\} \hookrightarrow L_1[0,1]$ (see, for example, \cite[Theorem~1(d)]{AM09}), we have $g\chi_{[0,1]} \in L_1[0,1]$,
		and the Lebesgue dominated convergence theorem (see \cite[Theorem~564F]{Fre15} and Remark~\ref{REMARK : measure theory without choice})
		shows that the right-hand side of \eqref{EQ : weakly null part 1} tends to zero.
		This gives \ref{EQ : brak Schura 2} and completes the proof.
	\end{proof}
	
	Combining Theorems~\ref{COR: Ces and ces isomorphic} and~\ref{THM : In ZF ces_infty and Ces_infty are NOT isomorphic} together with Shelah's result \cite{She84}, we obtain the following.
	
	\begin{corollary} \label{COR : independent of ZF}
		{\it Assume that $\mathsf{ZF}$ is consistent.
			Then the statement \enquote{$\ces$ and $\Ces$ are isomorphic} is independent of $\mathsf{ZF}+\mathsf{DC}$.}
	\end{corollary}
	
	\section{{\bf Loose ends}} \label{SECTION : Open questions}
	\addtocontents{toc}{\protect\setcounter{tocdepth}{0}}
	We list and briefly discuss some problems arising from this work.
	
	\subsection{On the Schur property} \label{SUBSECTION : On the Schur property}
	
	By Lemma~\ref{PROPOSITION: Schur Tandori l1 ZF}, the space $\widetilde{\ell}_1$ has the $16$-strong Schur property.
	
	\begin{question}
		{\it What is the best constant $c$ for which $\widetilde{\ell}_1$ has the $c$-strong Schur property? Is it equal to $1$?}
	\end{question}
	
	More generally, for a Banach sequence space $\mathsf{X}$, the {\bf Tandori sequence space} $\widetilde{\mathsf{X}}$ consists
	of those sequences $x = \{x_n\}_{n=1}^{\infty} \in \mathbb{K}^{\mathbb{N}}$ whose decreasing majorant $\widetilde{x} = \{ \sup_{k \geqslant n} \abs{x_k} \}_{n=1}^{\infty}$
	belongs to $\mathsf{X}$, with the norm $\norm{x}_{\widetilde{\mathsf{X}}} \coloneqq \lVert \widetilde{x} \rVert_{\mathsf{X}}$
	(see, for example, \cite{ALM19}, \cite{KT24} and \cite{LM15a}; {\it cf.} Section~\ref{SUBSECTION : aide}).
	The following simple observation slightly extends \cite[Theorem~3.1]{ALM19}.
	
	\begin{proposition}[Schur property of Tandori sequence spaces] \label{PROP : X Schur => Tandori X Schur}
		{\it Let $\mathsf{X}$ be a Banach sequence space such that $\norm{x}_{\mathsf{X}} \approx \norm{x}_{\ell_1}$ for all non-increasing non-negative sequences $x$.
			Then $\widetilde{\mathsf{X}} = \widetilde{\ell}_1$, and hence $\widetilde{\mathsf{X}}$ has the strong Schur property.}
	\end{proposition}
	\begin{proof}
		Since $\widetilde{x}$ is non-increasing and non-negative, $$\norm{x}_{\widetilde{\mathsf{X}}} = \norm{\widetilde{x}}_{\mathsf{X}} \approx \norm{\widetilde{x}}_{\ell_1} = \norm{x}_{\widetilde{\ell}_1}.$$
		The second part follows from Lemma~\ref{PROPOSITION: Schur Tandori l1 ZF}.
	\end{proof}
	
	The assumption of Proposition~\ref{PROP : X Schur => Tandori X Schur} is rather restrictive, and the following examples show that the general situation is unclear.
	
	\begin{example}[Nakano, 1951] \label{EXAMPLE : Nakano, 1951}
		We borrow this example from Nakano's paper \cite[p.~512]{Nak51} ({\it cf.} \cite{KM02}).
		Let $\ell_{\{p_n\}}$ be the {\bf Nakano sequence space} with
		\begin{equation*}
			p_n = 1 + \frac{1}{\log\log(n+4)} \qquad (n\in\mathbb{N}). 
		\end{equation*}
		This is a Banach sequence space with a semi-normalised unit vector basis $\{\boldsymbol{e}_n\}_{n=1}^{\infty}$.
		Since, for all $0 < c < 1$,
		\begin{equation*}
			\sum_{n=1}^{\infty} c^{\, p_n/(p_n-1)} = \infty,
		\end{equation*}
		the space $\ell_{\{p_n\}}$ is not isomorphic to $\ell_1$ (see \cite[Theorem~1]{Nak51}),
		and since $\lim_{n \rightarrow \infty} p_n = 1$, it has the Schur property (see \cite[Theorem~2]{Nak51}).
		Note that the assumption of Proposition~\ref{PROP : X Schur => Tandori X Schur} fails here: the sequence $\{1/n\}_{n=1}^{\infty}$ belongs to $\ell_{\{p_n\}}$ but not to $\ell_1$.
		\demo
	\end{example}
	
	\begin{example}[Kami{\' n}ska and Masty{\l}o, 2002] \label{EXAMPLE : Kaminska & Mastylo, 2002}
		Let $\ell_{M}(w)$ be the modular sequence space\footnote{In some circles these spaces are called {\bf Musielak--Orlicz sequence spaces}
			or, more precisely, {\bf weighted Orlicz sequence spaces}
			(see \cite[Section~4.d]{LT79}).}
		generated by the modular
		$\mathfrak{m}(x) \coloneqq \sum_{n=1}^{\infty} M(\abs{x_n}) w_n$,
		where $M(t) = t\log(1+t)$ for $t \geqslant 0$
		and $w = \{w_n\}_{n=1}^{\infty} \in \ell_1$ is a positive weight.
		By \cite[Theorem~5.1]{KM02}, $\ell_{M}(w)$ has the Schur property.
		However, the unit vector basis is {\bf not} semi-normalised in $\ell_{M}(w)$
		(one can easily choose $w \in \ell_1$ so that $\norm{\boldsymbol{e}_n}_{\ell_{M}(w)} \rightarrow 0$ as $n \rightarrow \infty$).
		\demo
	\end{example}
	
	\begin{question}
		{\it Let $\mathsf{X}$ be a Banach sequence space with the Schur property.
			Does $\widetilde{\mathsf{X}}$ have the Schur property?
			In particular, do the spaces $\widetilde{\ell_{\{p_n\}}}$ and $\widetilde{\ell_{M}(w)}$ from Examples~\ref{EXAMPLE : Nakano, 1951} and \ref{EXAMPLE : Kaminska & Mastylo, 2002} have the Schur property?}
	\end{question}
	
	\subsection{Operators between $\mathcal{L}_p$-spaces beyond Theorem~\ref{THEOREM A}}
	
	Theorem~\ref{THEOREM A} leaves two natural problems open.
	
	\begin{question} \label{QUESTION : p > q}
		{\it Let $1 \leqslant q < p < \infty$, and let $\mathsf{X}$ and $\mathsf{Y}$ be separable $\mathcal{L}_p$- and $\mathcal{L}_q$-spaces, respectively.
			Is $\mathfrak{L}(\mathsf{X},\mathsf{Y})$ isomorphic to either $\mathfrak{L}(L_p[0,1],L_q[0,1])$  or  $\mathfrak{L}(\ell_p,\ell_q)$?}
	\end{question}
	
	\begin{question} \label{QUESTION : general targets}
		{\it Let $\mathsf{X}$ be a separable $\mathcal{L}_{\infty}$-space and let $\mathsf{Y}$ be an $\mathcal{L}_{\infty}$-space which is complemented in its bidual
			{\normalfont (}for example, an injective Banach space{\normalfont )}.
			Is $\mathfrak{L}(\mathsf{X},\mathsf{Y})$ isomorphic to $\mathfrak{L}(c_0,\mathsf{Y})$?}
	\end{question}
	
	\subsection{A million miles away} \label{SUBSECTION : Banach--Mazur distance}
	
	In view of Theorem~\ref{COR: Ces and ces isomorphic} and Remark~\ref{REMARK : d_BM(k,K)<=},
	and of the growing interest in computing Banach--Mazur distances between various $\mathcal{C}(K)$-spaces
	(see \cite{KL26} and \cite{KP25}; {\it cf.} \cite{Thi16}), it is natural to ask the following.
	
	\begin{question}
		{\it What is the exact value of the Banach--Mazur distance between $\ces$ and $\Ces$?}
	\end{question}
	
	Let us emphasise once again that neither $\ces$ nor $\Ces$ is isomorphic to any $\mathcal{C}(K)$ (see Claim~\ref{CLAIM A}),
	which perhaps makes the situation even more interesting.
	We are fully aware that our estimate from Remark~\ref{REMARK : d_BM(k,K)<=} is far from being optimal and that we provide no lower bound,
	which is usually much more difficult to obtain.
	
	Bearing in mind Theorem~\ref{THEOREM : ces and Ces NOT isometric} and paraphrasing Korpalski and Plebanek's question from \cite{KP25},
	we ask one more question of a slightly different nature.
	
	\begin{question}
		{\it For which constants $c \geqslant 1$ is the statement $d_{\mathrm{BM}}(\ces, \Ces)=c$ consistent with $\mathsf{ZF}$?}
	\end{question}
	
	Finally, it follows from Leung's paper \cite{Leu93} that the weak $L_p$-spaces $\ell_{p,\infty}$ and $L_{p,\infty}$,
	where $1 < p < \infty$, are isomorphic (since $\ell_{p,p} \equiv \ell_p$ and $L_{p,p} \equiv L_p$, this is yet another relative of Pe{\l}czy{\' n}ski's result \cite{Pel58}).
	
	\begin{question}
		{\it What is the exact value of $d_{\mathrm{BM}}(\ell_{p,\infty},L_{p,\infty})$?}
	\end{question}

	A very rough estimate, based entirely on \cite{Leu93}, yields $d_{\mathrm{BM}}(\ell_{p,\infty},L_{p,\infty}) \leqslant 10^5$. 
	
	\subsection{The number of preduals of $\ces$}
	
	Corollary~\ref{COROLLARY : c many preduals} shows that $\ces$ has at least $\mathfrak{c}$ mutually non-isomorphic preduals.
	However, since Bourgain and Delbaen's construction \cite{BD80} relies on the Hahn--Banach theorem, the following question seems intriguing.
	
	\begin{question}
		{\it Assume $\mathsf{ZF}$. Does $\ces$ have $\mathfrak{c}$ mutually non-isomorphic preduals?}
	\end{question}
	
	\subsection{On isometric preduals of $\Ces$}
	
	Due to Theorem~\ref{THEOREM : ces_infty has UNIQUE predual}, $\ces$ has the unique isometric predual, so we can ask the following.
	
	\begin{question} \label{QUESTION : isometric preduals of Ces_infty}
		{\it Does $\Ces$ have an isometric predual? If so, is it unique?
			If not, how many mutually non-isomorphic isometric preduals does $\Ces$ have?}
	\end{question}
	
	\subsection{On isometric structure of $\Ces$}
	
	Since we only know that the spaces $\Ces[0,1]$ and $\Ces[0,\infty)$ are isomorphic (Claim~\ref{CLAIM C}), and {\bf not} that they are
	isometric, the above question should be understood as two separate questions --- one about $\Ces[0,1]$ and one about $\Ces[0,\infty)$.
	
	\begin{question} \label{QUESTION : Are K[0,1] and K[0,infty) ISOMETRIC?}
		{\it Are $\Ces[0,1]$ and $\Ces[0,\infty)$ isometric?}
	\end{question}
	
	Despite several attempts, we have been unable to answer\footnote{We believe that the answer should be \enquote{no}.} this question.
	
	The proof that $\ces$ and $\Ces$ are not isometrically isomorphic was based on the lack of Fr{\' e}chet differentiability of the $\Ces$-norm.
	Recall that the norm $\norm{\cdot}_{\mathsf{X}}$ of Banach space $\mathsf{X}$ is {\bf G{\^ a}teaux differentiable} at a point $x_0 \in \mathsf{X} \setminus \{0\}$ if the limit 
	\begin{equation*}
		\lim_{t \rightarrow 0} \, \frac{\norm{x_0 + th}_{\mathsf{X}} - \norm{x_0}_{\mathsf{X}}}{t}
	\end{equation*}
	exists for every $h \in \mathsf{X}$ and, moreover, is linear in $h$
	(see, for example, \cite[Definition~5.4.15]{Meg98}).
	\begin{question} 
		{\it Is the norm $\norm{\cdot}_{\Ces}$ of $\Ces$ G{\^ a}teaux differentiable at any points?  If so, how can the set of such points be described?}
	\end{question}
	
	\addtocontents{toc}{\protect\setcounter{tocdepth}{2}}

\end{document}